\documentclass[reqno,11pt]{article}
\usepackage{amsfonts,amsmath,amssymb,amsthm,mathrsfs}%,yfonts}
\usepackage{graphicx,epsfig}
\usepackage{bbm}
\usepackage{t1enc}
\usepackage{subfig}
\usepackage{hyperref}
\usepackage{soul}
\usepackage{float}

\usepackage{a4wide}
\theoremstyle{plain}

\numberwithin{equation}{section}
\newtheorem{theorem}{Theorem}[section]
\newtheorem{corollary}[theorem]{Corollary}%[section]
\newtheorem{lemma}[theorem]{Lemma}%[section]
\newtheorem{example}[theorem]{Example}%[section]
\newtheorem{remark}[theorem]{Remark}%[section]

\newenvironment{proofad1}{\removelastskip\par\medskip
\noindent{\textbf {Proof of Theorem \ref{main1}}.}
\rm}{\penalty-20\null\hfill$\square$\par\medbreak} %%

\newenvironment{proofad2}{\removelastskip\par\medskip
\noindent{\textbf {Proof of Theorem \ref{mainbis}}.}
\rm}{\penalty-20\null\hfill$\square$\par\medbreak} %%

\newenvironment{proofad3}{\removelastskip\par\medskip
\noindent{\textbf {Proof of Theorem \ref{main3}}.}
\rm}{\penalty-20\null\hfill$\square$\par\medbreak}

\newenvironment{proofad4}{\removelastskip\par\medskip
\noindent{\textbf {Proof of Theorem \ref{main4}}.}
\rm}{\penalty-20\null\hfill$\square$\par\medbreak}

\usepackage{color}
\definecolor{darkred}{rgb}{0.8,0,0}
\definecolor{darkblue}{rgb}{0,0,0.7}
\definecolor{darkgreen}{rgb}{0,0.4,0}

\newcommand{\EEE}{\color{black}} % back to normal text

\newcommand{\eps}{\varepsilon}

\newcommand{\R}{{\mathbb R}}

\newcommand{\un}{{\rm 1\kern -2.5pt l}}

\def\argmin{\mathop{{\rm argmin}}\nolimits}

\newcommand{\BBB}{\color{black}}
\newcommand{\MMM}{\color{magenta}}
\title{Euler's optimal profile problem}
\author{Francesco Maddalena, Edoardo Mainini,  Danilo Percivale}

\date{}
\begin{document}

\newcommand{\Addresses}{{
  \bigskip
  \footnotesize
 Edoardo Mainini (Corresponding author), \textsc{Universit\`a degli Studi di Genova,
Dipartimento di     Ingegneria meccanica, energe\-tica, gestionale
e dei trasporti (DIME), Via all'Opera Pia 15, I-16145 Genova, Italy}
\par\nopagebreak
  \textit{E-mail address}: \texttt{mainini@dime.unige.it}\\

 \medskip
  Danilo Percivale, \textsc{Universit\`a degli Studi di Genova, Dipartimento di   Ingegneria meccanica, energe\-tica, gestionale
e dei trasporti (DIME), Via all'Opera Pia 15, I-16145 Genova, Italy}
\par\nopagebreak
  \textit{E-mail address}: \texttt{percivale@diptem.unige.it}\\
	
	\medskip
		
	Francesco Maddalena, \textsc{Politecnico di Bari, Dipartimento di Meccanica,
	Matematica, Management, via Re David 200, 70125 Bari, Italy}
	\par\nopagebreak
 \textit{E-mail address}: \texttt{francesco.maddalena@poliba.it}
}}

%\thanks{}

%\begin{document}

%\bigskip

\maketitle

\begin{abstract}
We study an old variational problem formulated by  Euler as Proposition 53 of his {\it Scientia Navalis}
by means of the direct method of the calculus of variations. Precisely, through  relaxation arguments, 
we prove the existence of minimizers. %  {\blue for a certain range of geometrical parameters}.
  We fully investigate the analytical structure of the minimizers in dependence of the geometric parameters and we identify the ranges of uniqueness and non-uniqueness.
%The physical implications of the results are analyzed in the context of a large class of real phenomena.
\end{abstract}

{\small

\medskip

}

%
%Inter omnes curvas AM cum axe AP et applicata PM eandem aream comprehendentes invenire eam AM, quae circa axem AP utrinque disposita formet %figuram AMN in aqua minimam maximamque patientem resistentiam, si quidem in directione diametri PA progrediatur.
%

\section{Introduction}
L. Euler in his treatise {\it Scientia Navalis} (1749), which  is considered to be one of the cornerstones of the eighteenth century naval architecture, 
 at Proposition 53, formulated the following optimal profile problem (see \cite{E}, \cite{L}).\\

{\it Among all curves AM which with the axis AP and perpendicular PM comprehend the same area, to find that one which with its symmetric branch on the opposite side of the axis AP will form the figure offering the least resistance in water when it moves in the direction PA along the axis} (Fig.1).\\  
%\mathcal{C}^+

The  problem can be viewed as a variant of the celebrated Newton's aerodynamic problem (Proposition 34 of Book 2 of the Principia, 1687, \cite{N}) which relies in optimizing the shape of a solid of revolution,
moving in a fluid along its axis, experiencing the least resistance, at parity of length and caliber. Actually, at Proposition 65 of the same 
treatise, Euler studies in different terms a very similar problem.
Newton's problem of 
minimal resistance was the first  {\it solved} problem in  the  calculus of variations (by Newton himself a decade before
 the {\it brachistochrone problem}, see \cite{g}) and assumes a fluid like medium made by particles of equal mass moving at a constant velocity
with a fixed direction, while the dynamic  interaction between solid and fluid is only due to the perfectly elastic collisions between the fluid particles and the surface of the solid body. 
Though Newton's constitutive assumptions ruling the fluid-solid interaction seems  too crude to copy the complex physical phenomena occurring at the interface 
(strongly influenced by the properties of the fluid and the dynamic features of the motion, \cite{MP}), certainly they capture the essential basic ingredients of the problem. 
Let us recall that  the {\it drag problem} is one of the oldest problems in 
fluid mechanics and at present it
 still seems to be out of reach of analytical results,
for realistic Reynolds numbers.
On the other hand, from a mathematical perspective, the variational integral representing the resistance functional is 
neither coercive nor convex, hence a natural route to prove existence of a minimum via the direct method relies in imposing additional constraints on the admissible shapes. These arguments explain the reasons the oldest problem of the calculus of variations still  provides continuous inspirations for new and challenging problems: 
we refer, for instance, to \cite{BeK,BW, Bo, B1,BK, BFK1,BFK2,BG,CL1,CL2,HKV,LO,LP1,LP2,Leg,M1,M2,MMOP1,MMOP2,p2,p1,W}.
%, \cite{BG}, \cite{BFK1}, \cite{BFK2}, \cite{CL1}, \cite{CL2}, \cite{HKV}, \cite{LO}, \cite{LP1}, \cite{LP2}, \cite{M1}, \cite{MMOP1}, \cite{MMOP2} \cite{p2}, \cite{p1}.

Unlike Newton's problem, the Euler optimal profile problem,  
%except for a short note by G. H. Light (\cite{L}), published exactly a century ago, 
as far as the authors know, 
 has never been studied in the framework of modern calculus of variations, with the only exception of the paper \cite{BW} which deals with a constrained Newton's problem in a special class of admissible functions.
 
In analytical terms the problem  admits the following formulation.
Given  $a>0$, $h>0$, $L\in(0,ah)$, find a curve $\gamma:[0,1]\rightarrow \R^2$, $\gamma=(\gamma_1, \gamma_2)$,  such that 
$\gamma(0)=(0,0),\, \gamma(1)=(a,h)$, and such that (with the notation $z_+:=z\vee 0$)

\begin{equation}
\label{minresist}
\mathcal F(\gamma)=\int_0^1\frac{(\gamma_2')_+^3}{(\gamma_1')^2+(\gamma_2')^2}\,dt\rightarrow {\min},
\end{equation}
 subject to the area constraint

\begin{equation}
\label{constr}
\int_0^1\gamma_1(t)\gamma_2'(t)\,dt=ah-L.
\end{equation}

\begin{figure}[h!] 
\centering
\includegraphics[scale=0.8]{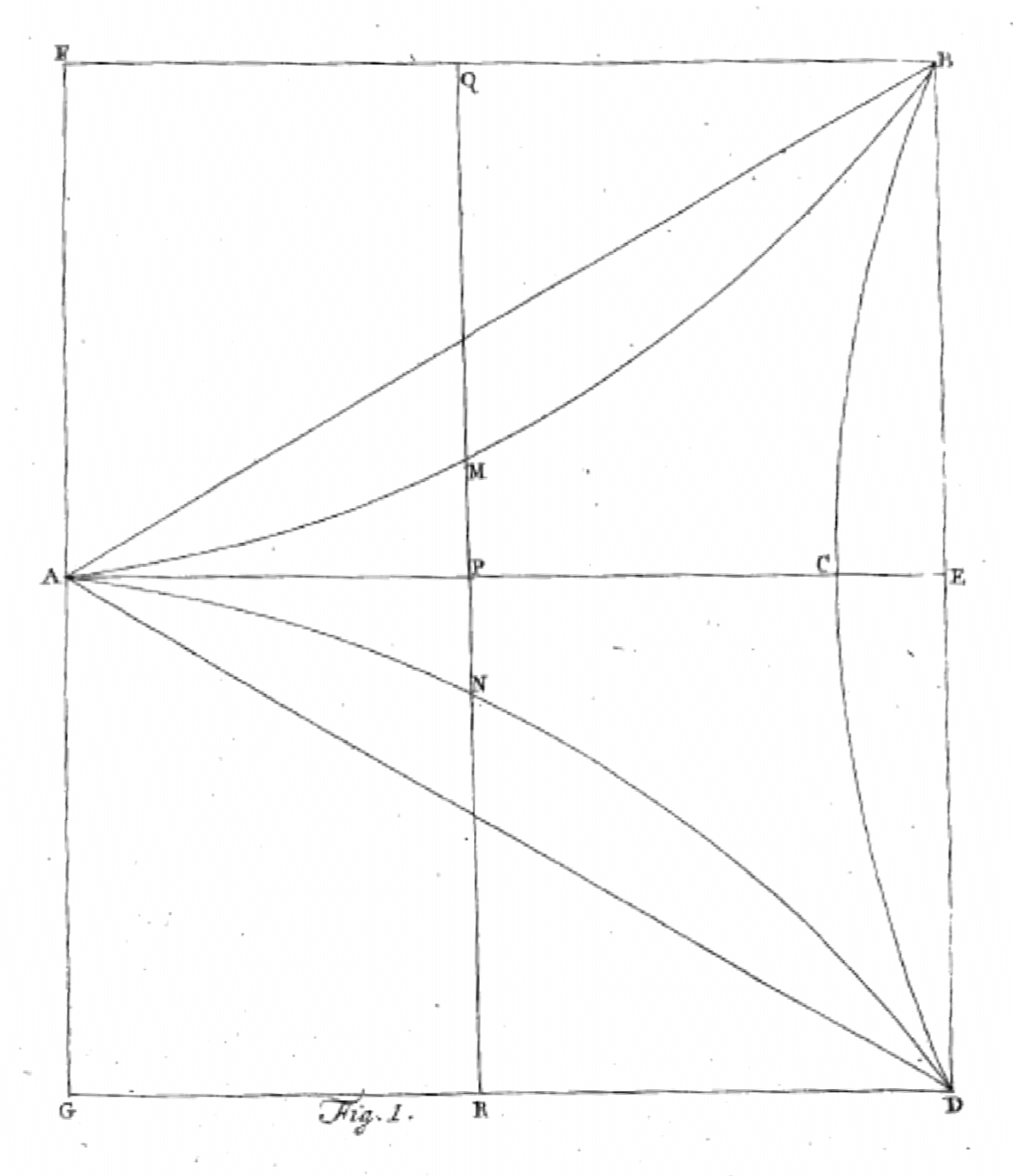}
 \caption{L. Euler, {\it Scientia Navalis}, 1749.} 
 \label{euleroboat} 
  \end{figure}

In fact, problem \eqref{minresist}-\eqref{constr} is a constrained Newton-like problem, since $L$ represents the area of the 
region between the curve $\gamma$ and the lines $y=0$ and $x=a$, taking $\{0;x,y\}$ as a coordinate system in $\R^2$.
L. Euler, after the problem statement ({\it Propositio 53, Scientia Navalis}, pg. 238) deduces the stationary  conditions in terms of differential equations and G.H. Light (in \cite{L}) proves that the  extremal curves are precisely { branches of} hypocycloids of three cusps. 
In this paper we provide an exhaustive solution of the problem \eqref{minresist}, \eqref{constr}, by exploiting the direct methods of the calculus of variations. It turns out that, { in the generality of  Euler's formulation, the problem doesn't admit a solution (see Example \ref{counterexample}). Indeed, we prove the existence of global minimizers (Theorem \ref{main1}) under the natural assumption $\gamma_1'\ge 0$. Then,  we study their precise analytical structure in dependence of the given geometric  parameters $a, h, L$. 
  In most cases,}
the optimal profile is the union of the graph of a convex or concave function (which is exactly Euler's solution)  and of a vertical segment
(Theorem \ref{explicit}). Moreover, 
 non-uniqueness of minimizers is shown to occur for certain ranges of the geometric parameters (Theorem \ref{main3}).
 % {\color{red} and the conjunction of Theorem \ref{main1}
 %and Example \ref{counterexample} shows that our result is as best as possible at least in the class of rectifiable curves.}
 
 {These results, obtained through relaxation techniques, seem to capture the essential ideas of naval architecture: indeed, it is easy to recognize that a lot of boat profiles are quite similar to the solutions of the Euler's problem (see Figure \ref{euleroboat}), suggesting that the global shapes realize a compromise between the 
dynamical performance and the total mass.}
On the other hand we guess that the non-uniqueness of solutions appearing for certain ranges of the parameters, 
suggests the possible occurrence of solutions exhibiting fine scale structures. 
Indeed, as it is well known \cite{DB} the skin of fast-swimming sharks is characterized (at the mesoscale) by the presence 
of riblet structures which are known to reduce skin friction drag in the turbulent-flow regime.
{ In this respect, it would be quite natural to ask 
if 
 a suitable modification of the Euler resistance 
 could  select a class of minimizers exhibiting at certain scales the riblet geometries which are responsible 
 of the impressive drag reduction characterizing the shark's skin, contributing in the comprehension of 
   this surprising natural morphology.}

\section{Statement of the problem and main results}\label{mainresults}

\subsection{Existence and uniqueness}

%As it is well known, the first crucial step in applying the direct method of the calculus of variations 
%relies in detecting the right function space in which to prove wellposedness of the problem. To this aim, 
Let $a>0$, $h>0$ and $L\in (0,ah)$.
We shall introduce a suitable function space for the minimization of the resistance functional.
Starting from the original formulation of the problem, a natural choice is the class of rectifiable simple curves connecting  $(0,0)$ with $(a,h)$. Admissible curves should be contained in $[0,a]\times[0,h]$ and should split such rectangle in two subsets with prescribed areas $L$ and $ah-L$.
A rectifiable simple curve is an equivalence class:
%Hence,  $\mathcal{A}_{a,h,L}$ should be defined as a set  of rectifiable simple curves that join $(0,0)$ to $(a,h)$, and  whose trace  contained in $[0,a]\times[0,h]$ splits such rectangle in two regions with prescribed area. 
the equivalence relation $\sim$ is given by orientation-preserving parametrizations, so that $\tilde\gamma\sim\gamma$ if a monotone nondecreasing mapping $\phi$ from $[0,1]$ onto itself exists such that $\tilde\gamma=\gamma\circ\phi$. %The class  of admissible curves is defined by   \eqref{aahl} since 
We shall identify each rectifiable simple curve $\gamma$ with an absolutely continuous parametrization (still denoted by $\gamma$) such that $|\gamma'(t)|\neq 0$ a.e in $(0,1)$. % motivating the definition of    $\mathcal{A}_{a,h,L}$  in \eqref{aahl}.
  %And we often tacitely identify an elemtent of $\mathcal{A}_{a,h,L}$ with
  % that we use for defining the resistance $\mathcal{F}$  from \eqref{minresist} and the area term $\int_0^1\gamma_1(t)\gamma_2'(t)\,dt$ from \eqref{constr}.
Therefore, we set
\begin{equation*}%\label{aahl}
\begin{aligned}
\mathcal{A}^0_{a,h,L}:&=\left\{\gamma\in AC([0,1]; [0,a]\times[0,h]): \gamma(0)=(0,0),\, \gamma(1)=(a,h),{\color{white}\int}\right.
\\&\gamma\ \hbox{simple,}\ \left. |\gamma'(t)|\neq 0 \mbox{ for a.e. $t\in(0,1)$},\,
\int_0^1\gamma_1(t)\gamma_2'(t)\,dt=ah-L\right\}.
\end{aligned}\end{equation*}

 We also consider the class
 \[
 \mathcal{A}_{a,h,L}:=\{\gamma\in\mathcal{A}^0_{a,h,L}: \gamma_1'(t)\ge  0 \mbox{ for a.e. $t\in(0,1)$}\}
 \]
  and the minimization problem for functional $\mathcal F$ from \eqref{minresist}, that is,
\begin{equation}\label{problem1}
\min\left\{\mathcal F(\gamma):\:\gamma\in\mathcal{A}_{a,h,L}\right\}.
\end{equation}

The  following is our first main result.

%\begin{theorem}\label{main1}
%There exists a solution to problem
%\begin{equation}\label{problem1}
%\min\left\{\mathcal F(\gamma):\:\gamma\in\mathcal{A}_{a,h,L},\, \gamma_1'(t)\ge  0 \mbox{ for a.e. $t\in(0,1)$}\right\}.
%\end{equation}
%Any solution $\gamma$ to problem \eqref{problem1} satisfies $\gamma_2'(t)\ge 0$ for a.e. $t\in(0,1)$.
%\end{theorem}

\MMM

\BBB

\begin{theorem}\label{main1} Let $a>0$, $h>0$, $L\in(0,ah)$. The following properties hold.
%Consider the problem \begin{equation}\label{problem1}
%\min\left\{\mathcal F(\gamma):\:\gamma\in\mathcal{A}_{a,h,L},\, \gamma_1'(t)\ge  0 \mbox{ for a.e. $t\in(0,1)$}\right\}.
%\end{equation}
\begin{itemize}
%\item [{\it i)}] If $h\le a$, then there exists a unique solution to problem \eqref{problem1}. 
%Any solution $\gamma$ to problem \eqref{problem1} satisfies $\gamma_2'(t)\ge 0$ for a.e. $t\in(0,1)$.
\item [ i)]  If $2L\notin(a^2,2ah-a^2)$ (in particular if $h\le a$), then there exists a unique solution to problem 
\eqref{problem1}.
\item [ii)] If  $2L\in(a^2,2ah-a^2)$, then there exist infinitely many solutions to problem \eqref{problem1}.% admits infinitely many solutions.
\end{itemize}
\end{theorem}

The choice of the subclass  $\mathcal{A}_{a,h,L}$ %that is made by those curves $\gamma$ 
%(an absolutely continuous parametrization of which is)
% such that $\gamma'_1\ge 0$ a.e. in $(0,1)$. Indeed,
% if $t\mapsto\gamma_1(t)$ is not a  nondecreasing mapping (\MMM if $\gamma_1'$ admits strong sign-changing oscillations \BBB), the 
is motivated by the fact that, without further constraints,
the problem  $\min\{\mathcal{F}(\gamma):\gamma\in\mathcal A^0_{a,h,L}\}$
admits no solution, as shown through the following 

\begin{example}\label{counterexample}\BBB
%\begin{example}
{\rm  Let $u:\R\rightarrow \R$ be a 1-periodic function defined as 

\[
u(t):=\left\{\begin{array}{lll}t\quad&\mbox{ if $0\leq t\leq \frac{1}{2}$}\\
1-t\quad&\mbox{ if $\frac{1}{2}\leq t\leq 1$}
\end{array}\right.
\]
and, for every  $n\in\mathbb N$, let $u_n(t)=u(nt)$, $t\in[0,1]$. 
Let us define $v_n\in AC[0,1]$ such that $v_n(0)=0$ and 
$v'_n(t)=\tfrac1n({u'_n}(t))_+$ for a.e. $t\in(0,1)$.
 Then we set 
\begin{equation}\label{xy}
x_n(t)=\left\{\begin{array}{lll}u_n(t)\quad&\mbox{ if $0\leq t\leq 1-\frac{1}{n}$}\medskip\\
\frac{n(t-1)+1}{2}\quad&\mbox{ if $1-\frac{1}{n}< t\leq 1,$}
\end{array}\right.
\quad\;\;
y_n(t)=\left\{\begin{array}{lll}v_n(t)\quad&\mbox{ if  $0\leq t\leq 1-\frac{1}{n}$}\medskip\\
\frac{t}{2}\quad&\mbox{ if  $1-\frac{1}{n}< t\leq 1$}
\end{array}\right.
\end{equation}

\noindent and we define $\gamma^n(t)=(x_n(t),y_n(t))$, $t\in[0,1]$. See Figure \ref{eulerocounter}. We have $\gamma^n(0)=(0,0)$, $\gamma^n(1)=(\frac{1}{2},\frac{1}{2})$ and $|(\gamma^n)'(t)|\neq 0$ for a.e. $t\in(0,1)$.
A direct computation shows that for every $n\in\mathbb N$ the area between the curve $\gamma^n$ and the
 lines $y=0$ and $x=\frac{1}{2}$ is
 \[
 \int_0^1 x_n(t)y_n'(t)\,dt=\sum_{j=0}^{n-2}\int_{j/n}^{(2j+1)/(2n)}(nt-j)\,dt+\int_{(n-1)/n}^1\frac{n(t-1)+1}{4}\,dt=\frac18.
 \]
  Thus, for any $n\in\mathbb N$ we have $\gamma^n\in \mathcal{A}^0_{a,h,L}$ with $a=h=\tfrac12$ and $L=ah-L=\tfrac18$. Moreover, another direct computation shows that
 $\mathcal F(\gamma^n)\rightarrow 0$ as $n\rightarrow \infty$.
Since $\mathcal F(\gamma)>0$ for every $\gamma\in  \mathcal{A}^0_{a,h,L}$, it follows that no minimizer exists. }
\end{example}

It is not difficult to modify the above example in order to see that, for any other value of $a,h,L$, there holds $\inf\{\mathcal F(\gamma):\gamma\in\mathcal A^0 _{a,h,L}\}=0$. Strong changing-sign oscillations of $\gamma_1'$ are indeed energetically favorable.

\begin{figure}[h!]
\begin{center}
\includegraphics[scale=0.4]{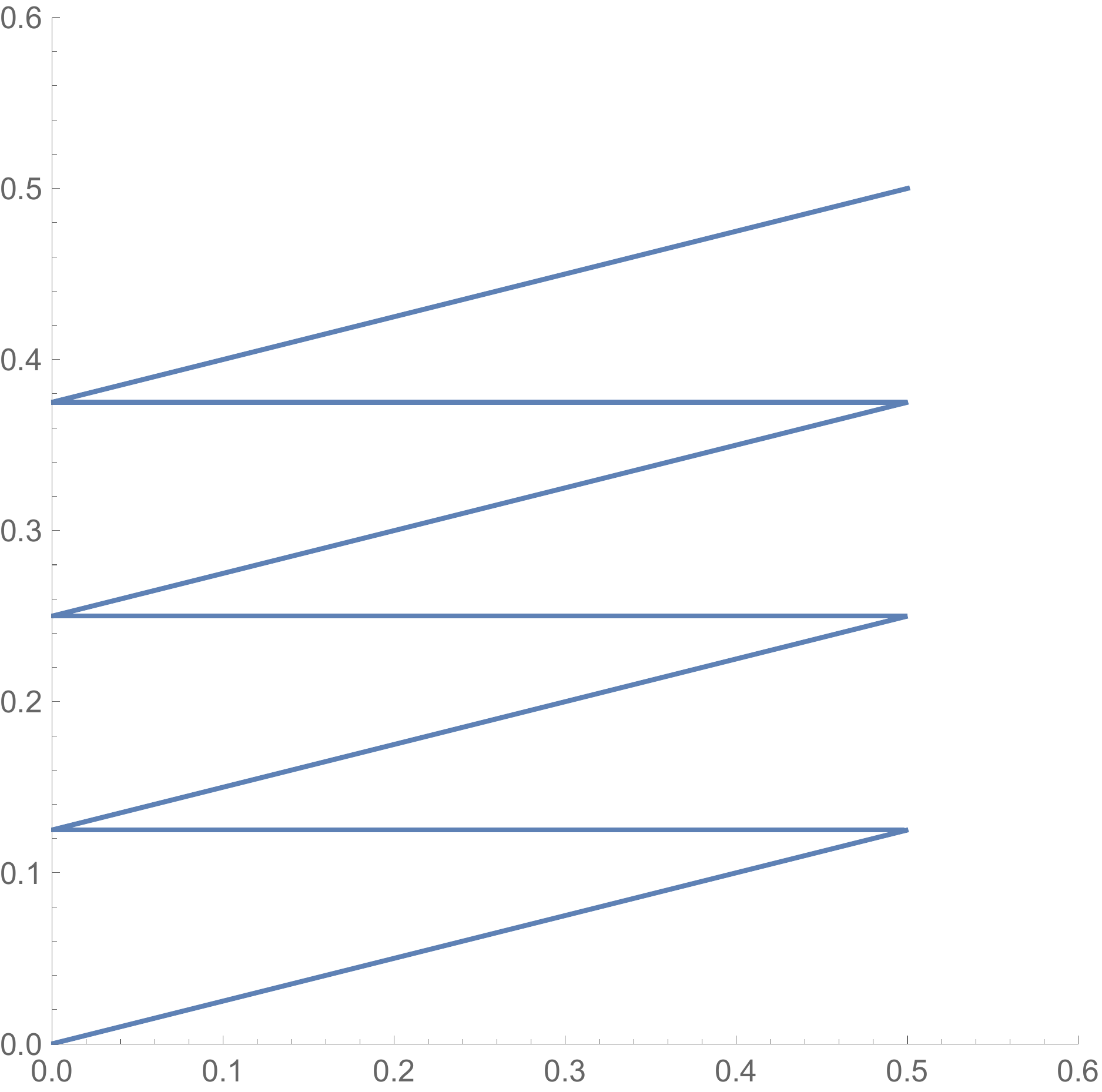}
\end{center}
 \caption{The curve $\gamma^n$, for $n=4$, in Example \ref{counterexample}}
 \label{eulerocounter}
  \end{figure}

%\begin{figure}[h!]
%\begin{center}
%\includegraphics[scale=0.5]{SQUALO.JPG}
%%\includegraphics[ height=9cm, width=13cm]{figure1.EPS}
%\end{center}
% \caption{Shark-skin surface}
% \label{shark}
%  \end{figure}

\subsection{Representation of solutions}

In the uniqueness range of Theorem \ref{main1}, the form of the solution can be obtained through an explicit parametrization. Towards this end, we need  some more notation. Here and in the following let 
\begin{equation}\label{gi}g(z):=\frac{z_+^3}{1+z^2},\quad z\in\mathbb{R}.\end{equation}
%where $z_+=\max\{z,0\}$. 
 Let $\Psi: [0,1]^{2}\to \mathbb R$ and  $\Phi : [0,1]^{2}\to \mathbb R$ be defined by
\begin{equation}\label{L}\Psi(\xi,\eta):= h+\BBB %\left\{\begin{array}{ll} &\displaystyle 
a\int_{\xi}^{\eta}\frac{1-t^{2}}{(1+t^{2})^{2}}\frac{g'(t)-g'(\xi)}{g'(\eta)-g'(\xi)}\,dt-\frac{a\eta}{1+\eta^{2}},
%\ \ \hbox{if}\ \ \xi\not =\eta\\
%&\\
%&\displaystyle-\frac{a\eta}{1+\eta^{2}}\ \ \hbox{otherwise,}\ \ \\
%\end{array}\right.
\end{equation}
\begin{equation}\label{phi}				
\Phi(\xi,\eta):=%\left\{\begin{array}{ll}&\displaystyle 
\frac{a^{2}\xi}{2}+\frac{a^{2}}{2}\int_{\xi}^{\eta}\left (\frac{g'(\eta)-g'(t)}{g'(\eta)-g'(\xi)}\right )^{2}\,dt,\BBB
%\ \ \hbox{if}\ \ \xi\not =\eta\\
%&\\
%&\displaystyle\frac{a^{2}\eta}{2}\ \ \hbox{otherwise}
%\end{array}
%\right.
\end{equation}
%\begin{equation}\label{phi}
%\Phi(\xi,\eta):=%\left\{\begin{array}{ll}&\displaystyle 
%\frac{a^{2}\eta}{2}-\frac{a^{2}}{2}\int_{\xi}^{\eta}\left (\frac{g'(t)-g'(\xi)}{g'(\eta)-g'(\xi)}\right )^{2}\,dt,
%\end{equation}
where the integral terms are understood to vanish in case $\xi=\eta$. Moreover,
let
\begin{equation}\label{T}
\mathcal T:=\left\{ (\xi,\eta): 0\le \xi\le\eta\le 1,\ \Phi(\xi,\eta)=L\right\}.
\end{equation}
Then we have
\begin{theorem} \label{mainbis}
\label{explicit} Let $a>0$, $h>0$.
Suppose that  $0<2L\le (ah)\wedge a^2$. \BBB
%\begin{equation} \label{h<aconv}\frac{2L}{a}\le h\le a,\quad \ \ \frac{2L}{a}\le a < h\end{equation}
%holds true.  
%If $2L=(ah) \wedge a^2$, then the unique solution of problem (\ref{problem1}) is $\gamma(t)=(at,ht)$, $t\in[0,1]$. 
%\begin{equation}\label{sol1}
%\gamma(\tau)=(a\tau,h\tau),\ \ \tau\in [0,1]
%\end{equation}
If $2L=(ah)\wedge a^2$, then the unique solution of problem \eqref{problem1} is given by the piecewise affine curve connecting the points $(0,0)$, $(a,a\wedge h)$ and $(a,h)$. \EEE  
Else  if $2L<(ah)\wedge a^2$, \EEE then there exists a unique minimizer $(\xi_{*},\eta_{*})$  of $\Psi$ on $\mathcal T$, there holds $\xi_{*}< \eta_{*}$, and the unique solution to problem \eqref{problem1} is
\begin{equation}\label{sol2}
\gamma_{*}(t)= \left\{\begin{array}{ll}(x_{*}(2t+\xi_{*}), y_{*}(2t+\xi_{*}))\quad &\mbox{if}\ \ t\in \left[0, \dfrac{\eta_{*}-\xi_{*}}{2}\right]\\
&\\
 (a,h+\dfrac{2(h-h_{*})}{2-\eta_{*}+\xi_{*}}(t-1))\quad &\mbox{if}\ \ t\in \left[\dfrac{\eta_{*}-\xi_{*}}{2}, 1\right]\\
\end{array}\right. 
\end{equation}
where %$(\xi_{*},\eta_{*})$ is the unique minimizer of $\mathcal L$ on $\mathcal T$, $\xi_{*}< \eta_{*},$
\begin{equation}\label{parametrbis}\displaystyle
x_{*}(\tau):=\frac{a(g'(\tau)-g'(\xi_{*}))}{g'(\eta_{*})-g'(\xi_{*})},\ \ y_{*}(\tau):=\int_{\xi_{*}}^{\tau}sx_{*}'(s)\,ds,\ \ \tau\in [\xi_{*}, \eta_{*}]
\end{equation}
and $h_{*}:= y_{*}(\eta_{*})< h$.
\end{theorem}
%\begin{remark} \rm
%\label{structure}
It has been argued in \cite{L} that, whenever $t\in [0, \frac{\eta_{*}-\xi_{*}}{2}]$,  the parametrization given in \eqref{sol2}-\eqref{parametrbis} is that of a branch of an hypocycloyd with three vertices and it is worth noticing that  its trace is the graph of a convex function. In particular, if $2L<(ah)\wedge a^2$, the optimal profile is the union of the graph of such convex function and of a vertical segment of length $h-h_{*}>0$. 

We also notice that Theorem \ref{mainbis} covers only half of the uniqueness range of the parameters.  The other half is $2L\ge(ah)\vee(2ah-a^2)$. However, the parameters fall in the latter range if $L$ satisfying the assumptions of Theorem \ref{mainbis} is changed to $ah-L$.
In particular, if $2L> (ah)\vee (2ah-a^2)$, then the corresponding optimal profile becomes the graph of a concave function joined to a vertical segment of strictly positive length.
Indeed, given the solution $\gamma_*$  in $\mathcal{A}_{a,h,L}$ from \eqref{sol2}-\eqref{parametrbis} and letting $t_*=\tfrac{\eta_*-\xi_*}{2}$, we will prove later on that the solution in $\mathcal{A}_{a,h, ah-L}$ is   just obtained by reflection and precisely it is given by

\begin{equation}\label{sterne}
\tilde\gamma_{*}(t)= \left\{\begin{array}{ll}\gamma_*(t+t_*) -(a,h_*)\quad &\mbox{if}\ \ t\in \left[0, 1- t_*\right]\\
 (a,h)-\gamma_*(1-t)\quad &\mbox{if}\ \ t\in \left[1-t_*, 1\right].
\end{array}\right. 
\end{equation}
We refer to Figure \ref{zzz} for a plot of the solutions obtained with a numerical simulation.
%\end{remark}

\begin{figure}[h!] 
\begin{center}
\includegraphics[scale=0.4]{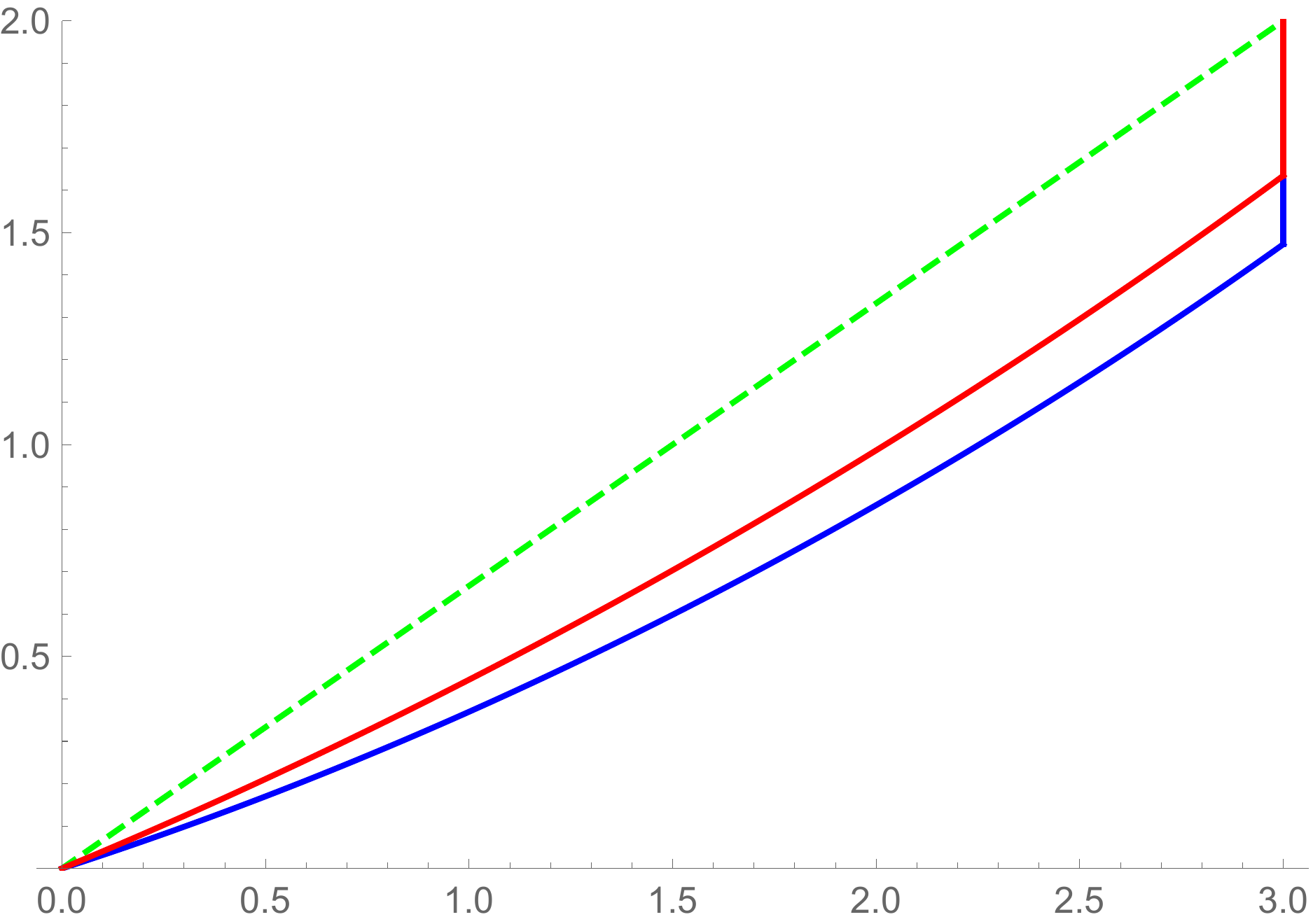} \includegraphics[scale=0.4]{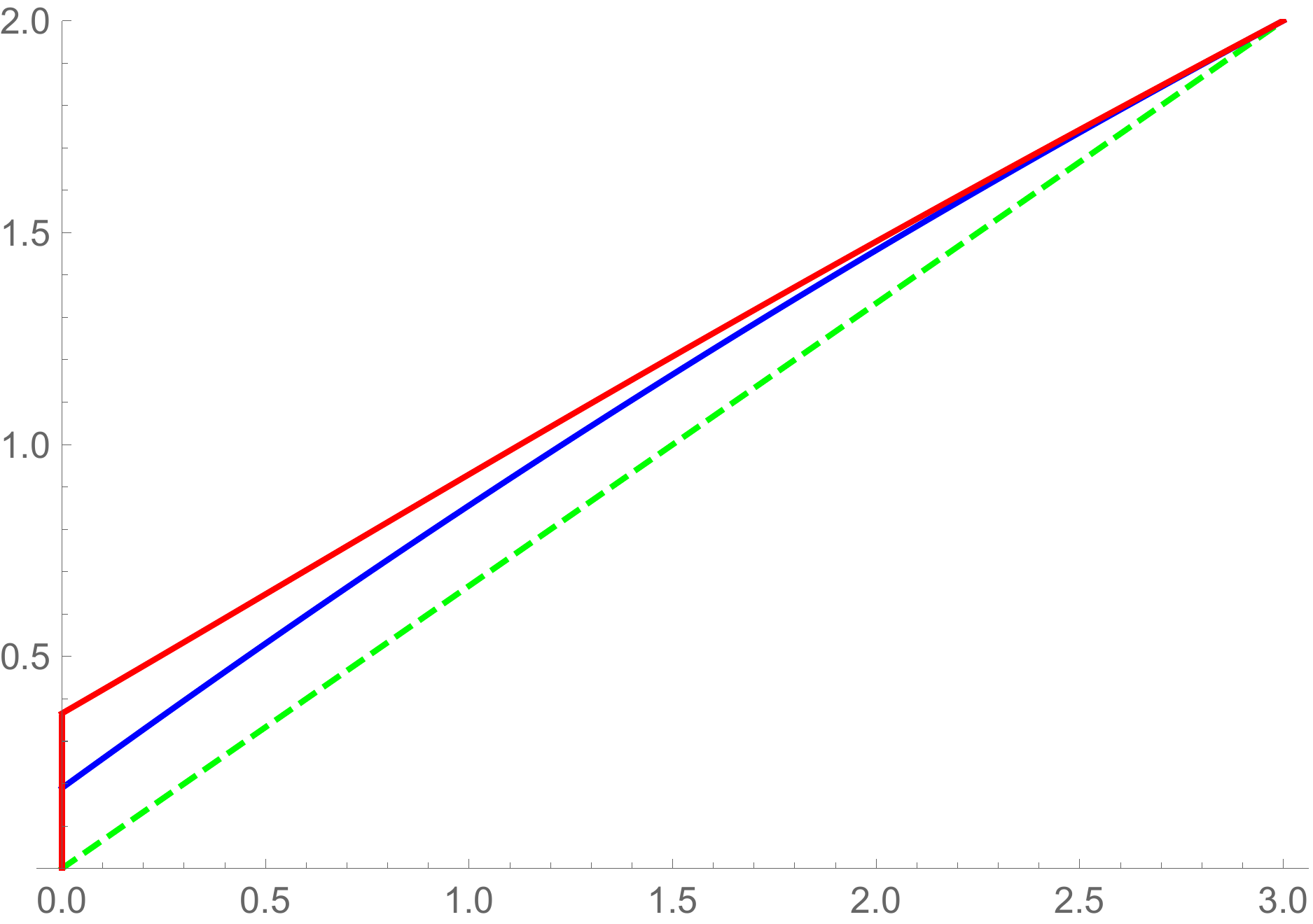}
\end{center}
 \caption{Numerical simulation of hypocycloidal solutions for $a=3$, $h=2$ and different values of $L$. 
Left:   $L=2$ (blue),  $L=2.3$ (red),  $L=3$ (green). Right: $L=3$ (green), $L=3.4$ (blue), $L=3.7$ (red).}
 \label{zzz}
  \end{figure}

Let us now discuss the non-uniqueness range of Theorem \ref{main1}. We have the following
%A simple consequence of the main theorem, that we shall prove in Lemma \ref{} below, is that in such range minimizers are characterized by the energy value $h-a/2$. Moreover,  Lemma \ref{} shows that
\begin{theorem}\label{main3}
Let $h>a>0$ and $2L\in(a^2,2ah-a^2)$.  
Then $\gamma\in\mathcal A _{a,h,L}$  is solution to problem \eqref{problem1} if and only if $\gamma_2'(t)\ge 0$ for a.e. $t\in(0,1)$ and $\gamma_1'(t)=\gamma_2'(t)$ for a.e. $t$ in $\{\gamma_1'(t)>0\}$.

%Then a solution to 
%  problem \eqref{problem1}  is given by \MMM the
   The piecewise affine curve $\gamma^\circ$ connecting the points $(0,0)$, $(0,p)$, $(a,p+a)$ and $(a,h)$, where $p:=\tfrac L a -\tfrac a 2$, is a solution to problem \eqref{problem1}. \EEE
% Such solution is  (wlog $t_1=1/3$ and $t_2=2/3$):
%\[
%\gamma^\circ(t)=\left\{\begin{array}{ll}
%(0, 3tp_L)\quad&\mbox{ if $t\in[0,1/3)$}\\
%(3at-a, 3at-a+p_L)\quad&\mbox{ if $t\in[1/3,2/3]$}\\
%(a, p_L+a+(3t-2)(h-p_L-a))\quad&\mbox{ if $t\in(2/3,1]$},
%\end{array}\right.
%\]
%where $p_L=\tfrac L a -\tfrac a 2$. 
Moreover, $\gamma^\circ$ is the unique solution to problem \eqref{problem1}  among all curves $\gamma$ that further satisfy $\{\gamma_1'(t)>0\}=(t_1,t_2)$ (up to a $\mathcal L ^1$-negligible set) for some $0<t_1<t_2<1$.
%the following further constraint: there exist $0<t_1<t_2<1$ such that 
%$\gamma_1'(t)=0$ for a.e. $t\in(0,1/3)\cup(2/3,1)$ and $\gamma_2'(t)\le\gamma_1'(t) \;\mbox{ for a.e. $t\in(1/3,2/3)$} $. 
%A direct computation shows that $\mathcal F(\gamma)=h-a/2$
\end{theorem}
%We note the the above values $1/3,2/3$ are arbitrary: up to reparametrizations, thay can be changed to any $0<t_1<t_2<1$. We also note that $\mathcal F(\gamma^\circ)=h-a/2$, as a direct computation shows.

More piecewise affine solutions to problem \eqref{problem1} can be constructed as follows.
Let $ k\in\mathbb{N}$, $k\ge 5$. Let $(x_j,y_j)$ be points in $\{(x,y)\in[0,a]\times[0,h]: x\le y\le h-a+x\}$, such that $0=x_0\le x_1\le\ldots\le x_k=a$, $0=y_0<y_1<\ldots<y_k=h$,
and such that for any $j=1,\ldots k$  there holds either $x_j=x_{j-1}$  or $x_j-x_{j-1}=y_j-y_{j-1}$. We denote by $J_2(k)$ the set of indices in $\{1,\ldots k\}$ such that $x_j=x_{j-1}$ and by $J_1(k)$ its complement in $\{1,\ldots k\}$.
Let $\hat\gamma(t)=(x_{j-1},y_{j-1})+\tfrac{t-t_{j-1}}{t_j-t_{j-1}}(x_j-x_{j-1},y_j-y_{j-1})$ for $t\in[t_{j-1},t_j]$, $j=1,\ldots, k$.
Then the energy of $\hat\gamma$  can be computed as
\begin{equation}\label{piecewiseenergy}
\mathcal F(\hat\gamma)=\sum_{j=1}^k\int_{t_{j-1}}^{t_j}\frac{(\gamma_2')_+^3}{(\gamma_1')^2+(\gamma_2')^2}\,dt=\sum_{j\in J_1(k)}\frac{y_j-y_{j-1}}2+\sum_{j\in J_2(k)}(y_j-y_{j-1})=h-\frac a2,
\end{equation}
%the contribution of the single term of the sum if $y_{j}-j_{j-1}$ if $x_{j}=x_{j-1}$ and $(y_j-y_{j-1})/2$ otherwise.
where we have exploited the fact that
$
\sum_{j\in J_1(k)}(y_j-y_{j-1})=a$ and $ \sum_{j\in J_2}(y_j-y_{j-1})=h-a$.
Hence, we see that any piecewise affine curve made by vertical segments and slope $1$ segments has the same energy of $\gamma^\circ$: it  is therefore solution to problem \eqref{problem1} as soon as the area constraint $\sum_{j\in J_1(k)}(y_j+y_{j-1})(x_j-x_{j-1})=2L$ is matched. See also Figure \ref{manysolutions}.

\begin{figure}[h!]
\begin{center}
\includegraphics[height=8cm, width=6cm]{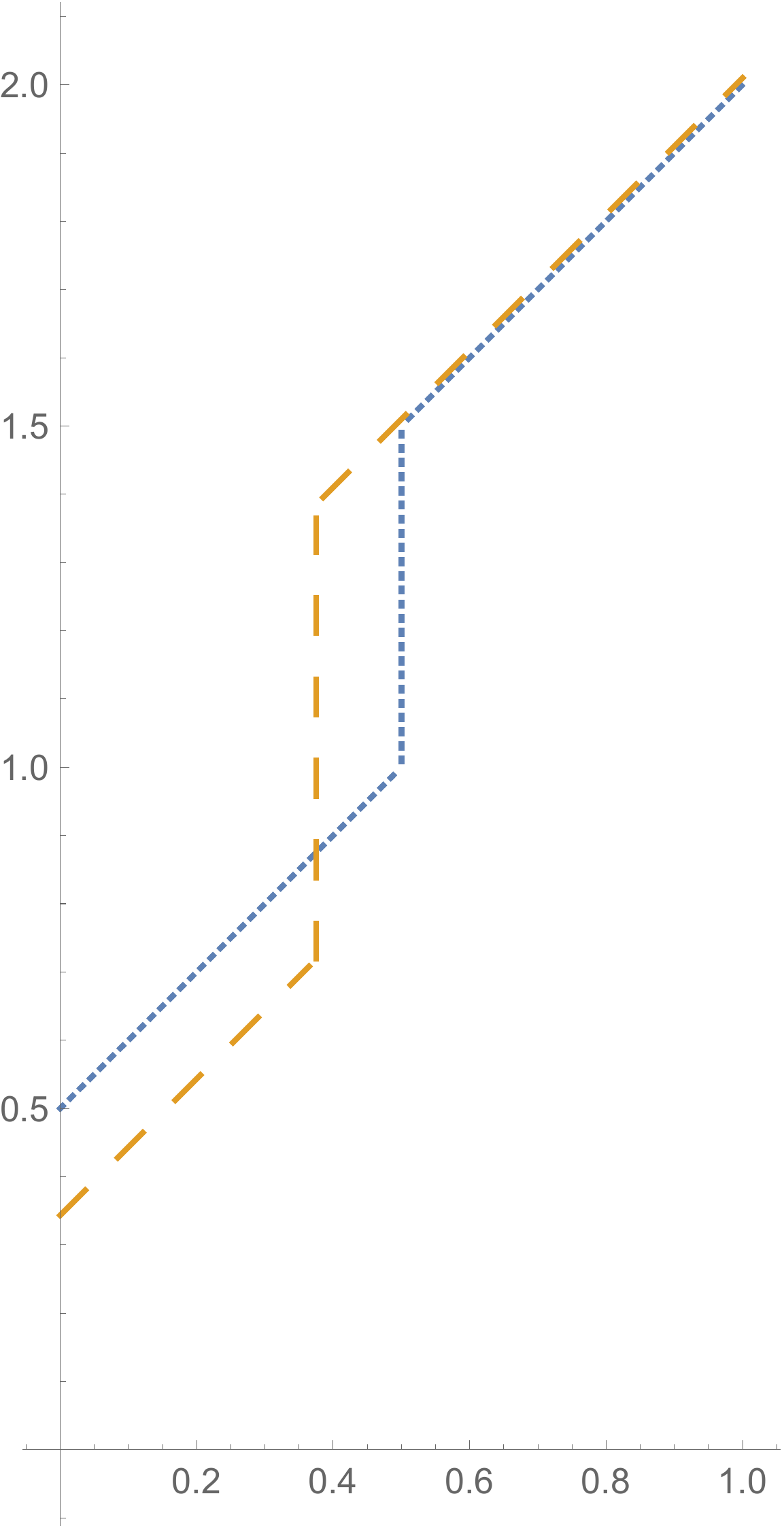}
\end{center}
 \caption{Two solutions with   $a=1$, $h=2$, $L=1.25$.}
 \label{manysolutions}
  \end{figure}

%\subsection{Parameters optimization}
Understanding $L$ as a material design constraint, it is natural to look for its optimal value, in case there is some freedom in its choice. Letting $\mathcal F_{min}(a,h,L)$ be the minimal value corresponding to the solution of problem \eqref{problem1}, we have the following  result (see also Figure \ref{concavityfigure}).

\begin{theorem}\label{main4} 
The mapping $(0,ah)\ni L\mapsto\mathcal F_{min}(a,h,L)$ is continuous and symmetric around $L=ah/2$.
 If $h\le a$, then
 % the mapping $L\mapsto\mathcal{F}_{min}(a,h,L)$
it  is  strictly decreasing on $(0,ah/2]$, strictly increasing on $[ah/2,ah)$, and its range is $[\tfrac{h^3}{a^2+h^2},h)$.
Else if $h>a$, then 
%the mapping $L\mapsto\mathcal{F}_{min}(a,h,L) $ is continuous, 
it is strictly decreasing on $(0,a^2/2]$, constant on $[a^2/2, ah-a^2/2]$, strictly increasing on $[ah-a^2/2,ah)$, and its range is $[h-a/2,h)$.
%Moreover, for fixed $L$,  $\mathcal F _{min}(a,h,L)$ can be made arbitrarily small by  taking small $h$ and large $a$.

%In particular,  $\mathcal{F}_{min}(a,2L/a ,L)$ is equal to $\frac{h^3}{a^2+h^2}$ if $2L\le a^2$ and it is equal to $h-a/2$ if $2L\ge a^2$. Therefore, if $L$ is fixed the resistance can be made (monotonically) arbitrarily small by chosing $h=2L/a$ and large $a$.
\end{theorem}

\begin{figure}[h!]
\begin{center}
\includegraphics[scale=0.4]{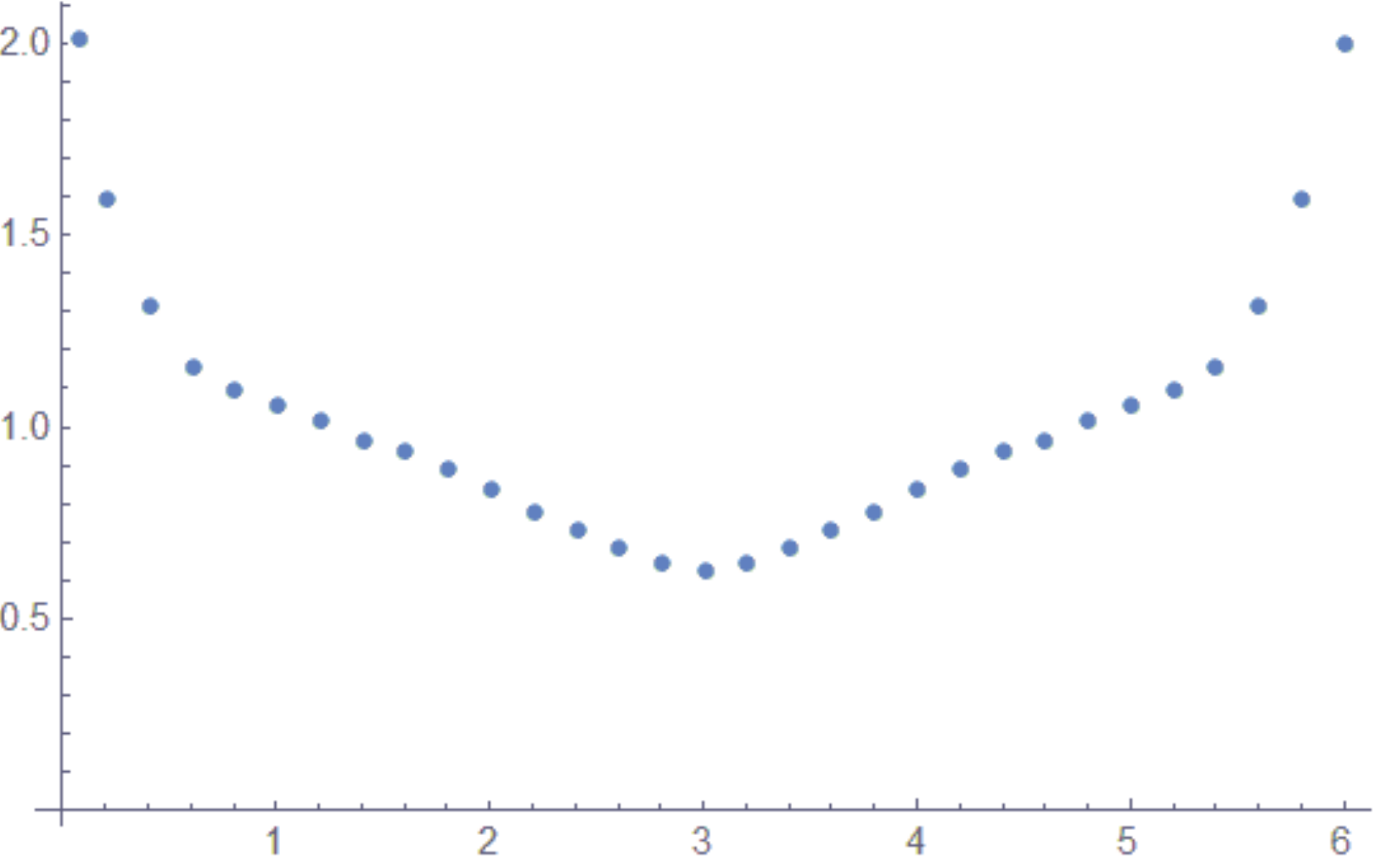}\hspace{0.31cm}
\includegraphics[scale=1.0]{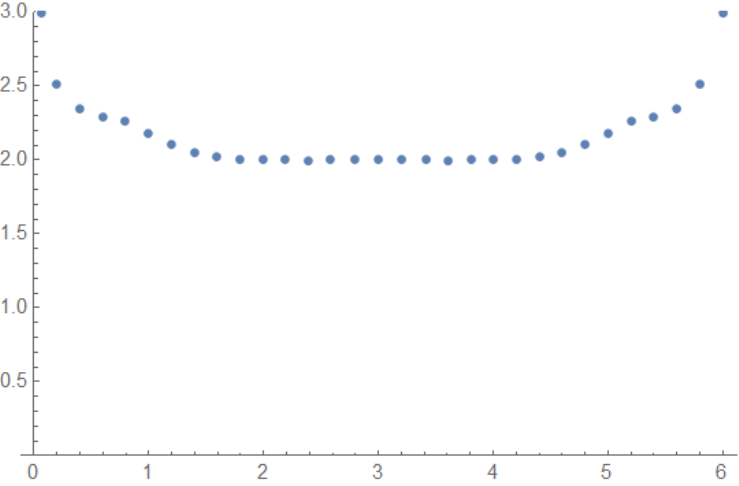}
\end{center}
 \caption{Numerical simulation of optimal energy values as a function of $L$, 
with the choice of parameters $a=3$, $h=2$ (left) and $a=2$, $h=3$ (right).}
\label{concavityfigure}
  \end{figure}

Let us conclude by remarking that the maximization problem is  easier. Indeed, we have $\sup\{\mathcal F(\gamma):\gamma\in\mathcal A_{a,h,L}\}=+\infty$. For instance, if $a=h=\tfrac12$ and $L=\tfrac18$, this can be seen by taking the sequence of curves $\bar\gamma^n(t):=(y_n(t),x_n(t))$, $t\in[0,1]$, where $x_n$ and $y_n$ are defined in \eqref{xy}. Again, the same behavior is clearly possible for any $a>0$, $h>0$, $L\in(0,ah)$. On the other hand, if we maximize $\mathcal F$ over $\mathcal A_{a,h,L}$ with the further constraint $\gamma_2'(t)\ge 0$ for a.e. $t\in(0,1)$, we may consider the estimate
\[
\mathcal F(\gamma)=\int_0^1\gamma_2'(t)-\frac{\gamma_2'(t)\,\gamma_1'(t)^2}{\gamma_1'(t)^2+\gamma_2'(t)^2}\,dt\le\int_0^1\gamma_2'(t)\,dt= h
\]
where equality holds if and only if $\gamma_1'(t)\wedge\gamma_2'(t)=0$ for a.e. $t\in (0,1)$. Hence, for any $a>0$, $h>0$ and $L\in(0,ah)$, the problem $$\max\{\mathcal F(\gamma):\gamma\in\mathcal A_{a,h,L},\,\gamma_2'(t)\ge 0 \mbox{ for a.e.  $t\in(0,1)$}\}$$  has infinitely many solutions. Any piecewise affine curve made by alternating horizontal ad vertical segments is indeed a solution as soon as the area constraint is matched, as it realizes the maximal value $h$. Such  construction is analogous to the one of piecewise affine minimizers in the nonuniqueness regime from Theorem \ref{main3}. However, these piecewise affine maximizers are found for any value of $a>0$, $h>0$ and $L\in(0,ah)$. 

\EEE

\subsubsection*{Plan of the paper} Section \ref{sectionF} provides some basic properties of functional $\mathcal{F}$. In Section \ref{BVrelaxhotel} we introduce the relaxed functional and we analyze the associated minimization problem. Section \ref{proofs} delivers the proof of the main results. 

\subsubsection*{Notation} Through the rest of the paper, without further explicit mention, it is always understood that the parameters are in the range $a>0$, $h>0$  and $L\in(0,ah)$.

\EEE

\section{Some properties of functional $\mathcal{F}$}\label{sectionF}

Let us start with a very simple estimate.
\begin{lemma}\label{<h}
There holds $$\inf\{\mathcal{F}(\gamma): \gamma\in\mathcal{A}_{a,h,L}\}<h.$$
\end{lemma}
\begin{proof}
%Let us first check the consistency of the assumptions by proving  the  simple claim $$\inf\{\mathcal{F}[\gamma]: \gamma\in\mathcal{A}_{a,h,L},\, \gamma_1'(t)\ge 0\mbox{ a.e. in $(0,1)$}\}<h.$$
Let us suppose that $2L\ge ah$ (the other case is analogous). It is enough to test the functional on the following curve made by two segments
\[
\gamma^{r}(t)=\left\{\begin{array}{ll}
(0,2tr)\quad&\mbox{ if $t\in[0,1/2]$}\\
(0,r)+(2t-1)(a,h-r)\quad&\mbox{ if $t\in[1/2,1]$},
\end{array}\right.
\]
where $r\in[0,h]$ is a parameter. Note that $\gamma^r\in\mathcal{A}_{a,h,L}$ if and only if $ar=2L-ah$. A direct computation shows that
\[
\mathcal{F}(\gamma^r)=r+\frac{(h-r)^3}{a^2+(h-r)^2}.
\]
The function $[0,h]\ni r\mapsto \mathcal{F}(\gamma^r)$ is strictly decreasing on $[0,r_*]$ and strictly increasing on $[r_*,h]$, where $r_*:=(h-a)_+$,  as easily checked. Moreover, $\mathcal{F}(\gamma^0)=\tfrac{h^3}{a^2+h^2}<h=\mathcal{F}(\gamma^h)$. In particular, such function is uniquely maximized for $r=h$ with value $h$. The result is proved.
\end{proof}
\begin{remark}\label{segment}\rm
Let $\gamma\in\mathcal{A}_{a,h,L}$. We note that if $\gamma_1(t_1)=\gamma_1(t_2)$ and $\gamma_2(t_2)-\gamma_2(t_1)=h$ for some $0\le t_1<t_2\le 1$, then $\mathcal{F}(\gamma)\ge h$. 
This happens in particular if $(0,h)\in \gamma([0,1])$ or $(a,0)\in\gamma([0,1])$.
Indeed, it is enough to compute the contribution to the functional coming from the interval $[t_1,t_2]$ where $\gamma$ is a vertical segment, which is exactly $h$. 
\end{remark}

We will often make use of approximations by means of piecewise affine curves. Here, we provide the approximation construction.

\begin{lemma} \label{infcurve}
For any $\epsilon>0$ and any $\gamma\in\mathcal{A}_{a,h,L}$,
 %and $\mathcal{F}[\gamma]<h$, 
 there exists $\bar\gamma\in\mathcal{A}_{a,h,L}$ such that 
\begin{itemize}
\item[i)] $\bar\gamma$ is piecewise affine
\item[ii)] $\bar\gamma_1'(t)>0$ for a.e. $t\in(0,1)$
\item[iii)] $\bar\gamma_2'(t)\ge0$ for a.e. $t\in(0,1)$ if the same holds for $\gamma$.
\item[iv)] $|\mathcal{F}(\bar\gamma)-\mathcal{F}(\gamma)|<\epsilon$
\item[v)] $\displaystyle\sup_{t\in[0,1]} |\bar\gamma(t)-\gamma(t)|<\epsilon$.
\end{itemize}
In particular,
there holds 
\begin{equation*}\begin{aligned}
& \displaystyle\inf\left\{\mathcal F(\gamma):\:\gamma\in\mathcal{A}_{a,h,L} \right\}
%=\\
%&\\
%&\displaystyle\inf\left\{\mathcal F(\gamma):\:\gamma\in\mathcal{A}_{a,h,L},\, \dot \gamma_1\ge  0,\ \gamma\ \hbox{piecewise affine} \right\}=\\
\\&\qquad
=\displaystyle\inf\left\{\mathcal F(\gamma):\:\gamma\in\mathcal{A}_{a,h,L},\,\gamma_1'(t)>0 \mbox{ for a.e. $t\in(0,1)$},\, \gamma\ \hbox{piecewise affine} \right\}.
\end{aligned}
\end{equation*}
\end{lemma}
\begin{proof}
	
% The arguments builds upon the same construction of previous lemmas.

\textbf{Step 1.}
We approximate any $\gamma\in\mathcal{A}_{a,h,L}$  with  a piecewise affine $\breve\gamma$  with nodes on the curve $\gamma$, such that $\breve\gamma(0)=(0,0)$ and $\breve\gamma(1)=(a,h)$.  This entails strong $W^{1,1}(0,1)$ (hence uniform) approximation of both $\gamma_1$ and $\gamma_2$. 
%Moreover, since $\gamma_1'\ge 0$ a.e. in $(0,1)$, $\gamma^N$ identifies with the graph of a picewise affine map, defined in $[0,a]$ and with values in $[0,h]$, possibly with a  finite number of jump points. 
In particular, for any $\delta>0$, $\breve\gamma$ can be chosen such that
\begin{equation}\label{areaerror}
\left|\int_0^1\breve\gamma_1(t)\,\breve\gamma_2'(t)\,dt -(ah-L)\right|=\left|\int_0^1\breve\gamma_1'(t)\,\breve\gamma_2(t)\,dt-L\right|<\delta/2,
\end{equation}
and \begin{equation}\label{gammaenne} \sup_{t\in[0,1]}|\breve\gamma(t)-\gamma(t)|<\delta/2,\qquad
|\mathcal{F}(\gamma)-\mathcal{F}(\breve\gamma)|\le C\int_0^1 |\gamma'(t)-\breve\gamma'(t)|\,dt<\delta/2,
\end{equation}
where $C=3\sqrt{3}/4$ is the Lipschitz constant of the map $\mathbb R^2\ni(x,y)\mapsto\frac{x^3}{x^2+y^2}$. %where $x_+:=\max\{x,0\}$ is the positive part.

Let $0=t_0<t_1<\ldots<t_n=1$ be the partition of $[0,1]$ such that $\gamma(t_i)$, $i=1,\ldots, n-1$ are the  nodes of $\breve\gamma$. We mention that since $ah>L>0$, if the partition is fine enough there are always grid points $t_i$, $i=1,\ldots, n-1$, such that $0<\gamma_2(t_i)<h$.  Let $I\subset\{1,\ldots,n\}$ denote the subset of indices such that $\breve\gamma_1'(t)=0$ on $(t_{i-1},t_i)$ if $i\in I$ and $\breve\gamma_1'(t)\neq 0$ on $(t_{i-1},t_i)$ otherwise. We assume wlog that $I$ does not contain two consecutive integers. We introduce the piecewise affine curve $\hat\gamma$, such that $\hat\gamma(0)=(0,0)$ and $\hat\gamma(1)=(a,h)$, whose  nodes are found at the points
\[\begin{aligned}
&\gamma(t_i)\quad\mbox{ for $i\in\{1,\ldots, n-2\}\setminus I$  (and also for $i=n-1$ if $n\notin I$)},\\& \gamma(t_i)+(2^{-n-2}(C\vee h)^{-1}\delta,0)\quad\mbox{ for  $i\in\{1,\ldots, n-1\}\cap I$},\\& \gamma(t_{n-1})-(2^{-n-2}(C\vee h)^{-1}\delta,0)\quad \mbox{ if  $n\in I$}.
\end{aligned}\] 
%for $i=1,\ldots n-1$.
For small enough $\delta$ the trace of $\hat\gamma$ is still contained in $[0,a]\times[0,h]$ and there holds $\hat\gamma_1'(t)>0$ for a.e. $t\in(0,1)$.
Clearly, if $\gamma_2'(t)\ge 0$ for a.e. $t\in(0,1)$, then $\hat\gamma$ and $\breve\gamma$ enjoy this same property. 
 It is readily seen that $\sup_{t\in[0,1]}|\breve\gamma(t)-\hat\gamma(t)|\le \delta /2$, and by computing the sums of trapezoidal areas we get
\[\begin{aligned}
& \left|\int_0^1\hat\gamma_1'(t)\,\hat\gamma_2(t)\,dt-\int_0^1\breve\gamma_1'(t)\,\breve\gamma_2(t)\,dt\right|\\
&\qquad=\frac12\left|\sum_{i=1}^n (\breve\gamma_2(t_{i-1})+\breve\gamma_2(t_i))(\breve\gamma_1(t_i)-\breve\gamma_1(t_{i-1})-\hat\gamma_1(t_i)+\hat\gamma_1(t_{i-1}))\right|\\
&\qquad\le h\sum_{i=1}^n|\breve\gamma_1(t_i)-\breve\gamma_1(t_{i-1})-\hat\gamma_1(t_i)+\hat\gamma_1(t_{i-1})|\le\delta/2.
\end{aligned}
\]
Moreover,
\[
|\mathcal{F}(\hat\gamma)-\mathcal{F}(\breve\gamma)|\le C\int_0^1|\hat\gamma'(t)-\breve\gamma'(t)|\,dt\le C\sum_{i=1}^n\int_{t_{i-1}}^{t_i}\frac{\delta 2^{-n}}{t_i-t_{i-1}}\,dt\le \delta/2.
\]
By combining the latter estimates with \eqref{areaerror} and \eqref{gammaenne}, we find
\begin{equation}\label{pmeps}
\sup_{t\in[0,1]}|\hat\gamma(t)-\gamma(t)|<\delta,\quad
\left|\int_0^1\hat\gamma_1'(t)\,\hat\gamma_2(t)\,dt-L\right|<\delta,\quad 
|\mathcal{F}(\hat\gamma)-\mathcal{F}(\gamma)|<\delta.
\end{equation}
Therefore, by taking $\delta$ small enough we see that $\hat\gamma$ satisfies properties {\it i)} to {\it v)}. Still, it does not necessarily belong to $\mathcal{A}_{a,h,L}$.

\textbf{Step 2.} In view of the previous step, we need to modify $\hat\gamma$ in order to match the area constraint.
A parametrization for $\hat\gamma$ is
\begin{equation}\label{vecchia}
\hat\gamma(t)=\hat\gamma(t_{i-1})+\frac{t-t_{i-1}}{t_i-t_{i-1}}\,(\hat\gamma(t_i)-\hat\gamma(t_{i-1}))\quad\mbox{ if $t\in[t_{i-1},t_i]$},\qquad\,i=1,\ldots,n.
\end{equation}
%where $0=t_0<t_1<\ldots<t_n=1$ is a partition of $[0,1]$.
%Let $t_0=0$, $t_n=1$ and $t_1<t_2\ldots<t_{n-1}$ such that the angle points of $\hat\gamma$ are located at $\hat\gamma(t_i)$, $i=1,\ldots, n-1$.
 Let $\sigma\in[-1,1]$. We define a new piecewise affine curve depending on $\sigma$. Let $\gamma_\sigma(t_0)=(0,0)$, $\gamma_\sigma(t_n)=(a,h)$, and  let  $\gamma_\sigma(t_i)=(\hat\gamma_1(t_i), (1-|\sigma|)\hat\gamma_2(t_i)+\sigma_+h)$, $i=1,\ldots n-1$. 
 %Let moreover 
  %$t_{-1}<0$ and $\gamma_\delta(t_{-1})=(0,0)$ if $\delta\ge0$ (resp. $t_{n+1}>1$ and $\gamma_\delta(t_{n+1})$ if $\delta<0$).
   Accordingly, let
 \begin{equation}\label{nuova}
 \gamma_\sigma(t)=\gamma_\sigma(t_{i-1})+\frac{t-t_{i-1}}{t_i-t_{i-1}}\,(\gamma_\sigma(t_i)-\gamma_\sigma(t_{i-1}))\quad\mbox{ if $t\in[t_{i-1},t_i]$},\qquad\,i=1,\ldots,n.
 \end{equation}
 %where $i$ ranges from $0$ to $n$ if $\delta\ge 0$ and from $1$ to $n+1$ if $\delta<0$.
  The  area in $[0,a]\times[0,h]$ that lies below the  curve $\gamma_\sigma$  is once more easily computed as sum of trapezoidal areas and there holds
 \begin{equation}\label{iddelta}
 \mathcal{I}(\sigma):=\int_0^1 (\gamma_\sigma)_1'(\gamma_\sigma)_2=\sigma_+ ah+(1-|\sigma|)\int_0^1\hat\gamma_1'\hat\gamma_2.
 %\int_0^1\hat\gamma_1'\hat\gamma_2+
 %\frac12\sum_{i=0}^n(2\delta_+h-|\delta|\gamma_2(t_i)-|\delta|\gamma_2(t_{i-1}))(\gamma_1(t_i)-\gamma_1(t_{i-1})).
 \end{equation}
%Since $ah>L$, taking $\eps$ small enough we also have
 Since $\int_0^1\hat\gamma_1'\hat\gamma_2<ah$, we see from \eqref{iddelta} that the map $[-1,1]\ni\sigma\mapsto\mathcal{I}(\sigma)$ is continuous strictly increasing. Moreover, it is readily seen using the second estimate in \eqref{pmeps} and \eqref{iddelta} that $\mathcal{I}(\tfrac{2\delta}{ah-L+\delta})>L+\delta$ and that $\mathcal{I}(-\tfrac{2\delta}{L+\delta})<L-\delta$. We conclude that there exists a unique value $\sigma_\delta\in(-\tfrac{2\delta}{L+\delta},\tfrac{2\delta}{ah-L+\delta})$  such that $\mathcal{I}(\sigma_\delta)=L$, so that $\gamma_{\sigma_\delta}\in\mathcal{A}_{a,h,L}$. 

It is  clear that $\sup_{t\in[0,1]}|\hat\gamma(t)-\gamma_{\sigma_{\delta}}(t)|<|\sigma_\delta|h$. 
%Since $\sigma_\delta$ vanishes with $\delta$, $\gamma_{\sigma_\delta}$ provides a uniform approximation of $\gamma$ thanks to \eqref{gammaenne}.
Eventually, by taking derivatives in \eqref{vecchia} and \eqref{nuova}  we get
\[\begin{aligned}
|\mathcal{F}(\gamma_{\sigma_\delta})-\mathcal{F}(\hat\gamma)|&\le C \int_0^1|\gamma'_{\sigma_\delta}(t)-\hat\gamma'(t)|\,dt=\sum_{i=1}^n\int_{{t_{i-1}}}^{t_i}
\frac{|\sigma_\delta||\hat\gamma_2(t_i)-\hat\gamma_2(t_{i-1})|}{t_i-t_{i-1}}\,dt\\
&\le|\sigma_\delta|\sum_{i=1}^n |\gamma_2(t_i)-\gamma_2(t_{i-1})|\le |\sigma_\delta|\int_0^1|\gamma'(t)|\,dt
%\le |\sigma_\eps|\int_0^1|\gamma'(t)|\,dt.
\end{aligned}\]
By taking \eqref{gammaenne} and the latter estimates into account, we get $$\sup_{t\in[0,1]}
|\gamma_{\sigma_\delta}(t)-\gamma(t)|<\delta+|\sigma_\delta|h,\qquad
|\mathcal{F}(\gamma_{\sigma_\delta})-\mathcal{F}(\gamma)|<\delta+|\sigma_\delta|\int_0^1|\gamma'(t)|\,dt.$$ 

Since $\sigma_\delta$ vanishes as $\delta\downarrow 0$,   if  we define, for $\delta$  small enough, $\bar\gamma:=\gamma_{\sigma_\delta}$ we obtain $\bar\gamma\in\mathcal{A}_{a,h,L}$ and  {\it i), ii) iii), iv), v)} hold.
\BBB
\end{proof}

\BBB

\section{Relaxation}\label{BVrelaxhotel}
%For any $\gamma\in\mathcal{A}_{a,h,L}$, the resistance functional $\mathcal{F}$ is defined by \eqref{minresist}.

In this section we gather some results about minimization of auxiliary functionals defined on $BV$ functions of one variable, rather than parametric curves of the plane. We start by introducing some more notation.

%For every $\gamma\in\mathcal{A}_{a,h,L},$ let
%\begin{equation}
%\displaystyle\mathcal F(\gamma):=\int_0^1\frac{(\gamma_2')_+^3}{(\gamma_1')^2+(\gamma_2')^2}\,dt
%\end{equation}
%and consider the following problem
%
%\begin{equation}
%\min\left\{\mathcal F(\gamma):\:\gamma\in\mathcal{A}_{a,h,L},\, \dot \gamma_1\ge  0 \hbox{ a.e. in }\ (0,1)\right\}.
%\end{equation}

 Let $g$ as in \eqref{gi} and let
\begin{equation}\label{g**}
g^{**}(z):=\left\{\begin{array}{lll}g(z)\quad&\mbox{ if $z<1$}\medskip\\
z-\tfrac12\quad&\mbox{ if $z\ge 1$}
\end{array}\right.
\end{equation}
be the convex envelope of $g$, i.e., the largest convex function that is smaller than or equal to $g$. In the following
%we assume  $a>0$, $h>0$ and $L$ are such that $L>ah$ and
 for every $u\in BV_{loc}(\mathbb R)$,
  $u'$ will denote the distributional derivative and $\dot u,\ u'_{s}$ its absolutely continuous and singular part respectively. 
  %Moreover, for $f\in L^1(0,a) $ we shall use the notation $$A(f):=\displaystyle\int_0^a f(x)\,dx.$$
%%%%%%%%%%%%%%%%%%%%%%%%%%
Let
\begin{equation*}\begin{aligned}
\mathcal{B}_{a,h,L}&:=\left\{u\in W^{1,1}_{loc}(\mathbb R):\, u(x)\equiv 0\ \hbox{if} \ x< 0 ,\, u(x)\equiv h\ \hbox{if} \ x> a, \; 0\le u\le h,\;\int_{0}^a u=L\right\},
%\end{equation*}
%\begin{equation*}
\\\mathcal{B}_{a,h,L}^{+}&:= \left\{u\in\mathcal{B}_{a,h,L}: u'\ge 0\right\},
%\end{equation*}
%and 
%\begin{equation*}
\\\mathcal{C}^+_{a,h,L}&:=\left\{u\in BV_{loc}(\mathbb R):\, u(x)\equiv 0\ \hbox{if} \ x< 0 ,\, u(x)\equiv h\ \hbox{if} \ x> a, \; u'\ge 0,\;\int_{0}^a u=L\right\}.
\end{aligned}\end{equation*}
We further define the functionals 
\begin{equation*}%\label{scJ+}
\begin{aligned}
\mathcal{G}(u)&:=\left\{\begin{array}{lll}\displaystyle\int_0^a g(\dot u(x))\,dx\quad&\mbox{ if $u\in\mathcal{B}_{a,h,L}$}\medskip\\
+\infty\quad&\mbox{ otherwise in }\ BV_{loc}(\mathbb R),
\end{array}\right.\\
%\end{equation}
%\begin{equation}\label{J}
\\\mathcal{J}(u)&:=\left\{\begin{array}{lll}\displaystyle\int_0^a g^{**}(\dot u(x))\,dx\quad&\mbox{ if $u\in\mathcal{B}_{a,h,L}$}\medskip\\
+\infty\quad&\mbox{ otherwise in }\ BV_{loc}(\mathbb R),
\end{array}\right.
\end{aligned}
\end{equation*}
and the functionals
\begin{equation}\label{scJ+}
\begin{aligned}
\\\mathcal{J}_{+}(u)&:=\left\{\begin{array}{lll}\displaystyle\int_0^a g^{**}(\dot u(x))\,dx\quad&\mbox{ if $u\in\mathcal{B}_{a,h,L}^{+}$}\medskip\\
+\infty\quad&\mbox{ otherwise in }\ BV_{loc}(\mathbb R),
\end{array}\right.\\
%\end{equation}
 %\begin{equation}
\\\overline{\mathcal J}_{+}(u)&:=\left\{\begin{array}{lll}\displaystyle\int_0^a g^{**}(\dot u(x))\,dx+ u'_{s}([0,a])\quad&\mbox{ if $u\in\mathcal{C}^+_{a,h,L}$}\medskip\\
+\infty\quad&\mbox{ otherwise in }\ BV_{loc}(\mathbb R).
\end{array}\right.
\end{aligned}
\end{equation}
We shall often use the shorthands $\inf\mathcal{G}$, $\inf\mathcal{J}$, $\inf\mathcal{J}_+$ $\inf{\overline{\mathcal{J}}}_+$ for the infimum over 
$BV_{loc}(\mathbb R)$. We  also write $\inf\mathcal{F}$ in place of $\inf\{\mathcal{F}(\gamma): \gamma\in\mathcal{A}_{a,h,L} \}$, which is the infimum of problem \eqref{problem1}.

 The first statement of this section is a suitable version of Lemma \ref{infcurve} for the new functionals.
 
  \begin{lemma}\label{piecewise}
 
  Let $\epsilon>0$. Let $u\in\mathcal{B}_{a,h,L}$. There exist a piecewise affine function $\bar u\in\mathcal{B}_{a,h,L}$ such that 
  $|\mathcal G(u)-\mathcal G(\bar u) |+|\mathcal J(u)-\mathcal J(\bar u)|<\epsilon$. Moreover, $\bar u \in\mathcal B^+_{a,h,L}$ if $u\in\mathcal B_{a,h,L}^+$. In particular,
  there hold
  \[
 \inf\mathcal{G}=\inf\{\mathcal{G}(u) : u\in\mathcal{B}_{a,h,L},\,\mbox{$u$ is piecewise affine}\},
 \]
  \[
 \inf\mathcal{J}=\inf\{\mathcal{J}(u) : u\in\mathcal{B}_{a,h,L},\,\mbox{$u$ is piecewise affine}\},
 \]
 \[
 \inf\mathcal{J}_+=\inf\{\mathcal{J}_+(u) : u\in\mathcal{B}^+_{a,h,L},\,\mbox{$u$ is piecewise affine}\}.
 \]

 \end{lemma}
 \begin{proof}
 By considering that both $g$ from \eqref{gi} and $g^{**}$ from \eqref{g**} are Lipschitz on $\mathbb R$, the proof follows the same line of that of Lemma \ref{infcurve}. It is in fact an application of the same construction to the case of curves in $\mathcal A_{a,h,L}$ that are graphs of functions in $\mathcal{B}_{a,h,L}$, therefore we omit the details.
  \end{proof}

 The following result shows that it is  convenient to consider nondecreasing functions.
 
 \begin{lemma}\label{steps}
 There holds 
 \[
 \inf\mathcal{J}=\inf \mathcal{J}_{+}= \inf\{\mathcal{J}(u):u\in \mathcal{B}_{a,h,L}^{+}\}.
 \]
 \end{lemma}
 \begin{proof}
 Thanks to Lemma \ref{piecewise}, it is enough to show that for any piecewise linear function $u\in\mathcal{B}_{a,h,L}$, there exists a piecewise linear nondecreasing function $w\in\mathcal{B}_{a,h,L}$ such that $\mathcal{J}(w)\le\mathcal{J}(u)$. This will be achieved is some steps.\\
 
 \noindent\textbf{Step 1.}
 For $n\in \mathbb{N}$ we shall consider sequences  of $N$ points $(x_i,y_i)_{\{i=1,\ldots N\}}\in S$, where 
 $$S:=\{(x_i, y_i)_{\{i=1,\ldots N\}}:\, (x_1,\ldots x_N)\in[0,a]^N,\, (y_1,\ldots,y_N)\in [0,h]^N,\,   0<x_1<\ldots< x_N<a\}$$
  is a connected subset of the rectangle $[0,a]^N\times [0,h]^N$.
 To a sequence of points  $(x_i,y_i)_{\{i=1,\ldots N\}}\in S$ we may associate a continuous piecewise linear function 
  $u=u_{\{x_1,y_1,\ldots x_N,y_N\}}$, joining the endpoints $(0,0)$ and $(a,h)$, with vertices located at the points $(x_i, y_i)$, and such that $0\le u\le h$. As a convention, we do not include the endpoints $(0,0)$ and $(a,h)$ in the list of vertices, and we do not exclude that three or more consecutive points lie on the same line segment.
 % In particular, we identify the set of such piecewise linear functions with the connected subset $S$ of $[0,a]^N\times [0,h]^N$  defined as
  % $$S:=\{(x_i, y_i)\subset [0,a]\times[0,h]:\: i=1,2,\ldots N,\; 0<x_1< x_2<\ldots< x_N<a\}.$$
  
 We notice that the energy of $u=u_{\{x_1,y_1,\ldots x_N,y_N\}}$ is
 $$\mathcal{J}(u)=\sum_{i=0}^{N}\int_{x_i}^{x_{i+1}}g^{**}(u(t))\,dt=\sum_{i=0}^N (x_{i+1}-x_i) \,g^{**}\!\!\left(\frac{y_{i+1}-y_i}{x_{i+1}-x_i}\right).$$ 
  In particular, $\mathcal{J}$ is continuous  on $S$, as $g^{**}$ is continuous on $\mathbb{R}$.
 We also notice that the area below the graph of $u=u_{\{x_1,y_1,\ldots x_N,y_N\}}$ is given by
 \[\begin{aligned}
 \int_0^a u(t)\,dt&=\sum_{i=0}^N\int_{x_i}^{x_{i+1}}\left(\frac{y_{i+1}-y_i}{x_{i+1}-x_i}\,(t-x_i)+y_i\right)\,dt\\&=\sum_{i=0}^N\left(\frac12 (y_{i+1}-y_i)(x_{i+1}-x_i)+(x_{i+1}-x_i)y_i\right)
 \end{aligned}
 \]
 and it is also a continuous function on $S$.
 
 Let us moreover introduce a connected subset of $S$ by
 \begin{equation}\label{S'}
 S':=\{(x_i,y_i)_{\{i=1,\ldots N\}}\in S: 0\le y_1\le y_2\le\ldots\le y_N\le h\},
 \end{equation}
 so that the corresponding function $u_{\{x_1,y_1,\ldots x_N,y_N\}}$ is a monotone nondecreasing piecewise constant functions with $N$ vertices.\\
  
 \noindent\textbf{Step 2.}
 Now, let us fix $(\bar x_i,\bar y_i)_{\{i=1,\ldots N\}}\in S$ and the corresponding function $u=u_{\{\bar x_1,\bar y_1,\ldots \bar x_N,\bar y_N\}}$. 
 %Our goal is to prove that there exists integer $M\le N$ and a piecewise  exists such that $\mathcal{G}^*(v)\le\mathcal{G}^*(u)$.
 Let $$\emptyset\neq V_1:=\mathrm{argmin}\{u(x):x\in\{\bar x_1,\ldots \bar x_N\}\}\qquad\mbox{and}\qquad v_1=\max V_1.$$ Then we recursively define $$V_j=\mathrm{argmin}\{u(x):x\in\{\bar x_1,\ldots \bar x_N\},\,x>v_{j-1}\}\qquad \mbox{and} \quad v_j=\max V_j,$$ for any $j\in\{2,\ldots N\}$ such that $\bar x_N>v_{j-1}$. Let $J:=\max\{j\in\{1,\ldots N\}: \bar x_N>v_{j-1} \}$, so we necessarily have  $v_J=\bar x_N$.
 Notice that  by construction
 \begin{equation}\label{slope}
 0<v_1<\ldots< v_J=\bar x_N,\quad 0\le u(v_1)<\ldots < u(v_J)\le h,\quad\{v_1,\ldots v_J\}\subseteq \{\bar x_1,\ldots \bar x_N\}.\end{equation}
 In particular, the continuous piecewise linear function $u_-$ having vertices exactly at the points $\{v_1,\ldots v_J\}$ (and endpoints at $(0,0), (a,h)$) is nondecreasing on $[0,a]$. 

 With the convention $v_0=0$ and $v_{J+1}=a$, on each interval $[v_j, v_{j+1}]$, $j\in (0,J)$, let us consider the line segment $$\mathfrak s_j(x)=\frac{u(v_{j+1})-u(v_j)}{v_{j+1}-v_j}\,(x-v_j)+u(v_j)$$  connecting $(v_j, u(v_j))$ and $(v_{j+1}, u(v_{j+1}))$. We claim that 
 \[
 u(x)\ge \mathfrak s_j(x) \quad\mbox{ on $[v_j,v_{j+1}]$}.
 \]
 This is obvious if $u\equiv 0$ or $u\equiv h$ in $[v_j,v_{j+1}]$, and in fact it holds with equality on $[v_J,v_{J+1}]$ since $v_J=\bar x_N$.
  Otherwise, from \eqref{slope} $\mathfrak{s_j}$ has positive slope and if by contradiction there is a point $p\in(v_j,v_{j+1})$ such that $u(p)<\mathfrak s_j(p)$, then since $u$ is piecewise linear and joins   $(v_j, u(v_j))$ with $(v_{j+1}, u(v_{j+1}))$, then $u$ needs to have at least one vertex $p'$ on the interval $(v_j,v_{j+1})$, such that $$u(p')<\mathfrak{s}_j(p')<\mathfrak{s}_j(v_{j+1})=u(v_{j+1}).$$ This is a contradiction, since by definition of $V_{j+1}$ and $v_{j+1}$ the value of $u$ at $v_{j+1}$ is minimal among all the vertex points $v$ of $u$ such that $v>v_j$. The claim is proved and since $j$ is arbitrary  we have $u_-(x)\le u(x)$ on $[0,a]$.

  For the sake of consistency, if $J<N$ we complete te sequence $(v_i, u(v_i))_{\{i=1,\ldots J\}}$ by adding $N-J$ vertices on a uniform partition of the line segment connecting $(v_J,u(v_J))$ to $(a,h)$, so that we obtain a sequence of points $(v_i, u(v_i))_{\{i=1,\ldots N\}}\in S$,  and the associated piecewise linear function is still $u_-$.

 All in all, we have constructed a sequence of $N$ vertices  $(v_i, u(v_i))_{\{i=1,\ldots N\}}\in S$, and the associated piecewise linear function $u_-$ is  nondecreasing
with $u_-(0)=0, u_-(a)=h$, it satisfies $0\le u_-\le h$, and moreover its vertices are on the graph of $u$.

  Eventually, with an analogous construction we provide another continuous piecewise constant function $0\le u_+\le h$, with $u^+(0)=0, u^+(a)=h$,  having a sequence of vertices in $S$ which lie on the graph of u, such that $u_+$ is nondecreasing and $u_+(x)\ge u(x)$ for any $x\in[0,a]$. In particular, the set of vertices of $u_+$ and $u_-$ belong to $S'$ from \eqref{S'}.\\
  
  \noindent\textbf{Step 3.}  
Given $(\bar x_i,\bar y_i)_{\{i=1,\ldots N\}}\in S$ and the associate piecewise linear function $u=u_{\{\bar x_1,\bar y_1,\ldots \bar x_N,\bar y_N\}}$ from the previous step, we consider the set 
\[
S'':=\{(x_i,y_i)_{\{i=1,\ldots N\}}\in S:(x_i,y_i)\in\mathrm{graph}(u), i=1,\ldots N \}.
\]  
  We claim that $S''$ is a conncected subset of $S$. Indeed, let $(x_i,y_i)_{i=1,\ldots N}\in S$ and $(\tilde x_i,\tilde y_i)_{i=1,\ldots N}\in S$. Then for each $t\in [0,1]$, we let $$x_i(t):=(1-t)x_i+t\tilde x_i,\qquad y_i(t):=u(x_i(t)), \qquad i=1,\ldots N$$ so that $[0,1]\ni t\mapsto (x_i(t),y_i(t))_{\{i=1,\ldots N\}}\in [0,a]^N\times [0,h]^N$ is a continuous mapping and by its very definition we have $(x_i(t),y_i(t))_{\{i=1,\ldots N\}}\in S''$ for any $t\in[0,1]$. This proves the claim.\\
  
  \noindent\textbf{Step 4.}
  We consider again a generic  piecewise linear mapping $u=u_{\{\bar x_1,\bar y_1,\ldots \bar x_N,\bar y_N\}}\in \mathcal{B}_{a,h,L}$, with vertices at $(\bar x_i,\bar y_i)_{\{i=1,\ldots N\}}\in S$. We consider the two piecewise linear nondecreasing mappings $u_+$, $u_-$, defined in Step 2. % Let $S_+$ and $S_-$ denote the respective set of vertices. %(if the vertices are less then $N$, we arbitrarily add the remaining vertices along the line segment connecting $(0,0)$ to $(v_1, u(v_1)$ for $u_-$ and similarly for $u_+$). 
  By the construction of $u^+$ and $u_-$, the respective sets of $N$ vertices belong to
   $ S'\cap S''$. Moreover, we recall that  the area below the graph is continuous on $S$, as seen in Step 1. On the other hand, still from Step 2 we have $u_-\le u\le u_+$ therefore $\int_0^au_-\le \int_0^a u=L\le \int_0^a u_+$. Since the set of vertices of $u_+$ and $u_-$ belong to $S'\cap S''$, which is a connected subset of $S$ by Step 3, and since the area is continuous on $S$, we deduce that there exists a set of vertices in $S'\cap S''$ which realizes the value $L$ of the area. We let $w$ the corresponding piecewise linear function, which therefore belongs to $\mathcal{B}_{a,h,L}$.
  If $0\le p<q\le h$ correspond to any two consecutive vertices of $w$ (or a vertex and an endpoint), since these points  lie on the graph of $u$ we have
  \[
  \int_{p}^{q} u'(t)\,dt=\int_{p}^{q} w'(t)\,dt.
  \]
 Since $g_*$ is convex on $\mathbb{R}$, by the above equality we may invoke Jensen inequality and get 
 \[
 \int_{p}^q g^{**}(w(t))\,dt\le \int_p^q g^{**}(u(t))\,dt.
 \]
 We conclude that $\mathcal{J}(w)\le \mathcal{J}(u)$, where $w$ is a nondecreasing piecewise linear function in $\mathcal{B}_{a,h,L}$.
 \end{proof}

 The following is not a $\Gamma$-convergence result since in the limsup inequality the sequence $u_j$ is not required to be converging to $u$. In any case, this will be sufficient for our later purposes. 
 
 \begin{lemma}\label{Gamma1}
 The following two properties hold true:\\
{\rm a)} for every $u\in BV_{loc}(\mathbb R)$ and every sequence $(u_j)\subset BV_{loc}(\mathbb R)$ such that $u_{j}\to u$ in $w^{*}- BV_{loc}(\mathbb R)$, there holds
\begin{equation}\label{liminf}
\displaystyle\liminf_{j\to \infty} \mathcal J_{+}(u_{j})\ge \overline{\mathcal J}_{+}(u);
\end{equation}
{\rm b)} for every $u\in BV_{loc}(\mathbb R)$ there exists a sequence $(u_{j})\subset BV_{loc}(\mathbb R)$ such that 
\begin{equation}\label{limsup}
\displaystyle\limsup_{j\to \infty} \mathcal J_{+}(u_{j})\le \overline{\mathcal J}_{+}(u).
\end{equation}
 \end{lemma}
 \begin{proof} We first prove a). Let $u_{j}\to u$ in $w^{*}- BV_{loc}(\mathbb R)$ and assume without restriction that  $\mathcal J_{+}(u_{j})$ is a bounded sequence. Then $u\in \mathcal{C}^+_{a,h,L}$ and \eqref{liminf} follow  from \cite[Theorem 3.4.1, Corollary 3.4.2]{B2}, see also \cite{AD}.
 \BBB
 
 In order to  prove b) it will be enough to assume that $u\in \mathcal{C}^+_{a,h,L}$. If this is the case by recalling that $u'_{s}$ has compact support we choose $c_{1}, c_{2}\ge 0$ such that
\begin{equation}\label{C1C2}
 c_{1}+c_{2}= u'_{s}([0,a])\qquad\mbox{and}\qquad (2c_{1}+3c_{2})\,a= 6\int_{-\infty}^{+\infty}xu'_{s}
\end{equation}
and we introduce the function $\widetilde u\in BV_{loc}(\mathbb R)$ defined by:
\begin{equation}\begin{array}{ll}\label{widetildeu}
&\widetilde u(x)= u(x)\ \ \hbox{if}\ \ x< 0, \\
&\\
& 
{\widetilde u}'= \dot u\, dx + c_{1}\delta_{a/3}+c_{2}\delta_{a/2}\qquad \mbox{in $\mathbb R$}.\\
\end{array}\end{equation}
It is readily seen that $\widetilde u(0^+)=0$
and  by using \eqref{C1C2}, \eqref{widetildeu} we get
\begin{equation*}\begin{aligned}
 \widetilde u(a^-)&= \int_{0}^{a}\widetilde u'=\int_{0}^{a} \dot u\,dx+ c_{1}+c_{2}=
%\\
%&\\
%&=
\int_{0}^{a} \dot u\,dx+u'_{s}([0,a])= u(a^+)=h,
\end{aligned}
\end{equation*}
and by taking into account \eqref{widetildeu} we get $\widetilde u(x)=h$ for every $x>a$. 
On the other hand, again by \eqref{widetildeu} and the relation $ \int_0^a \widetilde u+x \widetilde u'=\int_0^a (x\widetilde u)'=\displaystyle a \widetilde u(a^-)$, we get
\begin{equation*}
\begin{aligned}
\int_{0}^{a}\widetilde u\,dx&=a\widetilde u(a^-)-\int_{0}^{a} x\dot {\widetilde u}\,dx-\int_{0}^{a} x\widetilde u'_{s}= ah-\int_{0}^{a} x\dot u\,dx-\frac{ac_{1}}{3}-\frac{ac_{2}}{2}\\
&=\displaystyle ah-\int_{0}^{a} x\dot u\,dx-\int_{-\infty}^{+\infty}xu'_{s}
%= \\&
=\displaystyle ah-\int_{0}^{a} x\dot u\,dx-\int_{0}^{a}xu'_{s} +a(h-u(a^-))\\
&=\displaystyle au(a^-)-\int_{0}^{a} x\dot u\,dx-\int_{0}^{a}xu'_{s}=\int_{0}^{a} u\,dx=L.
\end{aligned}
\end{equation*}
Since  ${\widetilde u}'_{s}([0,a])=u'_{s}([0,a])$ we get $\overline{\mathcal J}_{+}(u)=\overline{\mathcal J}_{+}(\widetilde u)$ and it will be enough to find a sequence  $(\widetilde u_{j})\subset BV_{loc}(\mathbb R)$ such that  $\limsup_{j\to \infty} \mathcal J_{+}(\widetilde u_{j})\le \overline{\mathcal J}_{+}(\widetilde u)$ to achieve the result.

Let us consider the nondecreasing $W^{1,1}(0,a/3)$ function $w_1$ satisfying $w_1(0)=0$ and  $w_1(a/3)=\widetilde u(a/3^-)$, that is obtained by restricting $\widetilde u$ to $(0,a/3)$. Similarly, by taking the restriction of $\widetilde u$ to $(a/3,a/2)$ (resp. to $(a/2,a)$), we obtain a nondecreasing function  $w_2\in W^{1,1}(a/3,a/2)$ with $w_2(a/3)=\widetilde u (a/3^+)$ and $w_2(a/2)=\widetilde u (a/2^-)$ (resp. a nondecreasig function $w_3\in W^{1,1}(a/2,a)$ with $w_3(a/2)=\widetilde u(a/2^+)$ and $w_3(a)=h$). We let $a_0:=0$, $a_1:=a/3$, $a_2:=a/2$, $a_3:=a$. Thanks to Lemma \ref{piecewise}, for $i=1,2,3$ we approximate $w_i$ with nondecreasing piecewise affine functions $(w_{i,j})_{j\in\mathbb N}$ with same values at $a_{i-1}$ and $a_i$ and such that
\[
\int_{a_{i-1}}^{a_i} w_{i,j}\,dx= \int_{a_{i-1}}^{a_i}\widetilde u\qquad j=1,2,\ldots
\]
and
\[
\lim_{j\to+\infty}\int_{a_{i-1}}^{a_i} g^{**}(\dot w_{i,j})\,dx= \int_{a_{i-1}}^{a_i} g^{**}(\dot{\widetilde u})\,dx.
\]
Therefore, by defining $v_j:=w_{i,j}$ on $(a_{i-1},a_i)$, $i=1,2,3$ (extended to $\mathbb{R}$ with value $0$ for $x<0$ and with value $h$ for $x>a$), we get $v_j\in\mathcal C^+_{a,h,L}$ and for any $j\in\mathbb N$ the function $v_j$ is piecewise affine nondecreasing,  it is continuous outside at most two jump points at $a/3$ and $a/2$, and 
\begin{equation}\label{endpoints}
v_{j}(0)=\widetilde u(0)=0,\ v_{j}(a/3^\pm)=\widetilde u(a/3^\pm),\ v_{j}(a/2^\pm)=
\widetilde u(a/2^\pm),\ v_{j}(a)=\widetilde u(a)=h.
\end{equation}
Moreover, there holds
\begin{equation}\label{limiting}
\lim_{j\to+\infty} \int_0^a g^{**}(\dot v_j)\,dx=\int_0^a g^{**}(\dot{\widetilde u})\,dx
\end{equation}

If $c_1=c_2=0$, then $v_j\in\mathcal B^+_{a,h,L}$ and we let $\widetilde u_j=v_j$, thus the proof is concluded since \eqref{limsup} holds true. In general, 
as $v_j$  may have jump points at $a/3, a/2$, we  approximate it with a continuous piecewise affine function in $\mathcal{B}_{a,h,L}^+$ as follows. 
%We assume wlog that $c_{1}, c_{2}>0$ and

We choose  a decreasing vanishing sequence $(\lambda_{j})\subset\mathbb R$ such that $\dot v_{j}$ is constant on $(a/3-\lambda_{j},a/3),\ (a/3, a/3+\lambda_{j}), (a/2-\lambda_{j},a/2),\ (a/2, a/2+\lambda_{j})$ and we define for every $t\in [0,1]$

%we easily get that $v_{j}\to \widetilde u$ strongly in $BV_{loc}(\mathbb R)$ and that there exist 
%Assume first that $\dot{\widetilde u}\in L^{\infty}(0,a)$, let $\lambda_{j}\to 0^{+}$ and for every $t\in (0,1)$  define
%such that $v_{j}$ is continuous at $\lambda_{j},\ a/3\pm \lambda_{j},\ a/2\pm \lambda_{j}, a-\mu_{j}$ and by setting
\begin{equation*}\widetilde v_{j,t}(x):=\left\{\begin{array} {ll}
 %&\displaystyle x(t\lambda_{j})^{-1}v_{j}(\lambda_{j}t)\ \hbox{if} \
%&\\
%&\textstyle \frac{(v_{j}(\frac{a}{3}+\lambda_{j}t)-v_{j}(\frac{a}{3}-(1-t)\lambda_{j}))(x-\frac{a}{3}+(1-t)\lambda_{j})}{\lambda_{j}}+v_{j}(\frac{a}{3}-(1-t)\lambda_{j})
v_{j,t}^{*}(x)\quad&
 \mbox{if }  (t-1)\lambda_{j} < x-\frac{a}{3} < t\lambda_{j}\\
%&\\
%&\textstyle \frac{(v_{j}(\frac{a}{2}+\lambda_{j}t)-v_{j}(\frac{a}{2}-(1-t)\lambda_{j}))(x-\frac{a}{2}+(1-t)\lambda_{j})}{\lambda_{j}}+v_{j}(\frac{a}{2}-(1-t)\lambda_{j})
v_{j,t}^{**}\quad&
\mbox{if }  (t-1)\lambda_{j} < x-\frac{a}{2} < t\lambda_{j}\\
%&\\
v_{j}(x) \quad& \mbox{otherwise in  $\mathbb R$},
 \end{array}\right.
\end{equation*}
where 
\[
\begin{aligned}
&v_{j,t}^{*}(x):={\lambda_{j}^{-1}}\,{(v_{j}(\tfrac{a}{3}+\lambda_{j}t)-v_{j}(\tfrac{a}{3}-(1-t)\lambda_{j}))\,(x-\tfrac{a}{3}+(1-t)\lambda_{j})}+v_{j}(\tfrac{a}{3}-(1-t)\lambda_{j}),\\
&v_{j,t}^{**}(x):={\lambda_{j}^{-1}}\,{(v_{j}(\tfrac{a}{2}+\lambda_{j}t)-v_{j}(\tfrac{a}{2}-(1-t)\lambda_{j}))\,(x-\tfrac{a}{2}+(1-t)\lambda_{j})}+v_{j}(\tfrac{a}{2}-(1-t)\lambda_{j}).\\
\end{aligned}
\]
It is readily seen that $\widetilde v_{j,t}\in W^{1,1}_{loc}(\mathbb R),\ \widetilde v_{j,t}(0)=0,\ \widetilde v_{j,t}(a)=h$, that
$\Phi_{j}(t):=\int_0^a\widetilde v_{j,t}$ is continuous on the whole $[0,1]$ and that  $\Phi_{j}(1)\le L \le\Phi_{j}(0)$. Hence, there exists $t_{j}\in [0,1]$ such that $\int_0^a\widetilde v_{j,t_{j}}=L$, so that  $\widetilde u_{j}:=\widetilde v_{j,t_{j}}\in \mathcal{B}_{a,h,L}^{+}$ and
\begin{equation*}
\mathcal J_{+}(\widetilde u_{j})=\int_{(0,a)\setminus I_{j}}g^{**}(\dot{\widetilde u_{j}})\,dx+\int_{I_{j}}g^{**}(\dot{\widetilde u_{j}})\,dx, %=: A_{j}+B_{j}
\end{equation*}
where we have set 
$\textstyle I_{j}:=  (\frac{a}{3}+(t-1)\lambda_{j}, \frac{a}{3}+t\lambda_{j})\cup  (\frac{a}{2}+(t-1)\lambda_{j}, \frac{a}{2}+t\lambda_{j}).$
By taking into account  \eqref{g**}, \eqref{endpoints}, \eqref{limiting} and the fact that  $\lim_{j\to+\infty}|I_{j}|= 0$ we get 
\begin{equation*}
 \lim_{j\to+\infty} \int_{(0,a)\setminus I_{j}}g^{**}(\dot{\widetilde u_{j}})\,dx= \displaystyle\int_0^a g^{**}(\dot{\widetilde u}(x))\,dx,\qquad
 \lim_{j\to+\infty} \int_{I_{j}}g^{**}(\dot{\widetilde u_{j}})\,dx=  {\widetilde u}'_{s}([0,a]),
\end{equation*}
thus $\lim_{j\to+\infty}\mathcal J_{+}(\widetilde u_{j})= \overline {\mathcal J}_{+}(\widetilde u)$ and b) follows.
\end{proof}
  
  We next give an alternative representation for functional $\overline{\mathcal J}_+$ from \eqref{scJ+} and show that it admits a minimizer.
  \begin{lemma}\label{representationlemma}
  For every $u\in\mathcal{C}^+_{a,h,L}$ we have
  \begin{equation*}
\overline{\mathcal J}_{+}(u)=h+\int_0^a(g^{**}(\dot u(x))-\dot u(x))\,dx.
\end{equation*}
 \end{lemma}
 \begin{proof} Since $$u'_{s}([0,a])=u'_{s}((0,a))+u(0^+)+h-u(a^-)$$ and 
 $$u(a^-)-u(0^+)=\int_0^a\dot u\,dx+ u'_{s}((0,a))$$
  the result follows.
 \end{proof}
 \begin{lemma}\label{infJ+} The functional $\overline{\mathcal{J}}_+$ admits a minimizer over $\mathcal{C}^+_{a,h,L}$ and
 \begin{equation*}
\inf \mathcal{J}_{+}=\min \{\overline{\mathcal J}_{+}(u): u\in\mathcal{C}^+_{a,h,L}\}.
\end{equation*}
 \end{lemma}
 \begin{proof} Let $(u_{j})\subset \mathcal{B}^{+}_{a,h,L}$ be a sequence such that $\mathcal{J}_{+}(u_{j})= \inf \mathcal{J}_{+}+o(1)$ as $j\to+\infty$. Since $\dot u_{j}\ge 0,\ u_{j}(x)\equiv 0$ if $x\le 0,$ $u_{j}(x)\equiv h$ if $x\ge a$ then $u_{j}$ are equibounded in $BV_{loc}(\mathbb R)$
 hence there exists $u\in \mathcal{C}^+_{a,h,L}$ such that , up to subsequences, $u_{j}\to u$ in  $w^{*}-BV_{loc}(\mathbb R)$.
 By a)  of Lemma \ref{Gamma1} we get 
 \begin{equation*}%\label{liminf2}
\displaystyle \inf \mathcal{J}_{+}= \liminf_{j\to \infty} \mathcal J_{+}(u_{j})\ge \overline{\mathcal J}_{+}(u)
\end{equation*}
and by b) of  Lemma \ref{Gamma1} for any other $\widetilde u\in \mathcal{C}^+_{a,h,L}$there exists $\widetilde u_{j}\in  \mathcal{B}^{+}_{a,h,L}$ such that 
\begin{equation*}%\label{limsup2}
\displaystyle\limsup_{j\to \infty} \mathcal J_{+}(\widetilde u_{j})\le \overline{\mathcal J}_{+}(\widetilde u).
\end{equation*}
Therefore,% sis follows easily now by taking into account that 
$$ \overline{\mathcal J}_{+}(u)+o(1)\le \inf \mathcal{J}_{+}+o(1)\le  \mathcal J_{+}(\widetilde u_{j})+o(1)$$
and by taking the limit the result is proved.
\end{proof}
 We need now some fine properties of minimizers of $\overline{\mathcal J}_{+}$. To this aim we introduce for $\epsilon>0$ the  penalized functionals
% \begin{equation}
%G_{\epsilon}(u):=\left\{\begin{array}{lr} \overline{\mathcal J}_{+}(u)+\int_0^a\epsilon \dot u^{2}\,dx
%+\epsilon^{-1}\left[\Lambda(u_{-}^{2}+(u-h)_{+}^{2}+\dot u_{-}^{2})+(\Lambda(\MMM 0\vee u\wedge h\BBB)-L)^{2}\right]\ \;\mbox{ if $u\in \mathcal H$}&\medskip\\
%+\infty\;\mbox{ otherwise in }\ BV_{loc}(\mathbb R)&
%\end{array}\right.
%\end{equation}
 \begin{equation*}\begin{aligned}
\mathcal J_{\epsilon}(u):= \overline{\mathcal J}_{+}(u)+\int_0^a\epsilon \dot u^{2}\,dx
+\frac1\epsilon\int_0^a(u_{-}^{2}+(u-h)_{+}^{2}+\dot u_{-}^{2})\,dx
+\frac1\epsilon{\left(\int_0^a(0\vee u\wedge h\BBB)\,dx-L\right)}^2, 
%\mbox{ otherwise in }\ BV_{loc}(\mathbb R)&
\end{aligned}
\end{equation*}
defined for $u\in\mathcal H$ and extended with value $+\infty$ if $u\in BV_{loc}(\mathbb R)\setminus\mathcal H$,
where 
$$\mathcal H:=\{u\in W^{1,2}_{loc}(\mathbb R): u(x)\equiv 0\ \hbox{if} \ x< 0 ,\, u(x)\equiv h\ \hbox{if} \ x> a\}.$$
Minimizing sequences for $\mathcal J_\eps$ are equibounded in $W^{1,2}_{loc}(\mathbb R)$, therefore (up to subsequences) converging weakly in $W^{1,2}_{loc}(\mathbb R)$ and strongly in $L^2_{loc}(\mathbb R)$. By taking into account the convexity and nonnegativity of $x\mapsto x^2_-$ and $x\mapsto g^{**}(x)$, it is readily seen that the limit points minimize   $\mathcal J_\eps$ over $\mathcal{H}$. We next show that Lemma \ref{Gamma1} holds also for $\mathcal J_\eps$. 
\begin{lemma}\label{gammaconv} Let $\epsilon_{j}\to 0$ be a decreasing sequence, then\\
{\rm a)} %For every $u_{j}\to u$ in $w^{*}- BV_{loc}(\mathbb R)$  we have
for every $u\in BV_{loc}(\mathbb R)$ and every sequence $(u_j)\subset BV_{loc}(\mathbb R)$ such that $u_{j}\to u$ in $w^{*}- BV_{loc}(\mathbb R)$, there holds
\begin{equation*}
\displaystyle\liminf_{j\to \infty} \mathcal J_{\epsilon_{j}}(u_{j})\ge \overline{\mathcal J}_{+}(u);
\end{equation*}
{\rm b)} for every $u\in BV_{loc}(\mathbb R)$ there exists a sequence $(u_{j})\subset BV_{loc}(\mathbb R)$ such that 
\begin{equation*}
\displaystyle\limsup_{j\to \infty} \mathcal J_{\epsilon_{j}}(u_{j})\le \overline{\mathcal J}_{+}(u).
\end{equation*}
\end{lemma}
\begin{proof} a) is straightforward by sequential lower semicontinuity of $ \overline{\mathcal J}_{+}(u)$ an b) is obvious if $u\not\in \mathcal{C}^+_{a,h,L}$. If $u\in \mathcal{C}^+_{a,h,L}$, %we follow the line of the proof of Lemma \ref{Gamma1}.
we
  choose $c_{1}, c_{2}\ge 0$ 
  such that \eqref{C1C2} holds.
%\begin{equation}\label{c1c2}
%\left\{\begin{array}{ll} & c_{1}+c_{2}= u'_{s}((0,a))\\
%&\\
%&\displaystyle (2c_{1}+3c_{2})a= 6\int_{0}^{a}x\,du'_{s}
%\end{array}\right.
%\end{equation}
We define $\widetilde u\in BV_{loc}(\mathbb R)$ as in \eqref{widetildeu}:
as seen in the proof of Lemma \ref{Gamma1}, there holds  $\widetilde u(0+)=0,\ {\widetilde u}'_{s}([0,a])=u'_{s}([0,a])$, hence $\overline{\mathcal J}_{+}(u)=\overline{\mathcal J}_{+}(\widetilde u)$ and it is now enough to approximate $\overline{\mathcal{J}}_+(\widetilde u)$. We let $\delta_{j}\to 0^+$ such that $ \epsilon_{j}\delta_{j}^{-1}\to 0$ and we define
\begin{equation*}
\widetilde u_{j}(x)=\left\{\begin{array}{ll} x\delta_{j}^{-1}\widetilde u(\delta_{j})\quad &\mbox{if }\ 0\le x\le \delta_{j}\\
%\\
\textstyle\delta_j^{-1}{(\widetilde u(\frac{a}{3}^+)-\widetilde u(\frac{a}{3}-\delta_{j}))}\,(x-\frac{a}{3}+\delta_{j})+\widetilde u(\frac{a}{3}-\delta_{j})\quad &\mbox{if }\ -\delta_{j}\le x-\frac{a}{3}\le 0\\
%\\
\textstyle\delta^{-1}_j{(\widetilde u(\frac{a}{2}^+)-\widetilde u(\frac{a}{2}-\delta_{j}))}\,(x-\frac{a}{2}+\delta_{j})+\widetilde u(\frac{a}{2}-\delta_{j})\quad &\mbox{if }\ -\delta_{j}\le x-\frac{a}{2}\le 0\\
%\\
 \delta_j^{-1}{(h-\widetilde u(a-\delta_{j}))}\,(x-a)+h\quad &\mbox{if } \ a-\delta_{j}\le x\le a\\
%\\
\widetilde u(x)\quad &\mbox{otherwise in  $\mathbb R$}.
\end{array}\right.
\end{equation*}
Then $ \widetilde u_j\in\mathcal{H}$ and  $\limsup_{j\to\infty} \mathcal J_{\epsilon_j}(\widetilde u_j)\le\overline{\mathcal{J}}_+(\widetilde u)$  follows by arguing as in Lemma \ref{Gamma1}.
\end{proof}
%It is readily seen that $\argmin G_{\epsilon_{j}}\not\equiv \emptyset$ and the following holds.

The next lemma introduces the Euler-Lagrange equation for functional $\overline{\mathcal J}_+$, which will be a key step for the proof of Theorem \ref{mainbis}.

\begin{lemma}\label{minJbar} Let $j\in\mathbb{N}$. Let $\epsilon_{j}\to 0$ be a decreasing sequence and let $u_{j}\in \argmin_{\mathcal H} \mathcal J_{\epsilon_{j}}$. Then:\\
i) $\dot u_{j}$ is continuous and  monotone in $(0,a)$;\\
ii) ${\dot u}_{j}\ge 0$ a.e. in $(0, a)$ and  $\int_0^a\dot u_j=h$ for any $j\in\mathbb{N}$;\\
iii) there exists a (not relabeled) subsequence $(u_j)$ such that $u_{j}\to u_{*}$ in $w^{*}- BV_{loc}(\mathbb R)$ as $j\to+\infty$ and $u_{*}$ minimizes $ \overline{\mathcal J}_{+}$ over ${\mathcal{C}^+_{a,h,L}}$;\\
iv) either $\dot u_{*} \ge 1$ a.e. in $(0,a)$ or  $u_{*}\in W^{1,\infty}(0,a)$ with $0\le {\dot u}_{*}\le 1$ a.e. in $(0,a)$ and in the latter case we have
for suitable $\overline\lambda,\ \overline\mu\in \mathbb R$
\begin{equation*}%\label{eulereq}
g'(\dot u_{*} )=\overline\lambda x+\overline\mu\ \ \hbox{\ a.e. in }\ (0,a).
\end{equation*}
\end{lemma}
\begin{proof} Let $u_{j}\in \argmin \mathcal J_{\epsilon_{j}}$. Then $0\le u_j\le h$ in $\mathbb{R}$ (indeed, if this was not the case, $0\vee u_j\wedge h$ would provide a lower value for $\mathcal J_{\eps_j}$). Since $0\le u_j\le h$,  by the Du-Bois-Raymond equation, there exist a real constant $\mu_{j}$ such that
\begin{equation}\label{eulag}
h_{j}(\dot u_{j}):= -2\epsilon_{j}^{-1}\dot u_{j}^{-}+2\epsilon_{j}\dot u_{j}+(g^{**})'(\dot u_{j})=\lambda_{j}x+\mu_{j},\qquad\mbox{  $x\in(0,a)$ }
\end{equation}
where $\lambda_{j}=2\epsilon_{j}^{-1}((\int_0^a u_{j})-L)$.  Since $h_j$ is a continuous strictly increasing function, from \eqref{eulag} we have  $\dot u_{j}= h_{j}^{-1}(\lambda_{j}x+\mu_{j})$ and we see that
$\dot u_{j}$ is continuous and monotone on the whole $(0,a)$  thus proving i).

If $|\{\dot u_{j} < 0\}|> 0$ then there exists an interval $[\alpha_{j},\beta_{j}]\subset [0,a]$ such that $\dot u_{j}< 0$ in $(\alpha_{j},\beta_{j})$.  Since $0>\int_{\alpha_j}^{\beta_j}\dot u_j=u_j(\beta_j)-u_j(\alpha_j)$ and since $u_j(0)=0\le u(x) \le h=u_j(a)$ in $\mathbb{R}$, we can exclude both $\alpha_j=0$ and $\beta_j=h$. Therefore $0<\alpha_j<\beta_j<h$ and $\dot u_{j}^{+}(\alpha_{j})=\dot u_{j}^{-}(\beta_{j})=0$, hence
\begin{equation*}
2\epsilon_{j}^{-1}\dot u_{j}+2\epsilon_{j}\dot u_{j}=\lambda_{j}x+\mu_{j}
\end{equation*}
in $(\alpha_{j},\beta_{j})$ which implies $\lambda_{j}\alpha_{j}+\mu_{j}=\lambda_{j}\beta_{j}+\mu_{j}=0$, that is $\lambda_{j}=\mu_{j}=0$ so $\dot u_{j}\equiv 0$ in $(\alpha_{j},\beta_{j})$, a contradiction. Since ${\dot u}_{j}\ge 0$ a.e. in $(0, a)$ we get
\begin{equation*}%\label{equibound}
0\le \int_0^a |\dot u_{j}|=\int_0^a\dot u_{j}=h
\end{equation*}
and ii) is proven.

By ii) we get, up to subsequences, that $u_{j}\to u_{*}$ in $w^{*}- BV_{loc}(\mathbb R)$ and by point a) of Lemma \ref{gammaconv}
\begin{equation*}
\displaystyle\liminf_{j\to \infty} \mathcal J_{\epsilon_{j}}(u_{j})\ge \overline{\mathcal J}_{+}(u_{*}).
\end{equation*}
If  now $u\in \mathcal{C}^+_{a,h,L}$, we construct $\widetilde{u}_j$ from $u$ as done in the proof of Lemma  \ref{gammaconv}, which then entails along with the  minimality of $u_j$  
\begin{equation*}
\displaystyle\limsup_{j\to \infty} \mathcal J_{\epsilon_{j}}(u_{j})\le \limsup_{j\to \infty} \mathcal J_{\epsilon_{j}}(\widetilde u_{j}) 
%\mbox{\st{$\le\limsup_{j\to \infty} \mathcal J_{\epsilon_{j}}(\widetilde u\star \varphi_{j})$}}
\le \overline{\mathcal J}_{+}(u).
\end{equation*}
Hence, $\overline{\mathcal J}_{+}(u_{*})\le \overline{\mathcal J}_{+}(u)$ and iii) is proven.

We eventually prove iv). Since $(g^{**})'(\dot u_{j})\le 1$ by \eqref{eulag} and ii) we get
$$ 0\le \lambda_{j}x+\mu_{j}\le 1+2\epsilon_{j}\dot u_{j},$$
hence by integrating both members of previous inequality in $[0,a]$ and in $[0,a/3]$ and by assuming without restriction that $2h\epsilon_{j}\le a$ we get
\begin{equation*}
 0\le a\lambda_{j}+2\mu_{j}\le 4\qquad\mbox{and}\qquad 0\le a\lambda_{j}+6\mu_{j}\le 24.
\end{equation*}
Then,  by taking into account  ii) we have, up to subsequences, $\lambda_{j}\to \overline \lambda, \ \mu_{j}\to \overline \mu$ and $\epsilon_{j}\dot u_{j}\to 0$ in $L^{1}(0,a)$ so by recalling \eqref{eulag} we get
\begin{equation}\label{eulag2}
1=(g^{**})'(\dot u_{j})=\lambda_{j}x+\mu_{j}-2\epsilon_{j}\dot u_{j}\qquad\mbox{on the set $\{\dot u_{j}> 1\}$}
\end{equation}
 and $\lambda_{j}x+\mu_{j}-2\epsilon_{j}\dot u_{j}\to  \overline \lambda x+\overline \mu$ in $L^{1}(0,a)$. By i) $\dot u_{j}$ is monotone and continuous  and without restriction we may assume (up to subsequences) that $\{\dot u_{j}> 1\}=(s_{j},a)$ for some $s_{j}\in (0,a)$ and that $s_{j}\to s\in [0,a]$.
 
 If $s< a$ then by \eqref{eulag2} we get 
$\overline \lambda x+\overline \mu\equiv 1$ $\hbox{ in }$  $(s,a)$,
that is $\overline \mu=1,\ \overline \lambda=0$.
 Therefore, since
 \begin{equation*}
 (g^{**})'(\dot u_{j}){\bf 1}_{(0,s_{j})}=(\lambda_{j}x+\mu_{j}-2\epsilon_{j}\dot u_{j}){\bf 1}_{(0,s_{j})},
\end{equation*}
by taking into account the form of $(g^{**})'$ and the fact that $\dot u_j\le 1$ on $(0,s_j)$, \BBB
we get $\dot u_{j}\to 1$ a.e. on each compact subset of $(0,s)$ that is $u'_{*}=\dot u_{*}=1$ a.e on  $(0,s)$. On the other hand 
since for $j$ large enough $\dot u_{j} >1$ on each compact subset of $(s,a)$ we get $\dot u_{*}\ge 1$ a.e. on $(s,a)$ thus proving that $\dot u_{*}\ge 1$ a.e. on $(0,a)$ in this case.

If $s=a$ then $|\{\dot u_{j}> 1\}|\to 0$  and for every  $0<\beta<a$, we have $0\le \dot u_{j}\le 1$ in $(0,\beta)$  for $j$ large enough.  Thus (up to subsequences), we find $v\in L^\infty((0,\beta))$ with $\|v\|_{L^\infty((0,\beta))}\le 1$  such that  $\dot u_{j}\to v$ in $w^{*}-L^{\infty}((0, \beta))$, so $u'_{*}=v= \dot u_{*}$ on $(0, \beta)$. This holds for every $0< \beta<a$, that is,  $u'_{*}=\dot u_{*}$ on  $(0,a)$ and $u_{*}\in W^{1,\infty}(0,a), \ 0\le \dot u_{*}\le 1$ a.e. in $(0,a)$.

In addition by recalling that $\dot u_{j}\to \dot u_{*}$ in $w^{*}-L^{\infty}((0, \beta))$ and $u_{j}(x)=\int_{0}^{\beta}\dot u_{j}(t){\bf 1}_{(0,x)}\,dt$ for every $x\in (0,\beta)$ we get $u_{j}(x)\to u_{*}(x)$ in $(0,\beta)$
which, by taking into account that $u_{j}$ is convex, entails $\dot u_{j}\to \dot u_{*}$ a.e. in $(0, \beta)$ and iv) completely follows from 
\eqref{eulag} by passing to the limit as $j\to \infty$. 
\end{proof}

Next we  discuss property {\it iv)} of Lemma \ref{minJbar} in relation to the parameters range.

\begin{lemma}\label{moreproperties} %For every choice of $a, h>0$ and $L\in (0, ah)$ 
There exists a minimizer $u_{*}$ of  $\overline{\mathcal J}_{+}$ over ${\mathcal{C}^+_{a,h,L}}$ such that $u_{*}\in W^{1,\infty}(0,a)$ and $0\le {\dot u}_{*}\le 1$ a.e. in $(0,a)$.
% and $\overline{\mathcal J}_{+}(u_{*})=\mathcal J_{+}(u_{*})$. \MMM Quest'ultimo fatto forse non \`e vero, infatti $u_*$ potrebbe saltare in $0$ e in $a$ e dunque non essere nel dominio di $\mathcal{J}_+$...\BBB
Moreover, if $2L\notin[a^2, 2ah-a^2]$, %if  at least one of the conditions $$h\ge a\qquad\mbox{and}\qquad a^{2}\le 2L\le a(2h-a)$$ does not hold,
then any minimizer $u$ of $\overline{\mathcal J}_{+}$ over ${\mathcal{C}^+_{a,h,L}}$ satisfies $u\in W^{1,\infty}(0,a)$ and $ 0\le {\dot u}\le 1$ a.e. in $(0,a)$.
\end{lemma}
\begin{proof} {\underline{Case I}}:  $0< h< a$ (hence $2L < a^{2}$). By iv) of Lemma 3.10 there exists  $u_{*}\in \argmin_{{\mathcal{C}^+_{a,h,L}}} \overline{\mathcal J}_{+}$  such that either $\dot u_{*}\ge 1$ a.e. in $(0, a)$ or $u_{*}\in W^{1,\infty}(0,a)$ with $0\le \dot u_{*}\le 1$ a.e. in $(0, a)$. If the first case occurs then by taking into account that $u'_{*}\ge 0$ and $u_{*}(0^+)\ge 0$ we get  $u_{*}(a^-)\ge \int_0^a \dot u_{*}\ge a>h$, a contradiction. Hence,  $u_{*}\in W^{1,\infty}(0,a)$,  $0\le \dot u_{*}\le 1$ a.e. in $(0, a)$
and $\overline{\mathcal J}_{+}(u_{*})=\mathcal J_{+}(u_{*})$, thus proving the thesis.

{\underline{Case II}}:  $h=a$ and  $2L < a^{2}$. Choose $u_{*}\in \argmin_{\mathcal{C}^+_{a,h,L}} \overline{\mathcal J}_{+}$ as in the previous case: if  $\dot u_{*}\ge 1$ a.e. in $(0, a)$ then  by taking into account that $u'_{*}\ge 0$ and $u_{*}(0^+)\ge 0$ we get $u_*(x)\ge x$ hence $L=\int_0^au_{*}\ge a^{2}/2$, a contradiction. The thesis follows by arguing as before.

{\underline{Case III}}: $h\ge a$ and $a^{2}\le 2L\le a(2h-a)$. It is readily seen that there exists $u_{*}\in W^{1,\infty}(0,a)$ such that $\dot u_{*}=1$ a.e. in $(0, a)$ and $\int_0^a u_{*}=L$. Since $g^{**}(1)-1\le g^{**}(z)-z$ for every $z\in \mathbb R$ we get $u_{*}\in \argmin_{\mathcal{C}^+_{a,h,L}} \overline{\mathcal J}_{+}$ and a direct computation shows that $\overline{\mathcal J}_{+}(u_{*})=\mathcal J_{+}(u_{*})$, thus proving the thesis.

{\underline{Case IV}}: $h\ge a$ and $a(2h-a)< 2L< 2ah$. Assume by contradiction that $\dot u_{*}\ge 1$ a.e. in $(0, a)$: then either $u_{*}(0^+) > h-a$ or  $u_{*}(0^+) \le h-a$. In the first case we easily get $u_{*}(x)> h-a+x$, hence  $u_{*}(a^-)>h$, a contradiction. In the second one we claim that $u_{*}(x)\le h-a+x$: if this is true we get
$$a(2h-a) < 2L= 2\int_0^a u_{*}(x)\,dx\le 2\int_0^a(h-a+x)\,dx= a(2h-a),$$
a contradiction. To prove the claim it is enough to observe that if there exists $\overline x\in (0,a)$ such that $u_{*}(\overline x)> h-a+\overline x$ then by taking into account that $\dot u_{*}\ge 1$ a.e. in $(0, a)$ we get  $u_{*}(x)\ge u_{*}(\overline x)+x-\overline x> h-a+x$ for every $x\ge \overline x$ hence $u_{*}(a-)>h$, a contradiction. Therefore $u_{*}\in W^{1,\infty}(0,a)$ and $0\le u_{*}\le 1$ a.e. in $(0, a)$ also in this last case.
\end{proof}

The following is the version of Theorem \ref{main1} for functional $\overline{\mathcal J}_+$.

\begin{lemma}\label{uniquelemma} Suppose that 
%at least one of the conditions $$h> a\qquad\mbox{and}\qquad a^{2}< 2L< a(2h-a)$$ does not hold.
 $2L\notin (a^2,2ah-a^2)$. Then $\overline{\mathcal{J}}_+$ admits a unique minimizer over $\mathcal{C}^+_{a,h,L}$. Otherwise, $\overline{\mathcal{J}}_+$ admits infinitely many minimizers over $\mathcal{C}^+_{a,h,L}$.
\end{lemma}
\begin{proof}   By Lemma \ref{representationlemma}, there holds $\overline{\mathcal{J}}_+(u)=h+\int_0^a \psi(\dot u(x))\,dx$ for every $u\in BV_{loc}(\mathbb{R})$, where $\psi(x):=g^{**}(x)-x$ is a convex function on $\mathbb{R}$ which is strictly convex on $[0,1]$. 

Suppose first that  %one of the two conditions $h\ge a$ and $a^{2}\le 2L\le a(2h-a)$ does not hold 
$ 2L\notin [a^2,2ah-a^2]$.  By Lemma \ref{moreproperties},  $\overline{\mathcal{J}}_+$ admits a minimizer $u_*$ over $\mathcal{C}^+_{a,h,L}$, which necessarily satisfies $u_{*}\in W^{1,\infty}(0,a)$ with $ 0\le {\dot u}_{*}\le 1$ a.e. in $(0,a)$. On the other hand $\dot u_*=1$ a.e. in $(0,a)$ is not admissible in this range of the parameters $a,h,L$ (see Lemma \ref{moreproperties}), thus $\dot u_*<1$ on a set of positive measure in $(0,a)$. Since $\psi$ is strictly convex in $[0,1]$, if $v_*\in\mathcal{C}^+_{a,h,L}$ was another minimizer of $\overline{\mathcal{J}}_+$, not coinciding a.e. with $u_*$, we could consider $\mathcal{C}^+_{a,h,L}\ni w_*:=\tfrac12 u_*+\tfrac12 v_*$: by the strict convexity of $\psi$ in $[0,1]$, Jensen inequality would give $\overline{\mathcal{J}}_+(w_*)<\overline{\mathcal{J}}_+(u_*)$, contradicting minimality of $u_*$. Therefore, the minimizer $u_*$ of $\overline{\mathcal{J}}_+$ is unique. If $h\ge a$ and either $2L=a^2$ or $2L=a(2h-a)$, the $\mathcal{C}^+_{a,h,L}$ piecewise affine function $u_*$ having slope $1$ on $(0,a)$ is the unique minimizer of $\overline{\mathcal{J}}_+$. Indeed, in this case it is clear that if  $v_*\in\mathcal{C}^+_{a,h,L}$  satisfies $\dot v_\ast\ge 1$ a.e. in $(0,a)$, then $v_*=u_*$ a.e. $\mathbb{R}$. Therefore any admissible competitor $v_*$, not coinciding $a.e.$ with $u_*$, needs to satisfy $\dot v_*<1$ on a set of positive measure in $(0,a)$, thus it is not a minimizer due to the former Jensen inequality argument.  
\EEE

Else suppose that both the conditions $h> a$ and $a^{2}< 2L< a(2h-a)$ hold true. Since we are in Case III from the proof of Lemma \ref{moreproperties}, we see that Lemma \ref{moreproperties} and Lemma \ref{minJbar} entail existence of a minimizer $u_*$ of $\overline{\mathcal{J}}_+$ over ${\mathcal{C}^+_{a,h,L}}$ such that $\dot u_*=1$ a.e. in $(0,a)$. In this range of parameters, there necessarily holds $0<u_*(0^+)<u_*(a^-)<h$ (in order to match the area constraint).
%If $h>a$, we shall prove that nonuniqueness occurs.
 Therefore, we may consider the family $u_\epsilon(x):=(1+\epsilon)(x-a/2)+u_*(a/2)$, $x\in (0,a)$, and  for any $\epsilon>0$ small enough  $u_\epsilon$  fits the strip $[0,h]$. After having extended $u_\epsilon$ to $\mathbb{R}$ in such a way that it belongs to $\mathcal{C}^+_{a,h,L}$, from the representation of $\overline{\mathcal{J}}_+$ given by $\overline{\mathcal{J}}_+(u)=h-\int_0^a \psi(\dot u(x))\,dx$, it is clear that $\overline{\mathcal{J}}_+(u_\eps)$ does not depend on $\epsilon$, as $\psi$ is constant on $[1,+\infty)$ and the slope of $u_\eps$ is greater than $1$ for any $\epsilon>0$. 
\end{proof}

\begin{remark}\label{infinity}\rm
In case competitors with $\dot u>1$ a.e. in $(0,a)$ are present,  a large nonuniquenss phenomenon occurs. Solutions are not restricted to functions such that $\dot u$ is constant in $(0,a)$ as in the proof of Lemma \ref{uniquelemma}. For instance, it is clear that any other continuous piecewise affine curve with slopes greater or equal than $1$ on $(0,a)$, as soon as it satisfies the constraints that define $\mathcal{C}^+_{a,h,L}$, is a minimizer of  $\overline{\mathcal{J}}_+$ over $\mathcal{C}^+_{a,h,L}$. Any other graph enjoying the same properties will attain the minimum. However, by Jensen inequality we obtain that the solution defined by $\dot u=1$ in $(0,a)$ is unique among those elements $u$ of $\mathcal{C}^+_{a,h,L}$ that satisfy $u\in W^{1,\infty}(0,a)$ and $\dot u\le 1$ a.e. in $(0,a)$.
\end{remark}

This section ends with some further properties of minimizers of functional $\overline{\mathcal J}_+$.

\begin{lemma}\label{concaveconvex}
Suppose that %at least one of the conditions $$h> a\qquad\mbox{and}\qquad a^{2}< 2L< a(2h-a)$$ does not hold. 
 $2L\notin(a^2,2ah-a^2)$.
Then the unique minimizer $u$ of $\overline{\mathcal J}_+$ over $ {\mathcal{C}^+_{a,h,L}}$ provided by {\rm Lemma \ref{uniquelemma}} is either convex on $(0,a)$ with $u(0^+)=0$ or concave on $(0,a)$ with $u(a^-)=h$.% In particular,  $u(x)=0\vee (hx/a)\wedge a$ if $2L=ah$.
\end{lemma}
\begin{proof}
%By point {\it i)} of Lemma \ref{minJbar}, the unique solution $u$ provided by Lemma \ref{uniquelemma} is either convex or concave in $(0,a)$. 
By points {\it i)} and {\it iii)} of Lemma \ref{minJbar}, $u$ can be obtained as $w^{*}- BV_{loc}(\mathbb R)$ limit of $W^{1,2}_{loc}(\mathbb{R})$ functions $u_j$   that are convex for all $j$ or concave for all $j$. Up to subsequences, $u_j$ converge to $u$ pointwise in $(0,a)$ and $u$ itself is therefore either concave or convex.

   If $\dot u(x)=1$ for any $x\in(0,a)$, by Lemma \ref{uniquelemma} we are necessarily in the case $2L=a^2$ or in the case $2L=a(2h-a)$ and the proof is concluded. Else suppose that $u$ is concave and that there exists $0<c<a$ such that $u'<1$ a.e. in $(0,c)$. Suppose by contradiction that $u(a^-)<h$. Let us consider a piecewise affine approximation of $\bar u$ of $u$, with nodes on the graph of $u$, such that $\int_0^a u-\int_0^a\bar u=\eps$. By Jensen inequality, due to the strict convexity of $\psi(x):=g^{**}(x)-x$ on $(0,1),$ we have $\overline{\mathcal{J}}_+(\bar u)<\overline{\mathcal{J}}_+(u)$. On the other hand, if $\eps$ is small enough we have that $v:=\bar u+(\eps/a){\bf 1}_{(0,a)}$ belongs to $\mathcal{C}^+_{a,h,L}$ and  $\overline{\mathcal{J}}_+(v)=\overline{\mathcal{J}}_+(\bar u)$. This contradicts the minimality of $u$. In case  $u$ is convex and $u'<1$ on a set of positive measure, an analogous argument shows that $u(0^+)=0$.
\end{proof}

\begin{corollary}\label{convconc}
Let $h\le a$. If $2L\le ah$ (resp. $2L\ge a h$), then the unique minimizer $u$ of $\overline{\mathcal J}_+$ over ${\mathcal{C}^+_{a,h,L}}$ provided by {\rm Lemma \ref{uniquelemma}} is  convex on $(0,a)$ with $u(0^+)=0$ (resp. concave on $(0,a)$ with $u(a^-)=h$). In particular,  $u(x)=0\vee (hx/a)\wedge h$ if $2L=ah$.

Else suppose that $h>a$. If $2L\le a^2$ (resp. $2L\ge a(2h-a)$), then the unique minimizer $u$ of $\overline{\mathcal J}_+$ over ${\mathcal{C}^+_{a,h,L}}$ provided by {\rm Lemma \ref{uniquelemma}} is convex with $u(0^+)=0$ (resp. $concave$ with $u(a^-)=h$). In particular,
 if $2L=a^{2}$ then $u(x)=x$ in $(0,a)$.
\end{corollary}
\begin{proof}
Let $h\le a$. Suppose that $2L< ah$. Suppose by contradiction that $u$ is concave on $(0,a)$. Letting $w(x):=0\vee (hx/a)\wedge h$, since $u(a^-)=h$ by Lemma \ref{concaveconvex} and since $u$ is concave, it is clear that   $u\ge w$ in $(0,a)$. This entails $\int_0^a u\ge \int_0^a w=ah/2>L$, a contradiction. In case $2L>ah$ the argument is analogous.

The same reasoning also applies for proving the result in case  $h>a$.  
\end{proof}

\BBB
\begin{remark}\label{endingremark}\rm It is worth noticing that by symmetry reasons, if $u\in \mathcal{C}^+_{a,h,L}$ is a minimizer and it is convex in $(0,a)$, then $v(x):= h-u(a-x)$ satisfies $\overline{\mathcal J}_+(v)=\overline{\mathcal J}_+(u)$ and it is a minimizer in $\mathcal{C}^+_{a,h,ah-L}$ which is concave in $(0,a)$. Therefore all significant cases of {\rm Corollary \ref{convconc}} can be reduced to
%to $h\le a$ with $2L\le ah$ or  $h>a$ with  $2L\le a^{2}$.
  ${2L}\le (ah)\wedge a^2$ (as in {\rm Theorem \ref{mainbis}}).
\end{remark}

\BBB

\section{Proof of the main results}\label{proofs}

We go back to the analyis of functional $\mathcal{F}$. The next two results give its relation with the auxiliary functionals from Section \ref{BVrelaxhotel}.

\begin{lemma}\label{goodparametrization}
 Let $\gamma\in\mathcal{A}_{a,h,L}$ be a piecewise affine curve such that  $\gamma_1'(t)>0$ for a.e. $t\in(0,1)$. Then there exists a piecewise affine function $u\in \mathcal{B}_{a,h,L}$ such that $\mathcal{G}(u)=\mathcal{F}(\gamma)$.

Conversely,
let $u\in\mathcal{B}_{a,h,L}$ be piecewise affine. Then there exists $\gamma\in\mathcal{A}_{a,h,L}$ with $\gamma_1'(t)>0$ for a.e. $t\in(0,1)$ such that $\mathcal{F}(\gamma)=\mathcal{G}(u)$.

In particular there holds
\begin{equation*}\begin{aligned}
&\inf\left\{\mathcal F(\gamma):\:\gamma\in\mathcal{A}_{a,h,L},\,  \gamma_1'(t)>  0\,\mbox{ for a.e. }\, t\in(0,1) ,\ \gamma\ \hbox{piecewise affine} \right\}\\
&\qquad\qquad=\inf\left\{ \mathcal G(u): u\in\mathcal{B}_{a,h,L},\, u\ \hbox{piecewise linear}\right\}.\end{aligned}
\end{equation*}
\end{lemma}
\begin{proof}
%By considering reparametrizations, there is a one-to-one relation between piecewise affine curves in $\mathcal{A}_{a,h,L}$ with $\gamma_1'>0$ a.e. in $(0,1)$ and graphs of  strictly increasing piecewise affine functions in $\mathcal{B}_{a,h,L}$. 
It is enough to exploit the fact that the values of $\mathcal{F}(\gamma)$ and $\int_0^1\gamma_1(t)\gamma_2'(t)\,dt$ are invariant by reparametrization. 
If $\gamma\in\mathcal{A}_{a,h,L}$ is piecewise affine with $\gamma_1'>0$ a.e. in $(0,1)$, then $\gamma$ can be reparametrized as the graph of a continuous piecewise affine map $u$ on $[0,a]$, that is, as $[0,a]\ni t\mapsto (t,u(t))$. This is done by defining $u:=\gamma_2\circ \gamma_1^{-1}$.
Note that $u$ is absolutely continuous, as the composition of an absolutely continuous function and an absolutely continuous strictly increasing function.
By changing variables, since $(\gamma_1^{-1})'(x)=1/\gamma_1'(\gamma_1^{-1}(x))$ for a.e. $x\in(0,a)$,  we get
\[
\mathcal{G}[u]=\!\int_0^a\!\frac{((\gamma_2\circ\gamma_1^{-1})'(x))_+^3}{1+(\gamma_2\circ\gamma_1^{-1}(x))^2}\,dx=\int_0^a\!\frac{(\gamma_2'(\gamma_1^{-1}(x)))_+^3\,/\gamma_1'(\gamma_1^{-1}(x))}{(\gamma_1'(\gamma_1^{-1}(x)))^2+(\gamma_2'(\gamma_1^{-1}(x)))^2}\,dx
=\int_0^1\!\!\frac{(\gamma_2'(t))_+^3\,dt}{\gamma_1'(t)^2+\gamma_2'(t)^2},
\]
\[
\int_0^a u(x)\,dx=\int_0^a \gamma_2(\gamma_1^{-1}(x))\,dx=\int_0^1\gamma_2(t)\gamma_1'(t)\,dt=ah-\int_0^1\gamma_1(t)\gamma_2'(t)\,dt=L,
\]
showing that indeed $u\in\mathcal{B}_{a,h,L}$ and $\mathcal{G}(u)=\mathcal{F}(\gamma)$.

  Similarly, if $u\in\mathcal{B}_{a,h,L}$ is a piecewise affine map, we may consider the curve $[0,1]\ni t\mapsto \gamma(t):=(at, u(at))$. It is immediate to check that $\gamma\in\mathcal{A}_{a,h,L}$ is piecewise affine with $\gamma_1'(t)>0$ in $(0,1)$ and that $\mathcal{F}(\gamma)=\mathcal{G}(u)$.
\end{proof}
\BBB

\begin{lemma}\label{FGJ} There holds $\;\inf\mathcal{G}=\inf\mathcal{J}=\inf\mathcal{F}=\min_{\mathcal C^+_{a,h,L}}\overline{\mathcal{J}}_+$.
%\begin{equation}
%\inf\left\{ \mathcal G(u): u\ \hbox{piecewise linear}\right\}=\inf\left\{ \mathcal J(u): u\ \hbox{piecewise linear}\right\}=\inf\mathcal{J}.
%\end{equation}
\end{lemma}
\begin{proof}
 Take $u_*\in BV_{loc}(\mathbb{R})$ from Lemma \ref{moreproperties}, such that $u_*\in\argmin_{\mathcal{C}^+_{a,h,L}}\overline{\mathcal{J}}_+$ and $0\le \dot u_*\le 1$ a.e. in $(0,a)$. 
 It is easy to check that minimality of $u_*$ implies that $\dot u_*>0$ on a set of postive measure in $(0,a)$, and since the inequality $g^{**}(z)<z$ holds in $(0,1]$, by Lemma \ref{representationlemma} we get
 \[
 \overline{\mathcal{J}}_+(u_*)=h+\int_0^a(g^{**}(\dot u(x))-\dot u(x))\,dx<h.
 \]
   Lemma \ref{steps} and    Lemma \ref{infJ+} entail  $$h>\overline{\mathcal{J}}_+(u_*)%=\min \overline{\mathcal{J}}_+
   =\inf\mathcal{J}_+=\inf\mathcal{J}.$$
On the other hand, by definition of $g^{**}$ in \eqref{g**} it is clear that $\mathcal{G}\ge\mathcal{J}$, so that by the above equalities we get $\inf\mathcal{J}\le\inf\mathcal{G}$. We are left to prove the opposite inequality.

% and by Lemma \ref{steps} we get
%\[
%\inf\mathcal{G}\le\mathcal{G}(u_*)=\mathcal{J}(u_{*})=\mathcal{J}_+(u_*)=\inf{\mathcal{J}_+}=\inf\mathcal{J}\le\inf\mathcal{G},
%\]
%where we have also used $0\le u_*'\le 1$ a.e. in $(0,a)$ which entails $\mathcal{G}(u_*)=\mathcal{J}(u_{*})=\mathcal{J}_+(u_*)$. Since we have shown that $\inf\mathcal{J}=\inf\mathcal{G}$, the result follws from Lemma \ref{piecewise}.

Let $0\le t_1<t_2\le 1$ and let
$\gamma_{*}: [0,1]\to [0,a]\times [0,h]$ be defined by
\begin{equation}\label{gammastar}
\gamma_{*}(t):=\left\{\begin{array}{ll}  (0, t\,t_1^{-1}u_{*}(0^+))\quad &\mbox{ if $t\in[0,t_1)$}\\
 (a(t_2-t_1)^{-1}(t-t_1), u_*(a(t_2-t_1)^{-1}(t-t_1)) \quad& \mbox{ if $t\in[t_1,t_2]$}\\
 (a, u_{*}(a^-)+ (h-u_{*}(a^-))(1-t_2)^{-1}(t-t_2) \quad &\mbox{if $t\in(t_2,1]$}.
\end{array}\right.
\end{equation}
It is readily seen that $\gamma_{*}\in \mathcal{A}_{a,h,L}$ and that 
\begin{equation*}\label{47}\begin{aligned}
\mathcal F(\gamma_{*})&=u_*(0^+)+h-u_*(a^-)+\int_{t_1}^{t_2}a(t_2-t_1)^{-1}\frac{u_*'(a(t_2-t_1)^{-1}(t-t_1))^3}{1+u_*'(a(t_2-t_1)^{-1}(t-t_1))^2}\,dt\\
&=\int_{0}^{a}(g({\dot u}_{*}(x))- {\dot u}_{*}(x))\,dx + h%=\int_{0}^{a}(g^{**}({\dot u}_{*})- {\dot u}_{*})\,dx + h,
\end{aligned}
\end{equation*}
where we can replace $g$ with $g^{**}$ in the last line  due to   $0\le {\dot u}_{*}\le 1$. Therefore, by Lemma \ref{representationlemma} we get $\mathcal F(\gamma_{*})=\overline{\mathcal J}_{+}(u_{*})<h$. We next take $\eps>0$ and a piecewise affine curve $\bar\gamma\in\mathcal{A}_{a,h,L}$ such that $\bar\gamma_1'(t)>0$ for a.e.  $t\in(0,1)$ and $|\mathcal{F}(\gamma_*)-\mathcal{F}(\bar\gamma)|<\eps$, which is possible by Lemma \ref{infcurve}.
By Lemma \ref{goodparametrization} there is $\bar u\in\mathcal{B}_{a,h,L}$ such that $\mathcal{G}(\bar u)=\mathcal{F}(\bar\gamma)$. Summing up we have 
\[
\inf\mathcal G\le \mathcal{G}(\bar u)=\mathcal{F}(\bar\gamma)<\mathcal{F}(\gamma_*)+\eps=\overline{\mathcal{J}}_+(u_*)+\eps,
\]
and by  arbitrariness of $\eps$ we get $\inf\mathcal{G}\le \overline{\mathcal{J}}_+(u_*)%=\min \overline{\mathcal{J}}_+
=\inf\mathcal{J}$.

We have shown that $%\min\overline{\mathcal{J}}_+=
\inf\mathcal{J}=\inf\mathcal{G}$.  Lemma \ref{infcurve}, Lemma \ref{goodparametrization} and Lemma \ref{piecewise} imply that $\inf\mathcal{G}=\inf\mathcal{F}$, concluding the proof.
\BBB
\end{proof}

%\section{Proof of the main results}

Before proceeding to the proof of  Theorem \ref{main1}, we need three more technical lemmas.

%With the next lemma we prove that the component $\gamma_2$ of a solution $\gamma$ to problem \eqref{main1} is necessarily nondecreasing. As a consequence, the problem \eqref{main1} and the problem $$\min\{\mathcal{F}[\gamma]:\gamma\in\mathcal{A}_{a,h,L},\,\gamma_1'(t)\ge 0\,\mbox{ for a.e. } t\in(0,1),\,\gamma_2'(t)\ge 0\, \mbox{ for a.e. } t\in(0,1)\}$$
%have same solutions and same minimal values.

\begin{lemma}\label{+class}
Let $\gamma\in\mathcal{A}_{a,h,L}$ and suppose that there exist $t_1,t_2$, with $0\le t_1<t_2\le 1$, such that $\gamma_2(t_2)<\gamma_2(t_1)$. Then 
$
\mathcal{F}(\gamma)>\inf\mathcal{F}.
$
\end{lemma}
\begin{proof}
%In view of Lemma \ref{<h}, it is not restrictive to assume that $\mathcal{F}[\gamma]<h$. By  Lemma \ref{infcurve},
 %it will be enough to prove that
%\[
%\mathcal{F}(\gamma)\ge \inf\left\{\mathcal F(\gamma):\:\gamma\in\mathcal{A}_{a,h,L},\, \dot \gamma_1>  0,\ \gamma\ \hbox{piecewise affine} \right\}.
%\]
We let  $\epsilon>0$ and we let $\hat\gamma\in\mathcal{A}_{a,h,L}$ be a piecewise affine approximation of $\gamma$ with $\hat\gamma_1'(t)>0$ a.e. in $(0,1)$, 
such that $|\mathcal{F}(\gamma)-\mathcal{F}(\hat\gamma)|<\epsilon$ and such 
 that $|\hat\gamma_2(t_i)-\gamma_2(t_i)|<\epsilon$, $i=1,2$.  We let $(x_p,y_p):=(\hat\gamma_1(t_1),\hat\gamma_2(t_1))$ and $(x_q,y_q):=(\hat\gamma_1(t_2),\hat\gamma_2(t_2))$. We may assume wlog that 
 \begin{equation*}y_q< \hat\gamma_2(t)< y_p\quad\mbox{ for any $t\in(t_1,t_2)$},\end{equation*} 
 otherwise we could define
 \[
 \tilde t_1:=\max\{t\in [t_1,t_2]:\hat\gamma_2(t)\ge\hat\gamma_2(t_1) \},\quad \tilde t_2:=\min\{t\in[\tilde t_1,t_2]:\hat\gamma_2(t)\le\hat\gamma_2(t_2)\}
 \]
 and subsequently redefine
 $(x_p,y_p):=(\hat\gamma_1(\tilde t_1),\hat\gamma_2(\tilde t_1))$ and $(x_q,y_q):=(\hat\gamma_1(\tilde t_2),\hat\gamma_2(\tilde t_2))$.

Notice that $\hat\gamma$ coincides on $[0,a]$ with the graph of a   piecewise affine function $\hat u\in\mathcal{B}_{a,h,L}$. Hence, by the proof of Lemma \ref{steps} there exists a new piecewise affine curve $u\in\mathcal{B}_{a,h,L}^+$ %(\MMM check the $+$ \EEE),
 having ordered vertices at the points $(0,0)=(s_0,u(s_0)),\,(s_1,u(s_1)),\ldots,(s_k,u(s_k))=(a,h)$ along the curve $\hat\gamma$. We let $S:=\{s_0,s_1,\ldots,s_k\}$. Jensen inequality ensures that
\begin{equation}\label{allintervals}
\int_{s_j}^{s_{j+1}}g^{**}(u'(x))\,dx\le\int_{s_{j}}^{s_{j+1}}g^{**}(\hat u'(x))\,dx,\qquad j=0,\ldots, k-1.
\end{equation}
 
 We let
 \[
 x_1=\max\{ s\in S, s\le x_p\},\quad    x_4=\min\{ s\in S, s\ge x_q\},\quad y_1=u(x_1),\quad y_4=u(x_4).
 \]
 Supposing that $\{s\in S: x_p< s< x_q\}\neq\emptyset$, we further define
 \[
 x_2=\min\{ s\in S, s> x_p\},\quad  x_3=\max\{ s\in S, s< x_q\},\quad y_2=u(x_2),\quad y_3=u(x_3),
 \]
 so that $y_1\le y_2\le y_3\le y_4$. Else if $\{s\in S: x_p< s< x_q\}=\emptyset$,
  we define $x_2=x_3=\tfrac{x_p+x_q}{2}$ and $y_2=y_3=\tfrac{y_p+y_q}{2}$. 
 %(we notice that this is always the case if $y_1\ge y_p$, since $u$ is nondecraeasing, by taking into account the constraint on $\hat u$ that follows from \eqref{explain}).
%We also introduce the corresponding values $y_j:=u(x_j)$, $j=1,2,3,4$.
 By construction, there always holds $y_q\le y_2\le y_3\le y_p$.

%Suppose first that $y_1<y_p$. 
By repeated use of Jensen inequality and since from \eqref{g**} we have $g^{**}(z)=0$ for $z\le 0$, there hold
\begin{equation}\begin{aligned}\label{4}
&\int_{x_1}^{x_2}g^{**}(u')+\int_{x_3}^{x_4}g^{**}(u')\le(x_2-x_1)\,g^{**}\left(\tfrac{y_2-y_1}{x_2-x_1}\right)+(x_4-x_3)\,g^{**}\left(\tfrac{y_4-y_3}{x_4-x_3}\right),\\
%&\int_{x_3}^{x_4}g^{**}(u')=(x_4-x_3)\,g^{**}\left(\tfrac{y_4-y_3}{x_4-x_3}\right),\\
& \int_{x_1}^{x_2}g^{**}(\hat u')\ge \int_{x_1}^{x_p}g^{**}(\hat u')\ge (x_p-x_1)\,g^{**}\left(\tfrac{y_p-y_1}{x_p-x_1}\right)\ge (x_2-x_1)\,g^{**}\left(\tfrac{y_p-y_1}{x_2-x_1}\right),\\
& \int_{x_3}^{x_4}g^{**}(\hat u')\ge \int_{x_q}^{x_4}g^{**}(\hat u')\ge (x_4-x_q)\,g^{**}\left(\tfrac{y_4-y_q}{x_4-x_q}\right)\ge (x_4-x_3)\,g^{**}\left(\tfrac{y_4-y_q}{x_4-x_3}\right),
\end{aligned}
\end{equation}
where the mapping $(0,+\infty)\times\mathbb{R}\ni(x,y)\mapsto x g^{**}(y/x)$ is understood to be extended by continuity to $x=0$ (with value $y_+$), and we used the fact that $[0,+\infty)\ni x\mapsto x g^{**}(y/x)$ is nonincreasing for any $y\in\mathbb{R}$. 
%only the first and third relations have to be considered in case   $x_2=x_3=x_4$.
Thanks to \eqref{allintervals} and \eqref{4} we get
\begin{equation}\label{doppia}
%\begin{aligned}
\mathcal{J}(\hat u)-\mathcal{J}(u)\ge (x_2-x_1)\!\left(g^{**}\!\left(\tfrac{y_p-y_1}{x_2-x_1}\right)-g^{**}\!\left(\tfrac{y_2-y_1}{x_2-x_1}\right)\!\right)
+(x_4-x_3)\!\left(g^{**}\!\left(\tfrac{y_4-y_q}{x_4-x_3}\right)-g^{**}\!\left(\tfrac{y_4-y_3}{x_4-x_3}\right)\!\right).
%\\
%& \ge \varphi(y_p-y_2)+\varphi(y_3-y_q),
%\end{aligned}
\end{equation}
We define  $\varphi(x):=\min\left\{\frac{x}{2},\,\frac{x^3}{a^2+h^2} \right\}$ for $x\ge 0$. Again
the definition of $g^{**}$ in \eqref{g**} entails
\begin{equation}\label{xg**}u\,\left(g^{**}\left(\tfrac{v_1}{u}\right)-g^{**}\left(\tfrac{v_2}{u}\right)\right)\ge\varphi(v_1-v_2)\,{\bf{1}}_{v_2\ge 0}\end{equation}
 for all $u\in(0,a), v_1\in(0,h), v_2\in(0,h)$ with $v_1\ge v_2$.  
%Therefore,
%\[
%\mathcal{J}(\hat u)-\mathcal{J}(u)\ge \varphi(y_p-y_2)+\varphi(y_3-y_q).
%\]
%Indeed, $g^{**}$ is strictly increasing and strictly positive on $(0,+\infty)$, and 

If $y_2\ge y_1$ and $y_4\ge y_3$, from \eqref{doppia} and \eqref{xg**} we get
\[
\mathcal{J}(\hat u)-\mathcal{J}(u)\ge \varphi(y_p-y_2)+\varphi(y_3-y_q)\ge \varphi(\max\{y_p-y_2, y_3-y_q\})\ge \varphi\left(\tfrac{y_p-y_q}{2}\right),
\]
where we have used 
 $y_3\ge y_2$ which entails $2\max\{y_p-y_2,y_3-y_q\}\ge y_p-y_2+y_3-y_q\ge y_p-y_q$.
 Else we notice that  $y_2<y_1$ or $y_4<y_3$  may happen only if  $\{s\in S: x_p< s< x_q\}=\emptyset$, in which case $y_p-y_2=y_3-y_q=\tfrac{y_p-y_q}{2}$. Moreover, in such case since $y_1\le y_4$ and $y_2=y_3$ it is clear that the two inequalities $y_2<y_1$ and $y_4<y_3$ do not simultaneously hold. Therefore, even in this case from \eqref{doppia} and \eqref{xg**} we get $\mathcal{J}(\hat u)-\mathcal{J}(u)\ge \varphi\left(\tfrac{y_p-y_q}{2}\right)$.
 
 We finally notice that
 $$y_p-y_q\ge \hat\gamma_2(t_2)-\hat\gamma_2(t_1)\ge\gamma_2(t_2)-\gamma_2(t_1)-2\eps,$$
where $\gamma_2(t_2)-\gamma_2(t_1)$ is, by assumption, a prescribed positive value (independent of $\epsilon$).
We conclude that for any small enough $\epsilon$
\[\begin{aligned}
\mathcal{F}(\gamma)\ge \mathcal{F}(\hat\gamma)-\epsilon=\mathcal{G}(\hat u)-\epsilon\ge\mathcal{J}(\hat u)-\epsilon
\ge \mathcal{J}(u)+ \varphi\left(\tfrac{y_p-y_q}{2}\right)-\epsilon
%(x_2-x_1)\left(g^{**}\left(\tfrac{y_2-y_1}{x_2-x_1}\right)-g^{**}\left(\tfrac{y_p-y_1}{x_2-x_1}\right)\!\right)
%+(x_4-x_3)\left(g^{**}\left(\tfrac{y_4-y_3}{x_4-x_3}\right)-g^{**}\left(\tfrac{y_4-y_q}{x_4-x_3}\right)\!\right)-\eps
>\mathcal{J}(u).
\end{aligned}\]
Since we have shown in Lemma \ref{FGJ} that $\inf\mathcal{J}=\inf\mathcal{F}$, the result follows. 
%Letting 
%\[
%\tilde\gamma(t):=\left\{
%\begin{array}{ll}\vspace{8pt}
%\gamma(t)\quad&\mbox{ if $t\in[0,t_1]\cup[t_2,1]$}\\
%\dfrac{t_2-t}{t_2-t_1}\,\gamma(t_1)+\dfrac{t-t_1}{t_2-t_1}\,\gamma(t_2)\quad&\mbox{ if $t\in(t_1,t_2)$}
%\end{array}
%\right.,
%\]
%Jensen inequality implies $\mathcal{J}$
%
%Let $\epsilon>0$ and let $\hat\gamma\in\mathcal{A}_{a,h,L}$ be a piecewise affine approximation of $\gamma$ with $\hat\gamma_1'(t)>0$ a.e. in $(0,1)$, 
%such that $|\mathcal{F}(\gamma)-\mathcal{F}(\hat\gamma)|<\epsilon$. 
%The uniform convergence ensures that $\hat\gamma_2(t_2)<\hat\gamma_2(t_1)+\epsilon$. Notice that $\hat\gamma$ coincides on $[0,a]$ with the graph of a piecewise affine function in $\mathcal{B}_{a,h,L}$. We may provide a nondecreasing piecewise affine function with lower value of $\mathcal{J}$ as in Lemma,
%
%We let $x_1=\gamma_1(t_1)$ and $x_2=\gamma_1(t_2)$, and we let  $ u(x)=\hat\gamma_2\circ\hat\gamma_1^{-1}(x)$.
%We have by Jensen inequality
%\[
%\int_{x_1}^{x_2}g^{**}(\dot u(x))\,dx\ge (x_2-x_1)g^{**}\left(\tfrac{u(x_2)-u(x_1)}{x_2-x_1}\right)=0.
%\]
%Let the grid be refined enough, then in each of the intervals contained in $[t_1,t_2]$ Jensen inequality provides $g^{**}()$,
%therefore $$\mathcal{J}(u)\ge \mathcal{J}(\bar u)+g^{**}(...\ge \inf\mathcal{J}+g^{**}(...)=\inf\mathcal{G}+g^{**}(...)$$
\end{proof}

Before stating the next lemma, as further notation we introduce  the class
\begin{equation}\label{plusclass}\begin{aligned}
{\mathcal{A}}^+_{a,h,L}&:=\left\{\gamma\in\mathcal{A}_{a,h,L}: \gamma_2'(t)\ge 0 \mbox{ for a.e. $t\in(0,1)$}
% \mbox{ for a.e. $t\in(0,1)$},  \gamma_1(t)=0  \mbox{ for $t\in[0,t_1]$}, \right.\\ &\qquad\qquad\left.\gamma_1(t)=a  \mbox{ for  $t\in[t_2,1]$}, \gamma_1'(t)>0  \mbox{ for a.e. $t\in(t_1,t_2)$}, \mbox{for some $0\le t_1<t_2\le 1$}   
     \right\}.
\end{aligned}\end{equation}

\begin{lemma} \label{nojumps}
Suppose that   $2L\notin (a^2,2ah-a^2)$. \BBB 
 %at least one of the conditions devono essere quelle del Lemma \ref{concaveconvex}\EEE $$h> a\qquad\mbox{and}\qquad a^{2}\le 2L\le a(2h-a)$$ does not hold.  
 Let $\gamma\in\mathcal{A}_{a,h,L}^+$.
If $0<t_1<t_2<1$ exist such that $\gamma_1'(t)=0$ on $(t_1,t_2)$, then $\mathcal{F}(\gamma)>\inf\mathcal{F}$.
%Indeed,  let $\hat\gamma:[0,t_]\to[0,a]\times[0,h]$ (resp. 
  %$\bar\gamma:[t_2,1]\to[0,a]\times[0,h]$) denote a piecewise affine approximation of $\gamma{\big|}_{[0,t_1]}$ with nodes on $\gamma([0,t_1])$ (resp. 
%of $\gamma{\big|}_{[t_2,1]}$ with nodes on $\gamma([t_2,1])$), such that $\hat\gamma_1'(t)>0$ for a.e. $t\in(0,t_1)$ (resp. such that $\bar\gamma_1'(t)>0$ for a.e. $t\in(0,1)$).
%We chose these approximations such that $|\mathcal{F}(\gamma)-\mathcal{F}(\tilde\gamma)|<\eps$
 \end{lemma}
 \begin{proof}
 For $\gamma_2(t_1)<s<\gamma_2(t_2)$, we define $t_s$ as the unique number in $(t_1,t_2)$ such that $\gamma_2(t_s)=\gamma_2(t_1)+s$, $h_s:=\gamma_2(t_2)-\gamma_2(t_s)$ and
\[
\gamma^s(t):=\left\{
\begin{array}{ll}
(0,ts(t_s-t_1)^{-1})\quad &\mbox{ if $t\in[0,t_s-t_1]$}\\
\gamma(t-t_s+t_1)+(0,s)\quad &\mbox{ if $t\in(t_s-t_1, t_s]$}\\
\gamma(t+t_2-t_s)-(0,h_s)\quad &\mbox{ if $t\in(t_s, t_s+1-t_2)$}\\
(a, h-h_s)+(t_2-t_s)^{-1}(t-1+t_2+t_s)(0,h_s)\quad &\mbox{ if $t\in[t_s+1-t_2,1]$}.
\end{array}
\right.
\]
It is clear that for any $s$, $\mathcal{F}(\gamma^s)=\mathcal{F}(\gamma)$, since $\gamma^s$ is just obtained from $\gamma$ by rearrangement of pieces (by translations).
It is also clear that $s$ can be (uniquely)  chosen such that $\gamma_s\in\mathcal{A}_{a,h,L}^+$. Let $ r$ denote such value of $s$ and let $\breve\gamma:=\gamma^{r}$. We next define suitable approximations by means of Lemma \ref{infcurve}. We let
\[
\breve\gamma^N(t):=\left\{
\begin{array}{ll}
(0,tr(t_r-t_1)^{-1})\quad &\mbox{ if $t\in[0,t_r-t_1]$}\\
\gamma^N(t)\quad &\mbox{ if $t\in(t_r-t_1, t_r+1-t_2)$}\\
%\gamma(t+t_2-t_r)-(0,h_r)\quad &\mbox{ if $t\in(t_r, t_r+1-t_2]$}\\
(a, h-h_r)+(t_2-t_r)^{-1}(t-1+t_2+t_r)(0,h_r)\quad &\mbox{ if $t\in[t_r+1-t_2,1]$},
\end{array}
\right.
\]
where $$\gamma^N:[t_{ r}-t_1,t_{r}+1-t_2]\to[0,a]\times[ r, h-h_{ r}]$$ are piecewise affine approximations  of $\breve\gamma{\big|}_{[t_r-t_1,t_r+1-t_2]}$,
with same initial point $\breve\gamma(t_r-t_1)=(0,r)$, same end point $\breve\gamma(t_r+1-t_2)=(a,h-h_r)$, with $(\gamma_2^N)'(t)\ge 0$ and $(\gamma_1^N)'(t)>0$ for a.e. $t\in(t_r-t_1,t_r+1-t_2)$. These approximating curves are constructed by means of Lemma \ref{infcurve}, so that $\breve\gamma^N\in\mathcal{A}_{a,h,L}$, $\breve\gamma_N\to\breve\gamma$ uniformly on $[t_{ r}-t_1,t_{r}+1-t_2]$ and $\mathcal{F}(\breve\gamma^N)\to\mathcal{F}(\breve\gamma)$ as $N\to+\infty$. %thanks to the usual Lipschitz estimate of \eqref{gammaenne}.
Since $\gamma_1^N$ is strictly increasing we may define the piecewise affine function
 $u_N:=\gamma_2^N\circ(\gamma_1^N)^{-1}$ on $[0,a]$, that we extend to $\mathbb{R}$ by setting $u_N(x)=0$ if $x<0$ and $u_N(x)=h$ if $x>a$. By changing variables  as done in the proof of Lemma \ref{goodparametrization}  we get 
 \begin{equation}\label{rh}\begin{aligned}
 \overline{\mathcal{J}}_+(u_N)&=r+h-h_r+\int_0^a g^{**}(\dot u_N)\le r+h-h_r+\int_0^a \frac{\dot u_N(x)^3}{1+\dot u_N(x)^2}\,dx\\
 &=r+h-h_r+\int_{t_r-t_1}^{t_r+1-t_2}\frac{(\gamma_2^N)'(t)^3}{(\gamma_1^N)'(t)^2+(\gamma_2^N)'(t)^2}\,dt=\mathcal{F}(\breve\gamma^N) 
 \end{aligned}\end{equation}
 and
 \begin{equation*} \int_0^a u_N=\int_{t_r-t_1}^{t_r+1-t_2}\gamma_2^N(\gamma_1^N)'=\int_0^1\breve\gamma_2^N(\breve\gamma_1^N)',\qquad
%A(u_N)=\int_{t_r-t_1}^{t_2-t_r} (\gamma_1^N)'\gamma_2^N,\quad
%\mathcal{G}(u_N)=\mathcal{F}(\gamma^N),\quad
\int_0^a|u_N'(x)|\,dx=\int_{t_r-t_1}^{t_r+1-t_2} |(\gamma_2^N)'(t)|\,dt\le h.
\end{equation*}
 Thanks to the latter estimate, $u_N$ admits a $w^*-BV_{loc}(\mathbb{R})$ limit $u$, which satisfies $u(0^+)\ge r$ and $u(a^-)\le h-h_r$. Up to extraction of a not relabeled subsequence, the convergence also holds strongly in $L^1(0,a)$, thus  $\int_0^a u=\lim_{N\to+\infty} \int_0^a u_n=L$ so that $u\in\mathcal{C}^+_{a,h,L}$. The lower semicontinuity of $\overline{\mathcal{J}}_+$  and \eqref{rh} entail
\[
\overline{\mathcal{J}}_+(u)\le\liminf_{N\to+\infty}\overline{\mathcal{J}}_+(u^N)\le\liminf_{N\to+\infty}\mathcal{F}(\breve\gamma^N)=\mathcal{F}(\breve\gamma)=\mathcal{F}(\gamma).
\] 
But $u(0^+)>0$ and $u(a^-)<h$, thus $u$ is not a minimizer of $\overline{\mathcal{J}}_+$ over $\mathcal{C}^+_{a,h,L}$ due to Lemma \ref{concaveconvex}. We conclude that $\mathcal{F}(\gamma)>\min_{\mathcal{C}^+_{a,h,L}}\overline{\mathcal{J}}_+$. By Lemma \ref{FGJ}, the result follows.
\end{proof}

\begin{lemma}\label{younglemma}
Let $\gamma\in\mathcal A_{a,h,L}^+$. Then $\mathcal{F}(\gamma)\ge h-a/2$ and equality holds if and only if $\gamma_1'(t)=\gamma_2'(t)$ for a.e. $t\in\{\gamma_1'(t)>0\}$.
\end{lemma}
\begin{proof}
Let $\gamma\in\mathcal A_{a,h,L}^+$.
%Let $S:=\{\gamma_1'(t)=0\}$ and $S^c:=(0,1)\setminus S$.  
We have
\begin{equation*}\label{young}
\begin{aligned}
\mathcal{F}(\gamma)&=\int_0^1\left(\gamma_2'(t)+\frac{\gamma_2'(t)^3}{\gamma_1'(t)^2+\gamma_2'(t)^2}-\gamma_2'(t)\right)\,dt\\
%\int_{S^c}\frac{\gamma_2'(t)^3}{\gamma_1'(t)^2+\gamma_2'(t)^2}\,dt=\int_0^1\gamma_2'(t)\,dt+\int_{S^c}
%\left(\frac{\gamma_2'(t)^3}{\gamma_1'(t)^2+\gamma_2'(t)^2}-\gamma_2'(t)\right)\,dt\\
&=h-\int_0^1\frac{\gamma_1'(t)\gamma_2'(t)}{\gamma_1'(t)^2+\gamma_2'(t)^2}\,\gamma_1'(t)\,dt\ge h-\frac12\int_0^1\gamma_1'(t)\,dt=h-\frac a 2.
\end{aligned}
\end{equation*}
Equality holds if and only if $\gamma_1'=\gamma_2'$ a.e. on $\{\gamma_1'(t)>0\}$, since the Young inequality $2\alpha\beta\le \alpha^2+\beta^2$ is an equality if and only if $\alpha=\beta$.
\end{proof}

%\begin{lemma}\label{younglemma}
%Let $\gamma\in\mathcal A_{a,h,L}^+$. Then $\mathcal{F}(\gamma)\ge h-a/2$ and equality holds if and only if $\gamma_1'(t)=\gamma_2'(t)$ for a.e. $t\in\{\gamma_1'(t)>0\}$.
%\end{lemma}
%\begin{proof}
%Let $\gamma\in\mathcal A_{a,h,L}^+$.
%Let $S:=\{\gamma_1'(t)=0\}$ and $S^c:=(0,1)\setminus S$.  We have
%\begin{equation*}\label{young}
%\begin{aligned}
%\mathcal{F}(\gamma)&=\int_S\gamma_2'(t)\,dt+\int_{S^c}\frac{\gamma_2'(t)^3}{\gamma_1'(t)^2+\gamma_2'(t)^2}\,dt=\int_0^1\gamma_2'(t)\,dt+\int_{S^c}
%\left(\frac{\gamma_2'(t)^3}{\gamma_1'(t)^2+\gamma_2'(t)^2}-\gamma_2'(t)\right)\,dt\\
%&=h-\int_0^1\frac{\gamma_1'(t)\gamma_2'(t)}{\gamma_1'(t)^2+\gamma_2'(t)^2}\,\gamma_1'(t)\,dt\ge h-\frac12\int_0^1\gamma_1'(t)\,dt=h-\frac a 2.
%\end{aligned}
%\end{equation*}
%Equality holds if and only if $\gamma_1'=\gamma_2'$ a.e. on $S^c$, since the Young inquality $2\alpha\beta\le \alpha^2+\beta^2$ is an equality if and only if $\alpha=\beta$.
%\end{proof}

We are ready for the proof of the main results. %and the associated constrained problem% for getting existence and uniqueness, thus we introduce the problem
%\begin{equation}\label{problem2}
%\min\{\mathcal{F}(\gamma):\gamma\in{\mathcal{A}}^+_{a,h,L}\}.
%\end{equation}

\begin{proofad1}  

Let us start by proving existence. %Take $u_{*}$ from Lemma \ref{moreproperties} such that   $\overline{\mathcal J}_{+}$ over ${\mathcal{C}^+_{a,h,L}}$ such that $u_{*}\in W^{1,\infty}(0,a)$ and $0\le {\dot u}_{*}\le 1$ a.e. in $(0,a)$.
Take $u_*\in \mathcal{C}^+_{a,h,L}$ from Lemma \ref{moreproperties}, such that $u_*\in\argmin_{\mathcal{C}^+_{a,h,L}}\overline{\mathcal{J}}_+$, $u_{*}\in W^{1,\infty}(0,a)$ and $0\le \dot u_*\le 1$ a.e. in $(0,a)$.
We have seen in the proof of Lemma \ref{FGJ} that there exists $\gamma_*\in\mathcal{A}_{a,h,L}$ such that 
$\overline{\mathcal{J}}_+(u_*)=\mathcal{F}(\gamma_*)= \inf\mathcal{F}$. This concludes the proof.
We also stress that from \eqref{gammastar} we deduce $\gamma_*\in\mathcal{A}^+_{a,h,L}$, which is the class defined in \eqref{plusclass}. In fact, any solution to problem \eqref{problem1} belongs to $\mathcal{A}^+_{a,h,L}$ by Lemma \eqref{+class}.
%By taking Lemma \ref{FGJ}, Lemma \ref{goodparametrization} and Lemma \ref{infcurve} into account, we \BBB
%%\begin{proof}
%let $u_{*}\in \argmin \overline{\mathcal J}_{+}$ such that $\overline{\mathcal J}_{+}(u_{*})=\mathcal J_{+}(u_{*})=\min \mathcal J_{+}=\inf \mathcal F$. We choose $0< \alpha< \beta < 1$ and define $\gamma_{*}: [0,1]\to [0,a]\times [0,h]$ 
%\begin{equation}
%\gamma_{*}(t):=\left\{\begin{array}{ll} & (0, t\alpha^{-1}u_{*}(0+))\ \hbox{if}\ 0\le t\le \alpha\\
%&\\
%& (a(\beta-\alpha)^{-1}(t-\alpha), u(a(\beta-\alpha)^{-1}(t-\alpha)) \ \hbox{if}\ \alpha\le t\le \beta\\
%&\\
%& (a, u_{*}(a-)+ (h-u_{*}(a-))(1-\beta)^{-1}(t-\beta) \ \hbox{if}\ \beta\le t\le 1\\
%\end{array}\right.
%\end{equation}
%It is readily seen that $\gamma_{*}\in \mathcal{A}_{a,h,L}$ and 
%\begin{equation}
%\mathcal F(\gamma_{*})=\int_{0}^{a}(g({\dot u}_{*})- {\dot u}_{*})\,dx + h.
%\end{equation}
%Therefore by recalling that   $0\le {\dot u}_{*}\le 1$ we get $\mathcal F(\gamma_{*})=\overline{\mathcal J}_{+}(u_{*})=\mathcal J_{+}(u_{*})$ and thesis follows.

Let us prove {\it i)}.
%Suppose first that $h>a$ and $a^2<2L< a(2h-a)$. By considering any minimizer $u$ of $\overline{\mathcal{J}}_+$ over $\mathcal{C}^+_{a,h,L}$ with jumps only at $x=0$ and $x=a$, among the infinitely many provided by Lemma \ref{uniquelemma} (see also Remark \ref{infinity}), a  parametric curve $\gamma\in\mathcal{A}^+_{a,h,L}$ can be constructed as done in the proof of Lemma \ref{FGJ}, so that $\mathcal{F}(\gamma)=\overline{\mathcal{J}}_+(u)$. 
%Indeed, \eqref{47} still holds and since $\dot u\ge 1$ a.e. in $(0,a)$ we get $h+\int_0^a (g^{**}(\dot u)-\dot u)=h+\int_0^a (g^{**}(1)-1)=\min_{\mathcal{C}^+_{a,h,L}}\mathcal{\overline{J}}_+$. 
%Since Lemma \ref{FGJ} yields $\min_{\mathcal{C}^+_{a,h,L}} \overline{\mathcal{J}}_+=\inf\mathcal{F}$, we conclude that $\gamma$ solves problem \eqref{problem1}, so that it also minimizes $\mathcal{F}$ over $\mathcal{A}^+_{a,h,L}$.
%In the other cases, the solution to problem \eqref{problem1} provided by the proof of Theorem \ref{main1} belongs again  to $\mathcal{A}^+_{a,h,L}$.
Suppose  that $2L\notin (a^2, 2ah-a^2)$.
 %either $h>a$ or $a^2<2L< a(2h-a)$ does not hold. 
  Let $\gamma$ be an element of $\mathcal{A}^+_{a,h,L}$ that solves problem \eqref{problem1}, 
  %whose existence is shown in  Theorem \ref{main1},
   so that $\mathcal{F}(\gamma)=\min_{\mathcal{C}^+_{a,h,L}} \overline{\mathcal{J}}_+$. Taking advantage of Lemma \ref{nojumps}, there exist
$0\le t_1<t_2\le 1$ such that $\gamma_1$ is constant on $[0,t_1]$ and $[t_2,1]$, and it is strictly increasing on $[t_1,t_2]$. We let
 $\bar\gamma=(\bar\gamma_1,\bar\gamma_2)$ denote the restriction of $\gamma$ to $[t_1,t_2]$ and we  define a monotonic $BV_{loc}(\mathbb{R})$ function by
%\[
%u(x)=\left\{ 
%\begin{array}{ll}
%0\quad&\mbox{ if $x<0$}\\
$u(x)=\bar\gamma_2\circ\bar\gamma_1^{-1}(x)$ for $x\in (0,a)$ %\quad&\mbox{ if $x\in[0,a]$}\\
%h\quad&\mbox{ if $x>a$}
%\end{array}\right.
%\]
(extended to $\mathbb{R}$ by $u(x)=0$ if $x<0$ and $u(x)=h$ if $x>a$).
%that belongs to $\mathcal{C}^+_{a,h,L}$. 
By invoking Lemma \ref{infcurve} as done in the proof of Lemma \ref{nojumps}, we introduce  piecewise affine approximations $\gamma^N\in\mathcal{A}^+_{a,h,L}$ of $\gamma$, with $(\gamma_1^N)'(t)>0$ for a.e. $t\in(0,1)$, so that $\mathcal{F}(\gamma^N)\to \mathcal{F}(\gamma)$ and $\gamma^N\to\gamma$ uniformly on $[0,1]$ as $N\to+\infty$. As a consequence, letting $u_N:=\gamma_2^N\circ(\gamma_1^N)^{-1}$ in $(0,a)$ (extended to $\mathbb{R}$ by $u_N(x)=0$ if $x<0$ and $u^N(x)=h$ if $x>a$) there also holds $u_N\to u$ pointwise a.e. in $\mathbb{R}$  as $N\to+\infty$.
By changing variables we get 
\[
\int_0^a u_N=\int_0^a \gamma_2^N((\gamma_1^N)^{-1}(x))\,dx
=\int_{0}^{1}\gamma_2^N(t)(\gamma_1^N)'(t)\,dt=\int_{0}^{1}\gamma_2(t)\gamma_1'(t)\,dt=L,
\]
so that $u_N\in\mathcal{C}^+_{a,h,L}$ for any $N$, and (by using $g^{**}\le g$) 
 \begin{equation*}
 \overline{\mathcal{J}}_+(u_N)\le \int_0^a g^{**}(\dot u_N)=\int_{0}^{1}\frac{(\gamma_2^N)'(t)^3}{(\gamma_1^N)'(t)^2+(\gamma_2^N)'(t)^2}\,dt=\mathcal{F}(\gamma^N). 
 \end{equation*}
  A $w^*-BV_{loc}(\mathbb{R})$ limit point of $u_N$ necessarily coincides with $u$ since   $w^*-BV_{loc}(\mathbb{R})$ and pointwise a.e. limit coincide.
 By passing to the limit with the $w^*-BV_{loc}(\mathbb{R})$ lower semicontinuity of $\overline{\mathcal{J}}_+$ we get $\overline{\mathcal{J}}_+(u)\le\mathcal{F}(\gamma).$
But Lemma \ref{FGJ} and Theorem \ref{main1} yield $\mathcal{F}(\gamma)=\inf\mathcal{F}=\min_{\mathcal{C}^+_{a,h,L}}\overline{\mathcal{J}}_+$. We conclude that $u$ coincides with the unique minimizer $u_*$ of $ \overline{\mathcal{J}}_+$ over $\mathcal{C}^+_{a,h,L}$ provided by Lemma \ref{uniquelemma}. Hence the curve $\gamma$ necessarily coincides with the graph of $u_*$ on $(0,a)$ plus the possible vertical segments at $x=0$ or $x=a$. This concludes the proof of {\it i)}.

Eventually, let us prove the statement {\it ii)}. %is a corollary of the following proof of Theorem \ref{main3}. 
Suppose  that $h>a$ and $a^2<2L< a(2h-a)$. All the piecewise affine curves $\gamma$ in $\mathcal A^+_{a,h,L}$ that are constructed in Section \ref{mainresults} after the statement of Theorem \ref{main3} satisfy $\mathcal F(\gamma)=h-a/2$ as seen in \eqref{piecewiseenergy}. Therefore, they solve problem \ref{problem1} thanks to Lemma \ref{younglemma}.
\end{proofad1}

Let us now give a precise characterization of solutions in the nonuniqueness range, by proving Theorem \ref{main3}. 
%We give the proof without using piecewise affine approximations.

\begin{proofad3}

Suppose  that $h>a$ and $a^2<2L< a(2h-a)$. By Lemma \ref{+class} any solution to problem \eqref{problem1} belongs to $\mathcal A^+_{a,h,L}$.  As $\gamma^\circ\in\mathcal A_{a,h,L}^+$ and $\mathcal F(\gamma^\circ)=h-a/2$, we conclude that  $\gamma^\circ$ solves problem \eqref{problem1} as a consequence of Lemma \ref{younglemma}. More generally, still by Lemma \ref{younglemma}, $\gamma$ is solution to problem \eqref{problem1} if and only if $\gamma\in\mathcal A^+_{a,h,L}$ and $\gamma_1'=\gamma_2'$ a.e. on $\{\gamma_1'>0\}$.  It is clear that
$\gamma^\circ$ is the unique curve in the latter class such that the set $\{\gamma_1'(t)>0\}$ is an interval $(t_1,t_2)$ for some $0<t_1<t_2<1$. 
\end{proofad3}

\begin{remark}\rm Suppose  that $h>a$ and $a^2<2L< a(2h-a)$.
$\gamma^\circ$ corresponds indeed to the unique minimizer of $\overline{\mathcal{J}}_+$ among functions $u$ in $\mathcal{C}^+_{a,h,L}$ such that $u\in W^{1,\infty}(0,a)$ with $\dot u=1$ on $(0,a)$, see %with jumps only at $x=0$ and $x=a$,
   Lemma \ref{uniquelemma} and Remark \ref{infinity}.
\end{remark}

The proof of Theorem \ref{mainbis} relies on a careful application of the Euler-Lagrange equation and it requires some preliminary lemmas.
We shall provide a parametrization in terms of $\dot u$ as originally done by Euler in the solution of Proposition 53 in Scientia Navalis \cite{E}. 
 Without loss of generality, as we have pointed out in Lemma 4.11 and in Remark 4.13, we may consider only the case of convex solutions. 
 %\MMM For a convex function $u$ on $(0,a)$ with $0\le \dot u\le 1 $ a.e., in the next three lemmas the notation $\dot u(0)$ (resp. $\dot u(a)$) is understood as the right derivative at $0$ (resp. the left derivative at $a$). 
 We start by proving the following\EEE
\begin{lemma} \label{explicit1}
Assume that %one of the following conditions 
$2L\le (ah)\wedge a^2$ 
%\begin{equation} \label{h<a}\frac{2L}{a}\le h\le a; \ \ \frac{2L}{a}\le a < h\end{equation}
holds true and let $\Psi,\ \Phi, \ \mathcal T$ as in \eqref{L}, \eqref{phi} and \eqref{T} respectively.
Then
\begin{equation*}%\label{minJ=minL}
\min_{\mathcal{C}^+_{a,h,L}}\overline{\mathcal J}_{+}=\min _{\mathcal T}\Psi.
\end{equation*}
\end{lemma}
\begin{proof} Since $2L\le (ah)\wedge a^2$,  then by Lemma
\ref{uniquelemma}  there exists a unique $u_{*}\in \argmin_{\mathcal C^+_{a,h,L}} \overline{\mathcal J}_{+}$, $u_{*}\in W^{1,\infty}(0,a)$  and $0\le \dot u_{*}\le 1$. By Corollary \ref{convconc} $u_{*}$ is convex in $(0,a)$, $u_{*}(0)=0$ and finally by Lemma \ref{minJbar} there exist $\overline \lambda, \ \overline\mu\in \mathbb R$ such that
\begin{equation}\label{eulereq2}
g'(\dot u_{*} )=\overline\lambda x+\overline\mu\ \ \hbox{\ a.e. in }\ (0,a)
\end{equation}
hence  $\overline\mu= g'(\dot u_{*}(0) )$ and $\overline\lambda= a^{-1}(g'(\dot u_{*}(a) )-g'(\dot u_{*}(0) ))\ge 0$ since $g'$ is increasing in $[0,1]$ and $\dot u_{*}(a) \ge \dot u_{*}(0)$ by convexity of $u$.
Moreover, due to continuity of the right hand side, \eqref{eulereq2} holds  everywhere in $[0,a]$ and $\dot u_{*}$ is continuous therein;  therefore the set 
\begin{equation*}
\mathcal{K}:=\left\{u\in C^1([0,a]): u(0)=0,\, \, \int_0^au=L,\,\; u \mbox{ convex},\,\; 0\le \dot u\le 1,\,\; g'(\dot u(x) )=\overline \lambda x+\overline \mu\right\}
\end{equation*}
%of $u\in C^{1}([0,a])$ such that\\
%i) $u$ is convex, $u(0)=0,\ 0\le \dot u\le 1$;\\
%ii)$\int_{0}^{a}u(x)\,dx=L;$\\
%iii) $g'(\dot u(x) )=a^{-1}(g'(\dot u_{*}(a) )-g'(\dot u(0) ))x+g'(\dot u(0)),\ 
%\ \forall x\in [0,a]$\\
is nonempty  and
\begin{equation*}%\label{minJ=minL}
\min_{\mathcal C^+_{a,h,L}} \overline{\mathcal J}_{+}= \min_{\mathcal K} \overline{\mathcal J}_{+}.
\end{equation*}

If  $u\in \mathcal K$ and $\dot u(0) < \dot u(a)$ then $\dot u$ is strictly increasing and by setting $t:= \dot u(x)$,  taking into account that  $g'(\dot u(x))=\overline\lambda x+\overline\mu$  and that \BBB $u(0)=0$, it is readily seen that
the curve $\sigma(x):= (x, u(x)),\ x\in [0,a]$, is equivalent to the one  parametrized by 
\begin{equation*}%\label{parametr}
\displaystyle
 x(t):=\dot u^{-1}(t)\BBB=\frac{a(g'(t)-g'(\dot u(0))}{g'(\dot u(a))-g'(\dot u(0))},\,\ \ y(t):=\int_{\dot u(0)}^{t}sx'(s)\,ds,\ \ t\in [\dot u(0), \dot u(a)]
\end{equation*}
and a direct computation using Lemma \ref{representationlemma} shows that 
\begin{equation}\label{J=L}
\begin{aligned}
\overline{\mathcal J}_{+}(u)&=h+\int_0^a (g^{**}(\dot u(x))-\dot u(x))\,dx=h-\int_0^a\frac{\dot u(x)}{1+\dot u(x)^2}\,dx
=h-\int_{\dot u(0)}^{\dot u(a)}\frac{t}{1+t^2}\,x'(t)\,dt\\
&=h-\frac{a\dot u(a)}{1+\dot u(a)^2}+\int_{\dot u(0)}^{\dot u(a)}\frac{1-t^2}{(1+t^2)^2}\frac{g'(t)-g'(\dot u(0))}{g'(\dot u(a))-g'(\dot u(0))}\,dt
=\Psi(\dot u(0), \dot u(a)).
\end{aligned}\end{equation} 
Moreover, by using again the change of variable $t:= \dot u(x)$, taking into account that $u(0)=0$, $x(\dot u(0))=0$, $x(\dot u(a))=a$,  the area constraint becomes
\begin{equation}\label{ellephi}
\begin{aligned}
L&=\int_0^a u(x)\,dx=\int_0^a (a-u(x))\dot u(x)\,dx=\int_{\dot u(0)}^{\dot u(a)}(a-x(t))\,tx'(t)\,dt\\&=\frac{a^2}2\,\dot u(0)+\frac12\int_{\dot u(0)}^{\dot u(a)}(a-x(t))^2\,dt=\frac{a^2}2\,\dot u(0)+\frac{a^2}2\int_{\dot u(0)}^{\dot u(a)}\left(\frac{g'(\dot u(a))-g'(t)}{g'(\dot u(a))-g'(\dot u(0))}\right)^2dt\\&=\Phi(\dot u(0),\dot u(a)).
\end{aligned}
\end{equation}\BBB
%\begin{equation*}%\label{area}
%\displaystyle L=\int_{0}^{a}u(x)\,dx=\int_{0}^{a}(a-x)\dot u(x)\,dx=\int_{\dot u(0)}^{\dot u(a)}(a-x(t))tx'(t)\,dt=\Phi(\dot u(0), \dot u(a)).
%\end{equation*}
On the other hand if $\dot u(0) = \dot u(a)$ then $\dot u(x)\equiv \dot u(0)$ and since $\int_{0}^{a}u=L$ we get $\dot u(x)\equiv 2La^{-2}$ and  \begin{equation}\label{lambdazero}
\overline{\mathcal J}_{+}(u)=\Psi(2La^{-2}, 2La^{-2}),\qquad \Phi(2La^{-2}, 2La^{-2})=L.
\end{equation}
That is, \eqref{J=L} holds true for every $u\in \mathcal K$.
Hence 
\begin{equation*}
\min_{\mathcal K} \overline{\mathcal J}_{+}(u)=\min \{\Psi(\dot u(0), \dot u(a)): u\in \mathcal K\}
\end{equation*}
and by noticing that $\{(\dot u(0), \dot u(a)): u\in \mathcal K\}\subset \mathcal T$ we get
\begin{equation*}
\min_{\mathcal K} \overline{\mathcal J}_{+}(u)=\min \{\Psi(\dot u(0), \dot u(a)): u\in \mathcal K\}\ge \min_{\mathcal T}\Psi.
\end{equation*}

We claim that if $(\xi_{*},\eta_{*})\in \argmin_{\mathcal T} \Psi$ then there exists $u_{*}\in \mathcal K$ such that $(\dot u_{*}(0), \dot u_{*}(a))=(\xi_{*},\eta_{*})$: indeed if $\xi_{*}=\eta_{*}$ then by \eqref{phi} we get $\Phi(\xi_{*},\xi_{*})=a^{2}\xi_{*}/2=L$, that is $\xi_{*}=\eta_{*}= 2L/a^{2}$ and therefore it is enough to choose $u_{*}(x)=\xi_{*} x$; otherwise we have $\xi_{*} < \eta_{*}$ and we may define a parametrized curve by
\begin{equation}\label{parametr}\displaystyle
x_{*}(t):=\frac{a(g'(t)-g'(\xi_{*}))}{g'(\eta_{*})-g'(\xi_{*})};\ \ y_{*}(t):=\int_{\xi_{*}}^{t}sx'(s)\,ds,\ \ t\in [\xi_{*}, \eta_{*}].
\end{equation}
It is readily seen that $x_{*}(t)$ is strictly increasing from $[\xi_{*}, \eta_{*}]$ onto $[0,a]$ and by denoting with $\varphi$ its inverse we define $u_{*}(x):= y_{*}(\varphi(x))$. A direct computation shows that $u_{*}$ is differentiable in $(0,a)$  and $\dot u_{*}(x)=\varphi(x)$ therein, so it is easy to see that $u_{*}\in \mathcal K$ and $\dot u_{*}(0)=\varphi(0)=\xi_{*},\ \dot u_{*}(a)=\varphi(a)=\eta_{*}$ thus proving the claim. Therefore
$$\min \{\Psi(\dot u(0), \dot u(a)): u\in \mathcal K\}\le \min_{\mathcal T}\Psi$$
and the proof is achieved.
\end{proof}
\begin{lemma}\label{xi<eta} Assume that $ 2L\le (ah)\wedge a^2$. Then there exists a unique  $(\xi_{*}, \eta_{*})\in \argmin_{\mathcal T}\Psi$. Moreover, if $2L<(ah)\wedge a^2$, then $\xi_{*}< \eta_{*}$.
\end{lemma}
\begin{proof} Let $(\xi_*,\eta_*)$ be a minimizer of $\Psi$ over $\mathcal T$. 
%By Lemma \ref{uniquelemma}, $\overline{\mathcal{J}}_+$ admits a unique minimizer $u_{*}$ over $\mathcal{C}^+_{a,h,L}$ and by the proof of  Lemma \ref{explicit1} we get 
%\begin{equation}\label{duepunti} \mathcal L(\dot u_{*}(0), \dot u_{*}(a))= \min_{\mathcal T}\mathcal L=\mathcal L(\xi_*,\eta_*),\end{equation}
%hence uniqueness follows. In particular, we have $\xi_*=\dot u_*(0)$ and $\eta_*=\dot u_*(a)$.
Following the proof of Lemma \ref{explicit1},  there exists $u\in\mathcal K$ such that $(\xi_*,\eta_*)=(\dot u(0),\dot u(a))$ and moreover $$\min_{\mathcal C^+_{a,h,L}}\overline{\mathcal J}_+=\min_{\mathcal T}\Psi=\Psi(\xi_*,\eta_*)=\Psi(\dot u(0),\dot u(a))=\overline{\mathcal J}_+(u)  $$
But  Lemma \ref{uniquelemma} shows that there exists a unique minimizer $u_*$ of $\overline{\mathcal{J}}$ over $\mathcal C_{a,h,L}^+$ (and $u_*\in\mathcal K$ as seen in the proof of Lemma \ref{explicit1}). Therefore  $u$ necessarily coincides with $u_*$. Thus  $(\xi_*,\eta_*)=(\dot u_*(0),\dot u_*(a))$ and this proves uniqueness.
%If $L=(ah)\wedge a^2$, from Corollary \ref{convconc} we have $\dot u(x)\equiv 2La^{-2}$, hence we directly deduce $\xi_*=\eta_*=2La^{-2}$ 

Assume now by contradiction that  $2L<(ah)\wedge a^2$ and $\xi_{*}= \eta_{*}$. Then by \eqref{phi} $\xi_{*}= \eta_{*}=2L/a^{2}$.  If we consider a couple $(\xi,\eta)$ that satisfies \begin{equation}\label{3ipotesi}\eta\in [2L/a^{2},1], \quad \xi\in [0,2L/a^{2}],\quad  
(a^{2}/2-L)\xi+L\eta=L,\end{equation}
then it is readily seen that $(\xi,\eta)\in\mathcal T$:
indeed by setting
\[
u_{\xi,\eta}(x):=\left\{\begin{array}{ll}  \xi x\ \ &\hbox{if} \ \ x\in [0,a-\sqrt {2L}]\\
\eta(x-a+\sqrt {2L})+\xi(a-\sqrt {2L})\ \ &\hbox {if} \ \ x\in [a-\sqrt {2L}, a],
\end{array}
\right.
\]
and by reasoning as done in \eqref{ellephi} we get
$$\Phi (\xi,\eta)=\int_{0}^{a}u_{\xi,\eta}(x)\,dx= (a^{2}/2-L)\xi+L\eta=L.
$$
At the same time, by computing as in \eqref{J=L},
$$\Psi(\xi,\eta)=h-\int_0^a\frac{\dot u_{\xi,\eta}(x)}{1+\dot u_{\xi,\eta}(x)^2}\,dx= h-(a-\sqrt{2L})\frac{\xi}{1+\xi^{2}}-\frac{\eta\sqrt{2L}}{1+\eta^{2}}.$$
If we set 
$$\phi(\eta):=\Psi\left (\frac{2L(1-\eta)}{a^{2}-2L},\eta\right ),\qquad \eta\in[2L/a^2,1],$$
we have 
$\phi(2L/a^{2})=\Psi(\xi_{*},\eta_{*})$
and by taking into account that $2L< a^{2}$ a direct computation shows that 
$$\phi'(2L/a^{2})=\left (\frac{2L(a-\sqrt{2L})}{a^{2}-2L}-\sqrt{2L}\right)\frac{1-4L^{2}/a^{4}}{(1+4L^{2}/a^{4})^{2}}< 0.$$
Hence there exist $1\ge \overline \eta > 2L/a^{2}$ and $0\le \overline \xi= \frac{2L(1-\overline \eta)}{a^{2}-2L}< 2L/a^{2}$ such that the couple $(\overline\xi,\overline\eta)$ satisfies the constraint \eqref{3ipotesi} and
$$\phi(\overline \eta)=\Psi(\overline\xi,\overline \eta)< \Psi(\xi^{*},\eta^{*}),$$
thus contradicting  minimality of $(\xi_{*}, \eta_{*})$.
\end{proof}\EEE
The previous results suggests the following parametric representation of the minimizer. 
\begin{lemma}\label{ultimolemma} Assume that $0< 2L\le (ah)\wedge a^2$.
Let $u_{*}$ be the unique minimizer of $\overline{\mathcal J}_{+}$ over $\mathcal{C}^+_{a,h,L}$ provided by {\rm Lemma \ref{uniquelemma}}. Then either $2L=(ah)\wedge a^2$ and $u_{*}(x)= 2La^{-2}x$ for any $x\in (0,a)$, or the curve $\sigma(x):= (x, u_{*}(x)),\ x\in [0,a]$, is equivalent to the one  parametrized by \eqref{parametr},
%\begin{equation*}%\label{parametr2}
%\displaystyle
%x_{*}(t)=\frac{a(g'(t)-g'(\xi_{*}))}{g'(\eta_{*})-g'(\xi_{*})};\ \ y_{*}(t):=\int_{\xi_{*}}^{t}sx'(s)\,ds,\ \ t\in [\xi_{*}, \eta_{*}]
%\end{equation*}
where $(\xi_{*}, \eta_{*})$ is the unique minimizer of $\Psi$ on $\mathcal T$,  $\xi_{*}< \eta_{*}$, and $h_{*}:=y_{*}(\eta_{*})=u_{*}(a^-)< h$.
\end{lemma}
\begin{proof} By Corollary \ref{convconc} if $2L=(ah)\wedge a^2$ then $u_{*}(x)= 2La^{-2}x$ for any  $x\in (0,a)$.  Assume now that  $2L< (ah)\wedge a^2$: if $(\xi_{*},\eta_{*})\in \argmin_{\mathcal T} \Psi$  then by Lemma \ref{xi<eta} $\xi_{*}<\eta_{*}$,  the unique minimizer $u_{*}$ can be parametrized as in \eqref{parametr} and in particular 
$(\xi_{*},\eta_{*})$ is the unique minimizer of $\Psi$ on $\mathcal T$.
We have only to prove that $h_{*}<h$. Indeed
\begin{equation*}%\begin{array}{ll}
\begin{aligned}
 L&=\int_{0}^{a}u_{*}(x)\,dx=au_{*}(a^-)-\int_{0}^{a}xu_{*}'(x)\,dx\\
 %\\
%&\\
%&\displaystyle=
&=ah_{*}-\int_{\xi_{*}}^{\eta_{*}}tx_{*}(t)x_{*}'(t)\,dt=ah_{*}-\eta_{*}\frac{a^{2}}{2}+\frac{1}{2}\int_{\xi_{*}}^{\eta_{*}}x_{*}(t)^{2}\,dt
%\end{array}
\end{aligned}
\end{equation*}
and 
\begin{equation*}
h_{*}=\int_{0}^{a}u_{*}'(x)\,dx=\int_{\xi_{*}}^{\eta_{*}}tx_{*}'(t)\,dt=a\eta_{*}-\int_{\xi_{*}}^{\eta_{*}}x_{*}(t)\,dt.
\end{equation*}
By gathering together the two last relations we get
\begin{equation*}
h_{*}=\frac{2L}{a}+\int_{\xi_{*}}^{\eta_{*}}\left(x_{*}(t)-\frac{x_{*}(t)^{2}}{a}\right)\,dt\le \frac{2L}{a}< h,
\end{equation*}
thus concluding the proof.
\end{proof}

\begin{proofad2} 

 By Lemma \ref{FGJ} and Lemma \ref{explicit1} 
 %and Theorem\ref{main2}
  we get
 \begin{equation}\label{bellissima}
\inf \mathcal F=\min_{\mathcal C^+_{a,h,L}} \overline{\mathcal J}_{+}=\min_{\mathcal T}\Psi.
\end{equation}
Let $\gamma_{*}\in \mathcal A_{a,h,L}$ as in \eqref{sol2}, $x_{*}, y_{*}, u_{*}$ as in Lemma \ref{ultimolemma}: since $h_{*}:= u_{*}(a^-)$ and $0\le \dot u_{*}\le 1$, a direct computation shows that \begin{equation*}
\mathcal F(\gamma_{*})=\int_{0}^{a} {\frac{\dot u_{*}^{3}}{1+\dot u_{*}^{2}}}\,dx+h-h_{*}=h+\int_{0}^{a}\left ( {\frac{\dot u_{*}^{3}}{1+\dot u_{*}^{2}}}-\dot u_{*}\right)\,dx=
\overline{\mathcal J}_{+}(u_{*})=\min_{{\mathcal C}^{+}_{a,h,L}} \overline{\mathcal J}_{+}
\end{equation*}
and the result follows easily  by taking  \eqref{bellissima} into account\EEE.
\end{proofad2}

\begin{remark}\label{allafine}\rm
 If we change $L$ to $ah-L$, the unique solution is given by $\tilde\gamma_{*}$ from \eqref{sterne}. Indeed, by considering the construction of the solution, this is a consequence of Remark \ref{endingremark}.
\end{remark}

The following simple lemma will be used for proving Theorem \ref{main4}.
\begin{lemma}\label{aid} Let $\mathcal S:=\{(\xi,\eta):0\le\xi\le\eta\le1\}$ and let $\Phi:\mathcal S\to\mathbb R$ be the function defined by \eqref{phi}. Then $\partial_\xi\Phi(\xi,\eta)>0$
for any $(\xi,\eta)\in \mathcal S$ such that $0<\xi<\eta$,  $\partial_\eta\Phi(0,\eta)>0$ for any  $\eta\in(0,1)$, and $q'(\xi)>0$ for any $\xi\in(0,1)$, where $q(\xi):=\Phi(\xi,\xi)$. Moreover, $\Phi(\mathcal S)=[0,a^2/2]$.
\end{lemma}
\begin{proof}
A computation exploiting \eqref{phi} shows that for any $(\xi,\eta)\in \mathcal S$ such that $0<\xi<\eta$ there holds
\[
\partial_\xi\Phi(\xi,\eta)=\frac{a^2\,g''(\eta)}{(g'(\eta)-g'(\xi))^3}\,\int_\xi^\eta (g'(\eta)-g'(t))^2\,dt.
\]
Similarly, for any $\eta\in(0,1)$ there holds
\[
\partial_\eta\Phi(0,\eta)=\frac{a^2\,g''(\eta)}{g'(\eta)}\,\int_0^\eta\left(1-\frac{g'(t)}{g'(\eta)}\right)\,dt.
\]
Positivity follows by considering the explicit expression of $g$ from \eqref{gi}. The statement about $q$ is obvious since $q(\xi)=\tfrac12\,a^2\xi$.
$\Phi$ is continuous on $\mathcal S$ with $\Phi(0,0)=0$ and $\Phi(1,1)=a^2/2$, thus having checked the sign of the derivatives, we conclude that $\Phi(\mathcal S)=[0,a^2/2]$.
\end{proof}

 \begin{proofad4}

% \begin{proof}
Let us first prove the continuity of $(0,\tfrac12((ah)\wedge a^2)]\ni L\mapsto\mathcal{F}_{min}(a,h,L)$. Let $\mathcal S$ as in Lemma \ref{aid}. If $0<2L\le(ah)\wedge a^2$ we take a sequence  $(L_j)\subset(0,\tfrac12((ah)\wedge a^2))\setminus\{L\}$ such that $L_j\to L$ as $j\to\infty$. By taking advantage of Lemma \ref{explicit1} we take a couple $(\xi_j,\eta_j)$ that minimizes $\Psi$ over $\mathcal T_j:=\{(\xi,\eta):0\le\xi\le \eta\le 1,\, \Phi(\xi,\eta)=L_j\}$, so that $\mathcal F_{min}(a,h,L_j)=\Psi(\xi_j,\eta_j)$. We extract a subsequence (not relabeled) such that $\xi_j\to\xi_\infty$, $\eta_j\to\eta_\infty$ as $j\to\infty$. By continuity of $\Phi$ over $\mathcal S$  we have $(\xi_\infty,\eta_\infty)\in\mathcal T:=\{(\xi,\eta):0\le\xi\le \eta\le 1,\, \Phi(\xi,\eta)=L\}$.

We claim that $(\xi_\infty,\eta_\infty)$ is a minimizer of $\Psi$ over  $\mathcal T$. Indeed, let us assume by contradiction that it is not, and by using Lemma \ref{explicit1} let us take a minimizer $(\hat\xi,\hat\eta)$ of $\Psi$ over $\mathcal T$, %thus $\hat\xi<\hat\eta$,
thus $\Psi(\hat\xi,\hat\eta)=\mathcal F_{min}(a,h,L)$ and $\Psi(\xi_\infty,\eta_\infty)-\Psi(\hat\xi,\hat\eta)=:\sigma>0$. For $\eps>0$, let $\hat B_\eps$ denote the $\eps$-neighbour of $(\hat\xi,\hat \eta)$ in $\mathcal S$. Thanks to Lemma \ref{aid}, for any $\eps>0$ there exists $\delta>0$ such that  the image of $\hat B_\eps$ through the continuous function $\Phi$ contains the interval $(L-\delta,L+\delta)\cap[0,a^2/2]$.     By using the continuity of $\Psi$ in $\mathcal S$,  let $\eps>0$ be small enough such that  $|\Psi(\hat\xi,\hat\eta)-\Psi(\xi,\eta)|<\sigma/2$ for any $(\xi,\eta)\in \hat B_\eps$. Therefore,  we can find $j$ large enough and $(\tilde\xi,\tilde\eta)\in \hat B_\eps$  such that $\Phi(\tilde\xi,\tilde\eta)=L_j$ (hence $(\tilde\xi,\tilde \eta)\in\mathcal T_j$) and such that   $|\Psi(\xi_\infty,\eta_\infty)-\Psi(\xi_j,\eta_j)|<\sigma/2$. Summarizing, we have the three relations
\[
\Psi(\xi_\infty,\eta_\infty)-\Psi(\hat\xi,\hat\eta)=\sigma,\quad |\Psi(\hat\xi,\hat\eta)-\Psi(\tilde\xi,\tilde\eta)|<\sigma/2,\quad
|\Psi(\xi_\infty,\eta_\infty)-\Psi(\xi_j,\eta_j)|<\sigma/2,
\]
and such relations imply $\Psi(\xi_j,\eta_j)>\Psi(\tilde\xi,\tilde\eta)$, contradicting the minimality of $(\xi_j,\eta_j)$ for $\Psi$ on $\mathcal T_j$. The claim is proved, and since the minimizer of $\Psi$ over $\mathcal T$ is unique by Lemma \ref{xi<eta}, the whole sequence $(\xi_j,\eta_j)$ converges to $(\xi_\infty,\eta_\infty)$, yielding 
\[
\lim_{j\to\infty}\mathcal F_{min}(a,h,L_j)=\lim_{j\to\infty}\Psi(\xi_j,\eta_j)=\Psi(\xi_\infty,\eta_\infty)=\mathcal F_{min}(a,h,L).
\]
This proves the continuity of the map $L\mapsto\mathcal{F}_{min}(a,h,L)$ on $(0,\tfrac12((ah)\wedge a^2))$ and the left continuity at $\tfrac12((ah)\wedge a^2)$.

Let us also remark that if $h\le a$,  \eqref{lambdazero}  yields
 $\Psi(h/a,h/a)=\tfrac{h^3}{a^2+h^2}$, therefore by the above left continuity we get $\lim_{L\uparrow ah/2}\mathcal F_{min}(a,h,L)=\tfrac{h^3}{a^2+h^2}$. Similarly,  if $h>a$, still by \eqref{lambdazero} we have $\Psi(1,1)=h-a/2$, hence     $\lim_{L\uparrow a^2/2}\mathcal F_{min}(a,h,L)=h-a/2$. 
In case $h>a$, we also have $\mathcal F_{min}(a,h,L)=h-a/2$ for any $L\in[a^2/2, ah/2]$ as a consequence of the characterization of the optimal energy in the nonuniqueness range, see Theorem \ref{main3}. 
%Summing up, we have shown that $(0,ah/2]\ni L\mapsto\mathcal{F}_{min}(a,h,L)$ is continuous, with limit $h$ as $L\downarrow 0$.

We next notice that $\Phi(\xi,\eta)=0$ implies $\xi=\eta=0$ and the elementary estimate $\Phi(\xi,\eta)\ge \tfrac{a^2}2(\xi\vee(\eta-\xi))$ on $\mathcal S$ shows that $\mathcal T$ shrinks to the origin as $L$ goes to $0$. Since $\Psi(0,0)=h$ we obtain by \eqref{bellissima} that $\lim_{L\downarrow 0}\mathcal F_{min}(a,h,L)=h$ by continuity of $\Phi$ and $\Psi$ over $\mathcal S$. 

All in all, we have proven the continuity of the map $L\mapsto\mathcal{F}_{min}(a,h,L)$ in $(0,ah/2)$, the left continuity at $ah/2$ 
%(with value $(h-a/2)\wedge\tfrac{h^3}{a^2+h^2}$) 
and $\lim_{L\downarrow 0}\mathcal F_{min}(a,h,L)=h$.
 The symmetry around $L=ah/2$ follows from Remark \ref{allafine}, and then it implies continuity on $(0,ah)$.

Let us eventually discuss the monotonicity.
Let $h\le a$.  Of course we have $\mathcal{F}_{min}(a,h,L)<h$  (see Lemma \ref{<h} and  Remark \ref{segment}). If $L<ah/2$ is increased to a close value $\overline L$, still with $2\overline L\le  ah $, from the curve that realizes the value $\mathcal F_{min}(a,h,L)$ we take a piecewise affine interpolating curve whose subtended area is $\overline L$. The energy goes down by convexity (slopes are smaller than $1$). Therefore $\mathcal{F}_{min}(a,h,\overline L)<\mathcal{F}_{min}(a,h,L)$, proving the monotonicity. The range is $[\tfrac{h^3}{a^2+h^2},h)$, as we have already obtained the continuity and the limit values at $L=0$ and $L=ah/2$.
About the case $2L\ge ah$, we obtain the desired monotonicity by making use of the symmetry of the optimal energy values  around $L=ah/2$. 
Let now $h>a$: we cross the nonuniqueness regime as $L$ grows from $0$ to $ah$.
The argument is the same, also taking into account that $\mathcal{F}_{min}(a,h,L)=h-a/2$ for any $L\in [a^2/2, ah-a^2/2]$ as seen in the proof of Theorem \ref{main3}.
%The last statement is also a consequence of Lemma \ref{<h}.
%Eventually, for fixed $L$ and large $a$ it is enough to $\mathcal{F}_{min}$ with the energy of the diagonal line segment with slope $h/a$, which is equal to $\tfrac{h^3}{a^2+h^2}$ and is therefore vanishing as $a\to\infty$ with $h=O(1)$.
%\end{proof}
 \end{proofad4}
 \EEE

\subsection*{Acknowledgements}
%E. M.  support by the Austrian Science Fund (FWF) project M1733-N20.
The authors are members of the
GNAMPA group of the Istituto Nazionale di Alta Matematica\newline (INdAM).

\Addresses

\end{document}